\documentclass[11pt]{amsart}

\usepackage[T1]{fontenc}
\usepackage[utf8]{inputenc}
\usepackage{amsmath,amssymb,amsthm,mathtools}
\usepackage{enumitem}
\usepackage{booktabs}
\usepackage{aliascnt}
\usepackage{xcolor}
\usepackage{hyperref}
\usepackage[capitalize,noabbrev]{cleveref}
\usepackage[margin=1.15in]{geometry}
\usepackage{microtype}

\hypersetup{
  colorlinks=true,
  linkcolor=blue!55!black,
  citecolor=blue!55!black,
  urlcolor=blue!55!black,
  breaklinks=true,
  pdftitle={Hilbert--90 quotient maps, torsion defects, and symmetric monodromy},
  pdfauthor={Henry Shin},
  pdfsubject={Hilbert--90 quotient maps and finite-field permutation polynomials},
  pdfkeywords={Hilbert--90 quotient maps, torsion defects, permutation polynomials, trace-zero maps, finite fields, monodromy groups, exceptional covers}
}
\newcommand{\doilink}[1]{\href{https://doi.org/#1}{\nolinkurl{doi:#1}}}
\newcommand{\arxivlink}[1]{\href{https://arxiv.org/abs/#1}{\nolinkurl{arXiv:#1}}}

\newtheorem{theorem}{Theorem}[section]
\newaliascnt{corollary}{theorem}
\newtheorem{corollary}[corollary]{Corollary}
\aliascntresetthe{corollary}
\newaliascnt{proposition}{theorem}
\newtheorem{proposition}[proposition]{Proposition}
\aliascntresetthe{proposition}
\newaliascnt{lemma}{theorem}
\newtheorem{lemma}[lemma]{Lemma}
\aliascntresetthe{lemma}
\newaliascnt{conjecture}{theorem}
\newtheorem{conjecture}[conjecture]{Conjecture}
\aliascntresetthe{conjecture}
\theoremstyle{definition}
\newaliascnt{definition}{theorem}
\newtheorem{definition}[definition]{Definition}
\aliascntresetthe{definition}
\newaliascnt{remark}{theorem}
\newtheorem{remark}[remark]{Remark}
\aliascntresetthe{remark}
\newaliascnt{example}{theorem}
\newtheorem{example}[example]{Example}
\aliascntresetthe{example}
\newaliascnt{question}{theorem}
\newtheorem{question}[question]{Question}
\aliascntresetthe{question}

\crefname{theorem}{Theorem}{Theorems}
\crefname{proposition}{Proposition}{Propositions}
\crefname{lemma}{Lemma}{Lemmas}
\crefname{corollary}{Corollary}{Corollaries}
\crefname{definition}{Definition}{Definitions}
\crefname{remark}{Remark}{Remarks}
\crefname{question}{Question}{Questions}
\crefname{conjecture}{Conjecture}{Conjectures}

\DeclareMathOperator{\Tr}{Tr}
\DeclareMathOperator{\Nm}{Nm}

\DeclareMathOperator{\Frob}{Frob}

\newcommand{\F}{\mathbb F}
\newcommand{\Gm}{\mathbb G_m}
\newcommand{\PP}{\mathbb P}
\newcommand{\Z}{\mathbb Z}

\newcommand{\muGP}{\boldsymbol\mu}
\newcommand{\eps}{\varepsilon}
\newcommand{\Ql}{\mathbb Q_\ell}
\newcommand{\barF}{\overline{\mathbb F}}
\newcommand{\Thetaq}{\Theta_q}

\title[Hilbert--90 quotient maps]{Hilbert--90 quotient maps, torsion defects, and symmetric monodromy}
\author{Henry Shin}

\begin{document}
\begin{abstract}
Let $\tau(z)=-1-z^{-1}$.  We study the reduced rational maps
$h_d:\PP^1\to\PP^1$ obtained by cancelling common factors in
\[
        H_d^{\rm raw}(z)=z^d\frac{\tau(z)^d-1}{z^d-1}.
\]
These maps arise by Hilbert--90 descent from the trace-zero maps
$X^{dq}-X^d$ on $\ker\Tr_{\F_{q^3}/\F_q}$, but the principal object is
the resulting $\tau$-equivariant quotient-map family; nonconstant
separable members are viewed as covers.

We prove that cancellation is exactly a torsion-defect phenomenon.  If $\ell(-)$ denotes scheme-theoretic length and
$\boldsymbol\mu_d=\ker([d]:\Gm\to\Gm)$, then
\[
        \deg h_d=d-\ell\bigl((1+X+Y=0)\cap\boldsymbol\mu_d^2\bigr),
\]
and, in characteristic $p>0$ with $d=p^sd_0$ and $p\nmid d_0$,
\[
        h_d=\Frob_{p^s}\circ h_{d_0},\qquad
        \deg h_d=p^s\deg h_{d_0}.
\]
We classify the tame quotient strata of morphism degree at most one
and exactly two; the maximal-defect stratum yields a characteristic-two
Mersenne trace-zero permutation family.  In characteristic zero we
prove the main monodromy theorem: every non-linear quotient is Morse
and has full symmetric geometric monodromy,
$G_{h_d}=S_{\deg h_d}$; the proof rules out branch-value collisions via
a cyclotomic cross-ratio equation.  In positive characteristic we
isolate the following degeneration mechanisms: Frobenius-sparse Kummer
and Artin--Schreier quotients, a certificate-verified
characteristic-$19$ Klein-four Galois quotient, and the first nonsparse
Frobenius--lacunary tower up to its stated primitivity and
wild-inertia boundary.  A twisted off-diagonal fiber-square trace
formula turns $2$-transitive monodromy into a uniform obstruction to
$\tau$-twisted exceptionality.
\end{abstract}


\maketitle

\section{Introduction}

This paper studies a natural family of rational maps of $\PP^1$ commuting with a fixed order-three automorphism; the nonconstant separable members are the covers to which monodromy is attached.  Let
\[
        \tau(z)=-1-z^{-1}.
\]
For each integer $d\ge1$ consider the raw Hilbert--90 quotient
\begin{equation}\label{eq:intro-Hd}
        H_d^{\rm raw}(z)=z^d\frac{\tau(z)^d-1}{z^d-1}
        =\frac{(-1)^d(z+1)^d-z^d}{z^d-1},
\end{equation}
and let $h_d$ be the reduced rational function obtained after cancelling common numerator and denominator factors.  The identity
\[
        h_d\circ\tau=\tau\circ h_d
\]
places these maps in a rigid cyclic-equivariant setting.  The central problem is algebraic: determine the cancellation divisor of $H_d^{\rm raw}$, the degree of $h_d$, the ramification and branch geometry of the resulting map, the geometric monodromy group in the nonconstant separable cases, and the positive-characteristic degenerations where the characteristic-zero behavior fails.

The finite-field trace-zero maps
\[
        P_d(X)=X^{dq}-X^d
\]
are the arithmetic origin of the family, but not the organizing object of the paper.  If $q=p^k$, put
\[
        \Gamma_q=\ker\Tr_{\F_{q^3}/\F_q},
        \qquad
        \Lambda_q=\{z:z^q=\tau(z)\}.
\]
Multiplicative Hilbert 90 identifies
\[
        \Gamma_q^*/\F_q^*\simeq \Lambda_q,
        \qquad
        x\longmapsto x^{q-1}.
\]
Under this identification, $P_d$ descends to the quotient expression \eqref{eq:intro-Hd}.  Thus the trace-zero permutation problem separates into a fiber condition, a denominator condition $\mu_d\cap\Lambda_q=\varnothing$, and the action of the quotient map $h_d$ on the twisted Frobenius fixed set $\Lambda_q$.  The paper therefore treats trace-zero permutations as arithmetic consequences of the geometry of the $\tau$-equivariant maps $h_d$, with cover language reserved for nonconstant separable members.

The guiding arc is
\[
\begin{gathered}
\text{Hilbert--90 quotient maps}
\longrightarrow
\text{torsion defects}
\longrightarrow
\text{branch geometry}\\
\longrightarrow
\text{monodromy}
\longrightarrow
\text{positive-characteristic degenerations}
\longrightarrow
\text{trace-zero arithmetic}.
\end{gathered}
\]
Each arrow is made explicit.  The quotient degree is a torsion-intersection number; characteristic-zero branch collisions are governed by a cyclotomic equation with no nondegenerate solutions; the resulting non-linear characteristic-zero maps, viewed as covers, have full symmetric monodromy; and positive characteristic is organized by sparse, sporadic, and lacunary degeneration mechanisms together with a precise residual wild-inertia problem.

\subsection*{Main theorem package}

Throughout the paper $G_f$ denotes the geometric monodromy group of the separable cover attached to a nonconstant rational function $f$.  In positive characteristic, purely inseparable Frobenius factors are removed before monodromy is taken; this convention is fixed precisely in \cref{def:quotient-conventions}.

The first theorem is the structural invariant behind the whole paper.  Let
\[
        L_0:\ 1+X+Y=0\subset\Gm^2.
\]
Here and below $\ell(-)$ denotes scheme-theoretic length.  In the tame case, meaning characteristic zero or positive characteristic $p$ with $p\nmid d$, cancellation in \eqref{eq:intro-Hd} occurs exactly at the torsion points of $L_0\cap\boldsymbol\mu_d^2$.  Scheme-theoretic multiplicity gives the correct wild factor.

\begin{theorem}[A. Torsion defect and quotient degree]\label{thm:intro-A-torsion}
Let $k$ be algebraically closed of characteristic $p\ge0$, and let $d\ge1$.  If $p=0$, then
\[
        \deg h_d
        =d-\ell\bigl(L_0\cap\boldsymbol\mu_d^2\bigr).
\]
If $p>0$ and $d=p^sd_0$ with $p\nmid d_0$, then
\[
        H_d^{\rm raw}=(H_{d_0}^{\rm raw})^{p^s},
        \qquad
        h_d=h_{d_0}^{p^s},
\]
where the exponent denotes Frobenius on values, not iteration, and
\[
        \deg h_d
        =d-\ell\bigl(L_0\cap\boldsymbol\mu_d^2\bigr)
        =p^s\deg h_{d_0}.
\]
In the tame case, namely in characteristic zero or when $p>0$ and $p\nmid d$, this is equivalently
\[
        \deg h_d=d-\#\{(x,y)\in\mu_d^2:1+x+y=0\}.
\]
\end{theorem}

The next results describe the largest possible torsion collapses.  They are map-theoretic rigidity statements, with finite-field permutation families appearing as arithmetic shadows.  In particular, the characteristic-two Mersenne family below is not a separate finite-field construction: it is the arithmetic shadow of the maximal torsion-defect stratum.

\begin{theorem}[B. Low-degree rigidity and Mersenne arithmetic]\label{thm:intro-B-low-degree}
Among tame exponents, meaning characteristic zero or positive characteristic $p$ with $p\nmid d$, the reduced quotient has morphism degree at most one precisely for
\[
        d=1,\qquad d=3,\qquad\text{or}\qquad d=p^a-1\quad(a\ge1,\ p>0),
\]
and it has morphism degree two precisely for $d=2$ in characteristic zero or in characteristic $p\ne2,3$, and for the sporadic pair $(p,d)=(11,5)$.  The nonconstant degree-one covers in the degree-at-most-one list are obtained by deleting the constant reduced quotients under \cref{def:quotient-conventions}.  For a general positive-characteristic exponent $d=p^sd_0$ with $p\nmid d_0$, these low-degree assertions apply to the reduced quotient $h_{d_0}$; cover and monodromy assertions apply only when that reduced quotient is nonconstant and separable, while the full morphism degree of $h_d$ is $p^s\deg h_{d_0}$.  In the maximal-defect branch $p=2$, $d=2^a-1$, the trace-zero map $X^{dq}-X^d$ permutes $\Gamma_{2^k}$ if and only if
\[
        \gcd(a,k)=1
        \qquad\text{and}\qquad
        3\nmid a.
\]
\end{theorem}

The displayed $d=1$ entry has morphism degree $0$ in characteristic $3$, since $h_1$ is constant there; it is therefore not a cover.  The displayed $d=3$ entry in the tame degree-one list means $p\ne3$ in positive characteristic.  In characteristic $3$, the exponent $d=3$ has separable part $d_0=1$ under the Frobenius convention, and the full reduced quotient is constant; it is therefore not an additional degree-one cover.

The second sparse branch is $d=p^a+1$.  It usually does not collapse to degree one, but after diagonalizing $\tau$ it becomes a Kummer monomial when $p\ne3$, and after conjugating $\tau$ to a translation it becomes an Artin--Schreier map when $p=3$.  In characteristic two this recovers the Ding--Song--Xiong $Q+1$ trace-zero branch.  Thus the two Frobenius-sparse families are structural degenerations of the quotient-map family, not isolated finite-field coincidences.

The characteristic-zero theorem is the main monodromy result.

\begin{theorem}[C. Characteristic-zero symmetric monodromy]\label{thm:intro-C-charzero}
Let $k$ be algebraically closed of characteristic zero, and let $n=\deg h_d$.  If $n>1$, then
\[
        G_{h_d}=S_n.
\]
Equivalently,
\[
        G_{h_d}=S_d\quad(3\nmid d,\ d\ge2),
        \qquad
        G_{h_d}=S_{d-2}\quad(3\mid d,\ d\ge6).
\]
\end{theorem}

The proof is explicit.  In the coordinate $x=z/(z+1)$ the critical-value collision equation reduces to a cross-ratio equation in $(d-1)$-st roots of unity.  Over $\mathbb C$ this becomes a signed sine identity.  A parity argument and a sine-convexity estimate rule out all nondegenerate solutions.  The quotients are therefore Morse; their finite branch cycles are transpositions; and transitivity gives the full symmetric group.

The fourth theorem is a general cover-theoretic obstruction to finite-field exceptionality in this twisted setting.  For a separable $\tau$-equivariant map $f$, let $V\to U$ be its finite etale restriction and put
\[
        C_f=(V\times_UV)\setminus\Delta.
\]
This is the ordinary off-diagonal fiber square; compactly supported cohomology is used only to count twisted fixed points.  For this theorem, fix an auxiliary prime $\ell\ne p$ for $\ell$-adic cohomology; this prime is unrelated to the length notation $\ell(-)$.

\begin{theorem}[D. Off-diagonal fiber squares and twisted exceptionality]\label{thm:intro-D-collision}
Let $f$ be a separable nonconstant $\tau$-equivariant rational map over $\F_p$, let $B$ be its branch locus, put
\[
        U=\PP^1\setminus B,
        \qquad
        V=f^{-1}(U),
\]
and let
\[
        C_f=(V\times_UV)\setminus\Delta.
\]
For $q=p^k$, define the finite-etale collision count
\[
        N_f(q)=\#\{(x,y)\in(\Lambda_q\cap V)^2:x\ne y,\ f(x)=f(y)\},
        \qquad
        \Lambda_q=\PP^1(\barF_p)^{\tau^{-1}\Frob_q}.
\]
The corresponding count on all of $\Lambda_q$ differs from this finite-etale count by $O_f(1)$, because only finitely many branch fibers are omitted.  Then
\[
        N_f(q)=a_f(q)q+O_f(q^{1/2}),
\]
where
\[
        a_f(q)=\operatorname{Tr}\!\left(\tau^{-1}\Frob_q\mid
        \Ql[\pi_0(C_{f,\barF_p})]\right),
\]
equivalently the number of geometric components of $C_f$ fixed by $\tau^{-1}\Frob_q$.  After Tate twist, $H_c^2(C_{f,\bar{\mathbb F}_p},\mathbb Q_\ell)(1)$ is the permutation representation on the $G_f$-orbits on ordered pairs of distinct sheets.  In particular, if $\deg f>1$ and $G_f$ is $2$-transitive on the separable sheets, then $N_f(q)=q+O_f(q^{1/2})$, so $f$ is not $\tau$-twisted exceptional.
\end{theorem}

The final group of results describes the positive-characteristic deviations from characteristic-zero genericity.

\begin{theorem}[E. Positive-characteristic proved mechanisms]\label{thm:intro-E-positive}
The following positive-characteristic mechanisms are proved in the body of the paper.
\begin{enumerate}[label=\textup{(\roman*)}]
\item The Frobenius-sparse branch $d=p^a-1$ gives the linear quotient $h_d=\tau^2$, while $d=p^a+1$ gives Kummer or Artin--Schreier normal forms.
\item The characteristic-$19$ degree-parameter $d=6$ quotient is, under successful execution of the certificate in \cref{app:computations}, a Klein-four Galois cover and is $\tau$-twisted exceptional.
\item For $p\ne2,3$ and $d=2p^a+1$, the first nonsparse Frobenius--lacunary tower has an explicit two-branch-value normal form and is indecomposable for every $a\ge1$; for $a=1$, the geometric monodromy of the separable quotient satisfies $A_n\le G_{h_d}\le S_n$, where $n=\deg h_d$ under the separable-quotient convention of \cref{def:quotient-conventions}.
\end{enumerate}
\end{theorem}

Outside these proved mechanisms, the paper does not claim a positive-characteristic classification.  The remaining cases are formulated as explicit wild-inertia or two-transitivity certificate problems in \cref{q:ramification-skeleton} and in the final question of \cref{sec:questions}.

\subsection*{Dependency map}

The following table records the architecture of the paper.
\begin{center}
\small
\begin{tabular}{c p{0.34\textwidth} p{0.43\textwidth}}
\toprule
Layer & Object & Output\\
\midrule
1 & $\Gamma_q^*/\F_q^*\simeq\Lambda_q$ & Hilbert--90 quotient map $h_d$\\
2 & $L_0\cap\boldsymbol\mu_d^2$ & exact torsion-defect formula for $\deg h_d$\\
3 & low quotient degrees & rigidity, Mersenne arithmetic, sparse normal forms\\
4 & critical equation & no characteristic-zero branch collisions\\
5 & branch cycles & $G_{h_d}=S_{\deg h_d}$ in characteristic zero\\
6 & $C_f=(V\times_UV)\setminus\Delta$ & obstruction to twisted exceptionality\\
7 & sparse/sporadic/lacunary degeneration & positive-characteristic structure and remaining wild-inertia problem\\
\bottomrule
\end{tabular}
\end{center}

\subsection*{Relation with prior work}

Within the broad literature on permutation polynomials over finite fields \cite{HouSurvey}, the closest comparison is the two-step construction of Ding, Song and Xiong \cite{DSX2026}.  Their work supplies the cubic and $Q+1$ trace-zero branches and several sparse full-field families.  The present paper uses the same descent philosophy in a different direction: the fixed-exponent quotient-map family \eqref{eq:intro-Hd} is treated as the primary algebraic object.  The main new contributions are the torsion-defect theorem, the low-degree rigidity results, the characteristic-two Mersenne trace-zero family arising from maximal torsion defect, characteristic-zero full symmetric monodromy, the off-diagonal fiber-square obstruction, the certificate-verified characteristic-$19$ Klein-four quotient, and the nonsparse Frobenius--lacunary tower.  A detailed comparison with Ding--Song--Xiong is placed in \cref{sec:dsx-comparison}.

The method also uses standard tools from the theory of exceptional covers and finite-field permutation polynomials: the AGW/fiber criterion \cite{AGW}, Hilbert 90, branch-cycle and monodromy arguments, Jacobi-sum estimates for cyclotomic defects, and the Grothendieck--Lefschetz trace formula for fiber squares.  The point is not to reprove this general machinery, but to show that the family \eqref{eq:intro-Hd} has an unusually rigid algebraic structure under these tools.  In particular, the finite-field statements are not presented as isolated constructions; they are consequences of the cancellation, ramification, and monodromy theory of the nonconstant quotient maps.

\subsection*{Scope of the positive-characteristic results}

The characteristic-zero monodromy theorem is complete: every non-linear characteristic-zero member of the family has full symmetric geometric monodromy.  In positive characteristic, the paper proves the Frobenius-sparse normal forms, gives computer-assisted proofs of the characteristic-$19$ Klein-four quotient and of full monodromy in the bad reduction $(p,d)=(7,5)$ under the certificate convention of \cref{app:computations}, and proves the first nonsparse Frobenius--lacunary tower up to the stated wild-inertia boundary.  Statements beyond these proved mechanisms are formulated either as conditional reductions, conjectural architecture, or concrete certificate problems.  This separation is intentional: it keeps the complete algebraic theorems distinct from the remaining positive-characteristic monodromy problem.

\subsection*{Organization}

Part I constructs the Hilbert--90 quotient-map family and records the equivariance and separability conventions.  Part II proves the torsion-defect formula, the low-degree rigidity theorems, the cyclotomic fixed-index estimates, and the sparse arithmetic consequences.  Part III proves the characteristic-zero branch-collision theorem and the full symmetric monodromy theorem.  Part IV develops the twisted off-diagonal fiber-square obstruction and analyzes the positive-characteristic degeneration mechanisms, including the certificate-verified sporadic Klein-four quotient, ramification skeletons, and Frobenius--lacunary towers.  The appendices collect attribution details, fixed-degree automata, higher-dimensional Hilbert--90 quotients, additive full-field lifts, and reproducible computational certificates.

\section*{Part I. The quotient-map family}
\section{Hilbert--90 quotient maps}\label{sec:descent}

Throughout this section $q=p^k$ with $k\ge1$, $L=\F_{q^3}$, $K=\F_q$, and $\sigma(x)=x^q$.  Write
\[
        \Gamma_q=\ker\Tr_{L/K},
        \qquad
        \Lambda_q=\{z\in L:z^{q+1}+z+1=0\},
        \qquad
        \tau(z)=-1-z^{-1}.
\]
For $B(X)=\sum b_iX^i\in L[X]$ we write
\[
        B^{(q)}(X)=\sum b_i^qX^i.
\]

\begin{proposition}[multiplicative Hilbert--90 quotient]\label{prop:h90-quotient}
The map
\[
        \lambda:\Gamma_q^*\longrightarrow \Lambda_q,
        \qquad
        \lambda(x)=\frac{\sigma x}{x}=x^{q-1},
\]
induces a bijection
\[
        \Gamma_q^*/K^*\xrightarrow{\sim}\Lambda_q.
\]
In particular $|\Lambda_q|=(q^2-1)/(q-1)=q+1$.
\end{proposition}

\begin{proof}
If $x\in\Gamma_q^*$ and $z=\sigma x/x$, then
\[
        0=\frac{x+\sigma x+\sigma^2x}{x}=1+z+z\sigma z.
\]
Since $\sigma z=z^q$, this gives $z^{q+1}+z+1=0$, so $z\in\Lambda_q$.  Also $z\ne0$, and $z\ne-1$ because substituting $z=-1$ in $1+z+z\sigma z$ gives $1$.

Conversely, let $z\in\Lambda_q$.  Then $z^q=-1-z^{-1}=\tau(z)$.  From $1+z+z\sigma z=0$ one checks that
\[
        z\sigma z\sigma^2z=1,
\]
so $\Nm_{L/K}(z)=1$.  By multiplicative Hilbert 90 there is $x\in L^*$ with $z=\sigma x/x$.  Then
\[
        x+\sigma x+\sigma^2x=x(1+z+z\sigma z)=0,
\]
so $x\in\Gamma_q^*$.  The ambiguity in $x$ is exactly multiplication by elements of $K^*$, which proves the bijection.  Since $L/K$ is a finite separable extension, the trace map $\Tr_{L/K}:L\to K$ is a nonzero $K$-linear functional and hence is surjective.  Therefore $\Gamma_q$ has $K$-dimension $2$, so $|\Gamma_q|=q^2$, and the cardinality follows.
\end{proof}

\begin{proposition}[trace-zero multiplicative descent]\label{prop:mult-desc}
Let $r\ge1$ and $B\in L[X]$.  Define the function $F_{r,B}:\Gamma_q\to L$ by
\[
        F_{r,B}(x)=x^rB(x^{q-1})
\]
and put
\[
        R_{r,B}(z)=z^r\frac{B^{(q)}(\tau(z))}{B(z)}.
\]
Then $F_{r,B}$ permutes $\Gamma_q$ if and only if
\begin{enumerate}[label=\textup{(\roman*)}]
\item $\gcd(r,q-1)=1$;
\item $B$ has no zero on $\Lambda_q$;
\item under condition \textup{(ii)}, the rational expression $R_{r,B}$ is defined on all of $\Lambda_q$ and the resulting map $R_{r,B}:\Lambda_q\to\Lambda_q$ is bijective.
\end{enumerate}
\end{proposition}

\begin{proof}
The zero element is fixed by $F_{r,B}$.  The nonzero points are controlled by the quotient map of \cref{prop:h90-quotient}.  For $c\in K^*$ and $x\in\Gamma_q^*$,
\[
        F_{r,B}(cx)=c^rF_{r,B}(x),
        \qquad
        (cx)^{q-1}=x^{q-1}.
\]
Thus, on a nonzero quotient fiber on which $B$ does not vanish, the map along the $K^*$-fiber is multiplication by $c^r$; it is bijective exactly when $\gcd(r,q-1)=1$.

Assume first that $F_{r,B}$ permutes $\Gamma_q$.  If $B(z)=0$ for some $z\in\Lambda_q$, choose $x\in\Gamma_q^*$ with $x^{q-1}=z$ by \cref{prop:h90-quotient}; then $F_{r,B}(x)=0=F_{r,B}(0)$, contradicting injectivity.  Hence condition \textup{(ii)} holds, and the fiber calculation above gives condition \textup{(i)}.  For $z=x^{q-1}\in\Lambda_q$ we have $F_{r,B}(x)\in\Gamma_q^*$, so
\[
        \frac{\sigma(F_{r,B}(x))}{F_{r,B}(x)}
        =z^r\frac{B^{(q)}(\sigma z)}{B(z)}
        =z^r\frac{B^{(q)}(\tau(z))}{B(z)}
        =R_{r,B}(z)
\]
lies in $\Lambda_q$.  The quotient map induced by a bijection of $\Gamma_q^*$ is bijective on $\Gamma_q^*/K^*\simeq\Lambda_q$, so condition \textup{(iii)} follows.

Conversely, assume \textup{(i)--(iii)}.  If $x\in\Gamma_q^*$ and $z=x^{q-1}$, then $B(z)\ne0$, so $F_{r,B}(x)\ne0$.  The same computation gives
\[
        \sigma(F_{r,B}(x))/F_{r,B}(x)=R_{r,B}(z)\in\Lambda_q.
\]
By the defining equation of $\Lambda_q$, this implies
\[
        F_{r,B}(x)+\sigma(F_{r,B}(x))+\sigma^2(F_{r,B}(x))=0,
\]
so $F_{r,B}(x)\in\Gamma_q^*$.  Condition \textup{(iii)} gives bijectivity on the quotient $\Lambda_q$, and condition \textup{(i)} gives bijectivity on every $K^*$-fiber.  Together with $F_{r,B}(0)=0$, this proves that $F_{r,B}$ permutes $\Gamma_q$.
\end{proof}

\begin{remark}
\Cref{prop:mult-desc} is the Ding--Song--Xiong second descent criterion in the notation of this paper; compare \cite[Lemma 3.3]{DSX2026}.  Its role here is to isolate the fixed-exponent quotient formula $H_d^{\rm raw}$ and the reduced maps $h_d$ studied below.
\end{remark}

For
\[
        P_d(x)=x^{dq}-x^d=x^d((x^{q-1})^d-1)
\]
we have $r=d$ and $B(z)=z^d-1$, hence the raw quotient formula is
\begin{equation}\label{eq:Hd}
        H_d^{\rm raw}(z)=z^d\frac{\tau(z)^d-1}{z^d-1}
        =\frac{(-1)^d(z+1)^d-z^d}{z^d-1}.
\end{equation}
The denominator condition in \cref{prop:mult-desc} is
\begin{equation}\label{eq:denomcondition}
        \mu_d\cap\Lambda_q=\varnothing.
\end{equation}

\begin{lemma}[tame nonconstant quotients are separable]\label{lem:tame-quotient-separable}
Let $k$ be algebraically closed of characteristic $p\ge0$, let $e\ge1$, and assume $p\nmid e$ if $p>0$.  If the reduced quotient $h_e$ is nonconstant, then $h_e$ is separable.
\end{lemma}

\begin{proof}
There is nothing to prove in characteristic zero.  Assume $p>0$ and $p\nmid e$.  Write
\[
        N_e=(-1)^e(z+1)^e-z^e,
        \qquad
        D_e=z^e-1,
        \qquad
        C_e=\gcd(N_e,D_e),
\]
and $h_e=n_e/d_e$ with $N_e=C_en_e$ and $D_e=C_ed_e$.  If $e=1$, then
\[
        H_1^{\rm raw}(z)=\frac{-2z-1}{z-1};
\]
in characteristic $3$ this quotient is constant, while in every other characteristic it has degree one and is separable.  Thus assume $e\ge2$.

A direct calculation gives
\[
        N_e'D_e-N_eD_e'=eF_e(z),
        \qquad
        F_e(z)=z^{e-1}-(-1)^e(z+1)^{e-1}(z^{e-1}+1),
\]
and, after cancellation,
\[
        N_e'D_e-N_eD_e'=C_e^2(n_e'd_e-n_ed_e').
\]
The polynomial $F_e$ is not identically zero, because its leading term is $-(-1)^e z^{2e-2}$.  Since $e\ne0$ in $k$, the Wronskian $n_e'd_e-n_ed_e'$ is not identically zero.  A nonconstant rational function over a perfect field of characteristic $p$ is inseparable exactly when its Wronskian is identically zero.  Hence $h_e$ is separable.
\end{proof}

\begin{definition}[raw, reduced, and separable quotients]\label{def:quotient-conventions}
The expression in \eqref{eq:Hd} is the \emph{raw quotient} and is denoted by $H_d^{\rm raw}$ whenever cancellation or degree is at issue.  The symbol $h_d$ denotes the reduced rational function obtained from $H_d^{\rm raw}$ by canceling the common factor of its numerator and denominator.  Unless explicitly stated otherwise, all quotient degrees are morphism degrees of $h_d$; by convention a constant reduced quotient has degree $0$.  A constant reduced quotient is not regarded as a cover.  Any statement about covers, branch cycles, or monodromy applies only to the associated nonconstant separable quotient.  Whenever indecomposability is used to infer primitivity of a geometric monodromy action, it means indecomposability after base change to the algebraic closure, unless a smaller field is explicitly specified.

In characteristic $p>0$, if $d=p^sd_0$ with $p\nmid d_0$, then \cref{prop:p-power-reduction} shows that $h_d=h_{d_0}^{p^s}$ as a purely inseparable Frobenius twist of $h_{d_0}$.  Concretely, if $h_{d_0}=A/B$ in lowest terms, then
\[
        h_{d_0}^{p^s}=\frac{A(z)^{p^s}}{B(z)^{p^s}}
        =\Frob_{p^s}\circ h_{d_0};
\]
this notation never denotes an iterate of the rational map $h_{d_0}$.  When $h_{d_0}$ is nonconstant, \cref{lem:tame-quotient-separable} shows that it is automatically separable; the separable cover associated with $h_d$ is therefore the cover associated with $h_{d_0}$, and any monodromy assertion for $h_d$ means the monodromy of that cover.  If $h_{d_0}$ is constant, no cover or monodromy is attached.  Throughout the paper, all geometric monodromy assertions in positive characteristic are assertions about the relevant nonconstant separable quotient; the full morphism degree is multiplied by the displayed Frobenius factor.
\end{definition}

\section{Equivariance, reduced quotients, and twisted fixed sets}\label{sec:equivariance}

The quotient maps studied in the paper are not arbitrary rational functions: they commute with the order-three automorphism $\tau$.  We record the equivariance and fix the conventions for reduced and separable quotients before entering the torsion and monodromy calculations.

\begin{proposition}[$\tau$-equivariance of the quotient]\label{prop:tau-equivariance}
For every $d\ge1$, as rational functions on $\PP^1$ one has
\[
        H_d^{\rm raw}\circ\tau=\tau\circ H_d^{\rm raw}.
\]
Consequently the reduced quotient $h_d$ also satisfies $h_d\circ\tau=\tau\circ h_d$.
\end{proposition}

\begin{proof}
Put
\[
        a=z,\qquad b=\tau(z),\qquad c=\tau^2(z).
\]
Then $abc=1$.  Let $A=a^d$, $B=b^d$, and $C=c^d$, so $ABC=1$.  From the definition,
\[
        H_d^{\rm raw}(a)=A\frac{B-1}{A-1}.
\]
Therefore
\[
\begin{aligned}
        \tau(H_d^{\rm raw}(a))
        &=-1-\frac1{H_d^{\rm raw}(a)}                                      \\
        &=-\frac{H_d^{\rm raw}(a)+1}{H_d^{\rm raw}(a)}                                \\
        &=-\frac{AB-1}{A(B-1)}                                    \\
        &=\frac{1-AB}{A(B-1)}.
\end{aligned}
\]
Since $C=(AB)^{-1}$,
\[
        H_d^{\rm raw}(b)=B\frac{C-1}{B-1}
        =B\frac{(AB)^{-1}-1}{B-1}
        =\frac{1-AB}{A(B-1)}.
\]
Thus $H_d^{\rm raw}(\tau z)=\tau(H_d^{\rm raw}(z))$ on the dense open set where the displayed expressions are defined, and hence as rational functions on $\PP^1$.  Passing to the equivalent reduced rational function gives the statement for $h_d$.
\end{proof}

\begin{proposition}[cyclic normal forms]\label{prop:cyclic-normal-form}
Let $k$ be an algebraically closed field and let $f\in k(z)$ be a nonconstant rational function satisfying $f\circ\tau=\tau\circ f$.
\begin{enumerate}[label=\textup{(\roman*)}]
\item If $\operatorname{char}k\ne3$, then after the change of coordinate
\[
        u=\frac{z-\omega}{z-\omega^2},
\]
where $\omega$ is a primitive cube root of unity and the coordinate is chosen so that $\tau$ acts as $u\mapsto\omega u$, the conjugate $\Phi(u)$ of $f$ has the form
\[
        \Phi(u)=uR(u^3)
\]
for some $R\in k(t)$.
\item If $\operatorname{char}k=3$, then after a fractional linear change of coordinate conjugating $\tau$ to $u\mapsto u+1$, the conjugate $\Phi(u)$ of $f$ has the form
\[
        \Phi(u)=u+R(u^3-u)
\]
for some $R\in k(t)$.
\end{enumerate}
\end{proposition}

\begin{proof}
Assume first that $\operatorname{char}k\ne3$.  The automorphism $\tau$ has two fixed points, the primitive cube roots $\omega,\omega^2$, and is diagonalizable in $\operatorname{PGL}_2(k)$.  The displayed coordinate sends these fixed points to $0$ and $\infty$, so $\tau$ becomes multiplication by a primitive cube root, which we denote again by $\omega$.  The equivariance relation is
\[
        \Phi(\omega u)=\omega\Phi(u).
\]
Thus $\Phi(u)/u$ is invariant under $u\mapsto\omega u$, and the invariant field is $k(u^3)$.  Hence $\Phi(u)=uR(u^3)$.

In characteristic $3$, every nontrivial order-three element of $\operatorname{PGL}_2(k)$ is unipotent and hence conjugate to $u\mapsto u+1$.  The equivariance relation becomes
\[
        \Phi(u+1)=\Phi(u)+1.
\]
Therefore $\Phi(u)-u$ is invariant under translation by $1$.  The invariant field for this Artin--Schreier action is $k(u^3-u)$, so $\Phi(u)=u+R(u^3-u)$.
\end{proof}

\begin{definition}[twisted exceptionality]\label{def:twisted-exceptional}
Let $f\in\F_p(z)$ be a nonconstant rational function satisfying $f\circ\tau=\tau\circ f$.  For $q=p^k$ put
\[
        \Thetaq=\tau^{-1}\circ\Frob_q
\]
on $\PP^1_{\barF_p}$.  We say that $f$ is $\tau$-twisted exceptional if $f$ induces a bijection on the finite set
\[
        \PP^1(\barF_p)^{\Thetaq}
\]
for infinitely many $k\ge1$.
\end{definition}

\begin{remark}\label{rem:fixedset-lambda}
The fixed-point condition $\Thetaq(z)=z$ is equivalent to $z^q=\tau(z)$.  The three points $0,-1,\infty$ are not fixed.  Hence
\[
        \PP^1(\barF_p)^{\Thetaq}
        =\{z\in\barF_p^*:z^q=-1-z^{-1}\}
        =\{z:z^{q+1}+z+1=0\}=\Lambda_q.
\]
Any solution of $z^q=\tau(z)$ satisfies $z^{q^3}=z$, because iterating gives $z^{q^2}=\tau^2(z)$ and $z^{q^3}=\tau^3(z)=z$.  Thus this geometric fixed set is in fact contained in $\F_{q^3}$, matching the finite-field set introduced in \cref{sec:descent}.
Consequently \cref{def:twisted-exceptional} is the quotient-level fixed-set condition associated with the trace-zero problem.  For the original map $P_d$ on $\Gamma_q$, \cref{prop:mult-desc} also requires the fiber condition and the denominator condition $\mu_d\cap\Lambda_q=\varnothing$.
\end{remark}

\section*{Part II. Torsion defects and arithmetic strata}
\section{Torsion defects: the exact degree formula}\label{sec:torsion-defect}

Let $k$ be an algebraically closed field of characteristic $p\ge0$.  In this section $d\ge1$, and when $p>0$ we first assume $p\nmid d$.  Put
\[
        N_d(z)=(-1)^d(z+1)^d-z^d,
        \qquad
        D_d(z)=z^d-1.
\]
Thus $H_d^{\rm raw}=N_d/D_d$ before cancellation.

\begin{definition}
Let
\[
        L_0:\ 1+X+Y=0\subset \Gm^2.
\]
For $p\nmid d$, define the $d$-torsion defect of $h_d$ by
\[
        \delta_d=\ell(L_0\cap\muGP_d^2),
\]
where the scheme-theoretic intersection is computed on the toric line $L_0$ with coordinate $X$ and $Y=-1-X$.  Since $\muGP_d$ is etale in the tame case, this length is simply
\[
        \delta_d=\#\{(x,y)\in\mu_d^2:1+x+y=0\}.
\]
\end{definition}

\begin{theorem}[toric torsion-defect formula]\label{thm:torsion-defect}
Assume $p\nmid d$ if $p>0$.  Then the reduced quotient $h_d$ satisfies
\[
        \deg h_d=d-\delta_d.
\]
Equivalently,
\[
        \deg h_d=d-
        \#\{\zeta\in\mu_d: \tau(\zeta)\in\mu_d\}.
\]
\end{theorem}

\begin{proof}
The denominator $D_d=z^d-1$ is squarefree.  Its roots are the elements $\zeta\in\mu_d$.  Such a root is canceled by the numerator precisely when
\[
        0=N_d(\zeta)=(-1)^d(\zeta+1)^d-\zeta^d.
\]
Since $\zeta^d=1$, this is equivalent to
\[
        (-1)^d(\zeta+1)^d=1.
\]
But
\[
        \tau(\zeta)^d=\left(-\frac{\zeta+1}{\zeta}\right)^d
        =(-1)^d(\zeta+1)^d,
\]
again because $\zeta^d=1$.  Hence the common roots of $N_d$ and $D_d$ are exactly the $\zeta\in\mu_d$ with $\tau(\zeta)\in\mu_d$.

The map
\[
        \zeta\longmapsto (\zeta,\zeta\tau(\zeta))=(\zeta,-1-\zeta)
\]
identifies this set with
\[
        L_0\cap\mu_d^2,
        \qquad L_0:\ 1+X+Y=0.
\]
Because $D_d$ is squarefree, each cancellation is simple, so the cancellation degree is $\delta_d$.

Let $C_d=\gcd(N_d,D_d)$.  Since $D_d$ is squarefree, $C_d$ is a product of exactly the $\delta_d$ canceled linear factors.  After cancellation,
\[
        d_d=D_d/C_d,
        \qquad \deg d_d=d-\delta_d.
\]
The numerator satisfies
\[
        n_d=N_d/C_d,
        \qquad \deg n_d\le d-\delta_d,
\]
because $\deg N_d\le d$ and the same common factor $C_d$ of degree $\delta_d$ has been removed.  Hence
\[
        \deg h_d=\max(\deg n_d,\deg d_d)=d-\delta_d.
\]
This also covers the constant case: if $d-\delta_d=0$, both reduced numerator and denominator are constant and the convention gives morphism degree $0$.
\end{proof}

\begin{corollary}[reduced denominator degree]\label{cor:reduced-denominator-degree}
Under the hypotheses of \cref{thm:torsion-defect}, after writing $h_d=n_d/d_d$ in lowest terms with $d_d$ monic, one has
\[
        d_d=(z^d-1)/C_d,
        \qquad C_d=\gcd(N_d,D_d),
\]
and therefore $\deg d_d=d-\delta_d=\deg h_d$.
\end{corollary}

\begin{proof}
This is the denominator computation in the proof of \cref{thm:torsion-defect}.  The denominator $D_d=z^d-1$ is squarefree, so exactly the $\delta_d$ canceled linear factors are removed from $D_d$ and no other denominator factor is lost.
\end{proof}

\begin{corollary}[three-term vanishing sums]\label{cor:vanishing-sum}
For $p\nmid d$, the defect $\delta_d$ is the number of ordered three-term vanishing sums
\[
        1+x+y=0,
        \qquad x,y\in\mu_d.
\]
Thus the quotient degree is controlled by the torsion intersection of the toric line $L_0$ with the $d$-torsion subgroup of $\Gm^2$.
\end{corollary}

\begin{proof}
This is the second formulation of \cref{thm:torsion-defect}.  The connection with vanishing sums of roots of unity is classical; see \cite{Mann1965,ConwayJones1976,LamLeung} and \cite{GranvilleRudnick} for broader background.
\end{proof}

\begin{theorem}[wild torsion-defect decomposition]\label{prop:p-power-reduction}
Suppose $p>0$ and write $d=p^s d_0$ with $p\nmid d_0$.  Then
\[
        H_d^{\rm raw}(z)=\bigl(H_{d_0}^{\rm raw}(z)\bigr)^{p^s}
\]
as raw quotient expressions over characteristic $p$, where the exponent means Frobenius on values rather than iteration.  Equivalently $h_d=h_{d_0}^{p^s}$ after reduction.

Moreover, scheme-theoretically on the toric line $L_0$, one has
\[
        \ell\bigl(L_0\cap\boldsymbol\mu_d^2\bigr)
        =p^s\ell\bigl(L_0\cap\boldsymbol\mu_{d_0}^2\bigr),
\]
and therefore
\[
        \deg h_d
        =d-\ell\bigl(L_0\cap\boldsymbol\mu_d^2\bigr)
        =p^s\deg h_{d_0}.
\]
If, in addition, $q=p^k$, $L=\F_{q^3}$, $K=\F_q$, $\Gamma_q=\ker\Tr_{L/K}$, and
\[
        P_e(X)=X^{eq}-X^e\qquad(e\ge1),
\]
then
\[
        P_d(X)=P_{d_0}(X)^{p^s}
\]
as functions on $L$.  Consequently $P_d$ permutes $\Gamma_q$ if and only if $P_{d_0}$ permutes $\Gamma_q$.
\end{theorem}

\begin{proof}
In characteristic $p$,
\[
        z^d-1=(z^{d_0}-1)^{p^s}
\]
and
\[
        (-1)^d(z+1)^d-z^d
        =\left((-1)^{d_0}(z+1)^{d_0}-z^{d_0}\right)^{p^s}.
\]
This proves $H_d^{\rm raw}=(H_{d_0}^{\rm raw})^{p^s}$, and therefore $h_d=h_{d_0}^{p^s}$ after canceling common factors.

On the toric line $L_0$ we use the coordinate $X$, with $Y=-1-X$, and localize away from $X=0$ and $1+X=0$.  The intersection with $\boldsymbol\mu_d^2$ is defined by
\[
        X^d-1=0,
        \qquad (-1-X)^d-1=0.
\]
These two equations are the $p^s$-th powers of the corresponding equations for $d_0$.  At every closed point of the reduced intersection for $d_0$, the functions $X^{d_0}-1$ and $(-1-X)^{d_0}-1$ vanish simply because $p\nmid d_0$ and the point lies in the torus.  Thus, in the completed local ring $k[[u]]$, both tame local equations are unit multiples of $u$, while the wild local equations generate $(u^{p^s})$.  Each reduced point therefore contributes length $p^s$.  The degree identity follows from \cref{thm:torsion-defect} applied to $d_0$ and the equality $\deg h_d=p^s\deg h_{d_0}$.

For the finite-field assertion, fix $q=p^k$, $L=\F_{q^3}$, and $K=\F_q$.  Then
\[
        X^{dq}-X^d=(X^{d_0q}-X^{d_0})^{p^s}.
\]
The $p^s$-power Frobenius is a bijection of $L$ and preserves $\Gamma_q=\ker\Tr_{L/K}$, so the permutation property is unchanged.
\end{proof}

\section{Low-degree rigidity and maximal torsion collapse}\label{sec:rigidity}

This section proves \cref{thm:degree-one,thm:degree-two}.  We first isolate the elementary binomial rigidity needed for degree one.

\begin{lemma}[geometric binomial coefficients]\label{lem:binom-geometric}
Let $k$ be a field of characteristic $p>0$, let $d\ge4$ with $p\nmid d$, and suppose that the interior binomial coefficients
\[
        \binom d1,\binom d2,\ldots,\binom d{d-1}
\]
are nonzero and form a geometric progression in $k^*$.  Then $d=p^a-1$ for some $a\ge1$.
\end{lemma}

\begin{proof}
Write
\[
        \binom d{j+1}=\rho\binom dj
        \qquad(1\le j\le d-2)
\]
for some $\rho\in k^*$.  The identity
\[
        (j+1)\binom d{j+1}=(d-j)\binom dj
\]
gives
\[
        (j+1)\rho=d-j
        \qquad(1\le j\le d-2).
\]
Taking $j=1$ and $j=2$ and subtracting gives $\rho=-1$, without dividing by $2$ or $3$.  Then $2\rho=d-1$ gives $d+1=0$ in $k$.

Since $\binom d1=d=-1$ in $k$ and the ratio is $-1$, the interior coefficients satisfy
\[
        \binom dj=(-1)^j\qquad(1\le j\le d-1).
\]
For $j=d-1$ this says $d=(-1)^{d-1}$; as $d=-1$ in $k$, we get $(-1)^d=1$, so the endpoint $\binom dd=1$ also equals $(-1)^d$.  Hence
\[
        (1+z)^d=\sum_{j=0}^d(-1)^jz^j.
\]
Multiplying by $1+z$ gives
\[
        (1+z)^{d+1}=1+z^{d+1}.
\]
By Lucas' theorem \cite{Lucas1878}, equivalently by writing $d+1$ in base $p$, all intermediate binomial coefficients of $(1+z)^{d+1}$ vanish in characteristic $p$ if and only if $d+1$ is a power of $p$.  Thus $d+1=p^a$.
\end{proof}

\begin{lemma}[degree one forces a geometric progression]\label{lem:degree-one-geometric}
Assume $p\nmid d$ if $p>0$, and let $d\ge4$.  If the reduced rational function $h_d$ has degree at most one, then the interior binomial coefficients
\[
        \binom d1,\ldots,\binom d{d-1}
\]
are nonzero and form a geometric progression in the ground field.
\end{lemma}

\begin{proof}
Write $N_d=(-1)^d(z+1)^d-z^d$ and $D_d=z^d-1$.  Since the reduced form of $N_d/D_d$ has degree at most one, there are constants $A,B,C,E$, not all zero, such that
\begin{equation}\label{eq:linear-relation}
        (Cz+E)N_d=(Az+B)D_d.
\end{equation}
A constant map would force $N_d$ to be a scalar multiple of $D_d$, which is impossible for $d\ge4$ because the coefficient of $z$ in $N_d$ is $(-1)^d d\ne0$ whereas $D_d$ has no $z$-term.  Thus \eqref{eq:linear-relation} represents a genuine degree-at-most-one reduced map.

Let $n_j$ be the coefficient of $z^j$ in $N_d$ for $1\le j\le d-1$.  Then $n_j=(-1)^d\binom dj$ and $n_1=n_{d-1}=(-1)^d d\ne0$.  We first show $C,E\ne0$.  If $E=0$, then comparing coefficients of $z^j$ for $2\le j\le d-1$ in \eqref{eq:linear-relation} gives $Cn_{j-1}=0$, hence $n_1=0$, contradiction.  If $C=0$, the same comparison gives $En_j=0$ for $2\le j\le d-1$, hence $n_{d-1}=0$, again a contradiction.

Now for $2\le j\le d-1$, the coefficient of $z^j$ on the right side of \eqref{eq:linear-relation} is zero, because $(Az+B)(z^d-1)$ has only the terms $Az^{d+1}$, $Bz^d$, $-Az$, and $-B$.  Hence
\[
        Cn_{j-1}+En_j=0
        \qquad(2\le j\le d-1).
\]
Since $C,E\ne0$ and $n_1\ne0$, the sequence $n_1,\ldots,n_{d-1}$ is nonzero and geometric.  Therefore the same is true of the interior binomial coefficients.
\end{proof}

\begin{theorem}[degree-one rigidity]\label{thm:degree-one}
Let $k$ be an algebraically closed field of characteristic $p\ge0$, let $d\ge1$, and assume $p\nmid d$ if $p>0$.  Let $h_d$ be the reduced rational function associated with \eqref{eq:Hd}.  If $p=0$, then $\deg h_d\le1$ if and only if $d=1$ or $d=3$.  If $p>0$, then
\[
        \deg h_d\le1
        \quad\Longleftrightarrow\quad
        d=1,
        \quad d=3,
        \quad\text{or}\quad d=p^a-1\text{ for some }a\ge1.
\]
Here the positive-characteristic entry $d=3$ is subject to the standing hypothesis $p\nmid d$, hence means $p\ne3$.  For $d=p^a-1$ one has $h_d=\tau^2$.
\end{theorem}

\begin{proof}
The small cases are explicit.  For $d=1$,
\[
        H_1^{\rm raw}(z)=\frac{-(z+1)-z}{z-1}=\frac{-2z-1}{z-1}.
\]
For $d=2$,
\[
        H_2^{\rm raw}(z)=\frac{(z+1)^2-z^2}{z^2-1}=\frac{2z+1}{z^2-1},
\]
which has degree $2$ unless characteristic $3$, in which case $d=2=3^1-1$.  For $d=3$ and $p\ne3$,
\[
        H_3^{\rm raw}(z)=\frac{-(z+1)^3-z^3}{z^3-1}
        =-\frac{2z+1}{z-1}.
\]
In the small overlaps with the branch $d=p^a-1$, these formulas also give the asserted identity $h_d=\tau^2$: in characteristic $2$ for $d=1$ and $d=3$, and in characteristic $3$ for $d=2$, the displayed expressions reduce to $-1/(z+1)=\tau^2(z)$.

Now assume $d\ge4$ and $p=0$.  If $\deg h_d\le1$, the same coefficient comparison as in \cref{lem:degree-one-geometric} implies that the integer binomial coefficients $\binom d1,\ldots,\binom d{d-1}$ form a geometric progression in $\mathbb Q^*$.  This is impossible for $d\ge4$ since the ratios $\binom d2/\binom d1=(d-1)/2$ and $\binom d3/\binom d2=(d-2)/3$ are unequal over $\mathbb Q$.

Finally assume $p>0$ and $d\ge4$.  By \cref{lem:degree-one-geometric,lem:binom-geometric}, degree at most one forces $d=p^a-1$.  Conversely, if $d=p^a-1$, then in characteristic $p$,
\[
        (z+1)^{p^a}=z^{p^a}+1.
\]
Using $d=p^a-1$ gives $(-1)^d=1$ in the ground field.  Therefore
\[
        (z+1)^d=\frac{z^{p^a}+1}{z+1},
        \qquad
        z^d=\frac{z^{p^a}}{z}.
\]
Thus
\[
\begin{aligned}
        H_d^{\rm raw}(z)
        &=\frac{(z^{p^a}+1)/(z+1)-z^{p^a}/z}{z^{p^a}/z-1}  \\
        &=\frac{z-z^{p^a}}{z(z+1)}\cdot\frac{z}{z^{p^a}-z}
        =-\frac1{z+1}=\tau^2(z)
\end{aligned}
\]
as an identity in $k(z)$, so the reduced quotient is $h_d=\tau^2$.  This completes the proof.
\end{proof}

\begin{corollary}[maximal torsion defect]\label{cor:max-defect}
Assume $p\nmid d$.  Then
\[
        \ell(L_0\cap\muGP_d^2)\ge d-1
\]
if and only if $d=1$, $d=3$, or, in positive characteristic, $d=p^a-1$ for some $a\ge1$.
\end{corollary}

\begin{proof}
By \cref{thm:torsion-defect}, the inequality is equivalent to $\deg h_d\le1$.  Apply \cref{thm:degree-one}.
\end{proof}

We now classify the next stratum.  The proof is a coefficient-comparison argument.  The exceptional pair $(p,d)=(11,5)$ is not an artifact of the method; it is the index-two cyclotomic case over $\F_{11}$ in \cref{cor:index2-low}.  The following lemma spells out the algebraic eliminations used in the characteristic-zero and characteristic-$p\ge5$ part of the proof.

\begin{lemma}[degree-two elimination details]\label{lem:degree-two-elimination-details}
Let $d\ge4$, let $b_j=\binom dj$, let $\eps=(-1)^d$, and suppose that
\[
        Q(z)N_d(z)=P(z)D_d(z),
        \qquad Q(z)=z^2-Sz+T,
        \qquad T\ne0,
\]
with $\deg P\le2$.  For $3\le j\le d-1$ put
\[
        R_j=b_{j-2}-S b_{j-1}+T b_j.
\]
Then coefficient comparison gives $R_j=0$ in this range.  In characteristic zero or characteristic $p\ge5$, the boundary coefficients in degrees $0,1,2,d,d+1,d+2$ give the following explicit eliminations, all interpreted in the ground field.  Whenever a displayed residual contains a denominator, it is used only in characteristics where that denominator is invertible; the only residual with denominator $120$ is also recorded below in cleared form before it is specialized.

If $\eps=1$, the boundary equations are
\[
        b_2-Sd+T=0,
        \qquad S=d(T+1),
        \qquad b_2T-Sd+1=0.
\]
If $d+1\ne0$, they imply
\[
        (d-2)(3d-2)=0.
\]
On the branch $d=2$ one obtains $S=0$ and $T=-1$, and $R_3=d-b_3$.  On the branch $3d-2=0$, after substituting the boundary solution one has
\[
        R_3=-\frac{d(d-3)(d+1)}3,
        \qquad
        R_4=-\frac{d(d-4)(d+1)(5d-6)}{48}.
\]

If $\eps=-1$, the boundary equations are
\[
        b_2-Sd+3T=0,
        \qquad 3S=d(T+1),
        \qquad b_2T-Sd+3=0.
\]
If $d=3$ in the ground field, these equations reduce to $T=S-1$ and $R_3=2-2S$, so $R_3=0$ forces $S=1$ and hence $T=0$.  If $d-3\ne0$, the boundary equations imply
\[
        (d+2)(d+6)=0.
\]
Modulo this relation, the residual at $j=3$ is
\[
        R_3=-\frac{d^3+3d^2+2d-12}{6}.
\]
On the branch $d=-6$, if this residual vanishes, then the characteristic is $11$.  On the same branch the $j=5$ residual satisfies the cleared identity
\[
        120R_5=-(d-2)(d-1)(3d^3+3d^2-32d-72).
\]
Consequently, after the preceding residual has forced characteristic $11$, this is equivalently
\[
        R_5=-\frac{(d-2)(d-1)(3d^3+3d^2-32d-72)}{120}.
\]
\end{lemma}

\begin{proof}
The recurrence $R_j=0$ is exactly the coefficient comparison in the degrees $3\le j\le d-1$, because the right side $P(z)(z^d-1)$ has no terms in those degrees.  The boundary equations are obtained by comparing the remaining coefficients in degrees $0,1,2,d,d+1,d+2$ and using $b_2=d(d-1)/2$.

For $\eps=1$, substituting $S=d(T+1)$ into the first and third boundary equations gives two linear equations in $T$:
\[
        (1-d^2)T+\left(b_2-d^2\right)=0,
        \qquad
        \left(b_2-d^2\right)T+(1-d^2)=0.
\]
Their resultant is
\[
        (1-d^2)^2-\left(b_2-d^2\right)^2
        =\frac{(d-2)(d+1)^2(3d-2)}4.
\]
Since $d+1\ne0$, this gives $(d-2)(3d-2)=0$.  If $d=2$, the boundary equations give $S=0$ and $T=-1$, hence $R_3=b_1-b_3=d-b_3$.  On the branch $3d-2=0$, the first two boundary equations give
\[
        S=\frac{d(d-2)}{2d-2},
        \qquad
        T=-\frac d{2d-2},
\]
and substitution into $R_3$ and $R_4$ gives the two displayed formulas.

For $\eps=-1$, if $d=3$ in the ground field, the displayed reductions are immediate from the three boundary equations and the formula for $R_3=b_1-Sb_2+Tb_3$.  Now assume $d-3\ne0$.  If $d+3=0$, the first two boundary equations are inconsistent: they reduce to $S+T+2=0$ and $S+T+1=0$.  Hence $d+3\ne0$.  Solving the first two boundary equations gives
\[
        S=\frac{d(d+2)}{2d+6},
        \qquad
        T=\frac d{2d+6}.
\]
Substitution into the third boundary equation gives
\[
        -\frac{(d-3)(d+2)(d+6)}{4(d+3)}=0,
\]
so, under $d-3\ne0$, one obtains $(d+2)(d+6)=0$.  Substitution into $R_3$ gives
\[
        R_3+\frac{d^3+3d^2+2d-12}{6}
        =\frac{(d-1)(d+2)(d+6)}{2(d+3)},
\]
which proves the displayed residual modulo $(d+2)(d+6)$.  Similarly, after clearing the factor $120$ and the already-invertible factor $d+3$, one has
\[
\begin{aligned}
 &(d+3)\left(120R_5+(d-2)(d-1)(3d^3+3d^2-32d-72)\right) \\
 &\qquad=(d-3)(d-2)(d-1)(d+2)(d+6)^2 .
\end{aligned}
\]
Since $d+3\ne0$, the cleared $R_5$ identity holds on the branch $d=-6$.  Finally, the $j=3$ residual on $d=-6$ is $22$, so its vanishing forces characteristic $11$; only after this point do we divide by $120$.
\end{proof}

\begin{theorem}[degree-two rigidity]\label{thm:degree-two}
Let $k$ be an algebraically closed field of characteristic $p\ge0$, let $d\ge1$, and assume $p\nmid d$ if $p>0$.  Then $\deg h_d=2$ if and only if either
\[
        d=2\quad\text{and}\quad \bigl(p=0\text{ or }p\ne2,3\bigr),
\]
or
\[
        (p,d)=(11,5).
\]
In the exceptional case,
\[
        h_5(z)=\frac{-2z^2-2z+1}{z^2+4z+1}
        \qquad\text{in characteristic }11.
\]
\end{theorem}

\begin{proof}
The case $d=2$ follows from
\[
        H_2^{\rm raw}(z)=\frac{2z+1}{z^2-1}.
\]
It has degree $2$ in characteristic zero and in characteristic $p\ne2,3$; in characteristic $3$ it is the degree-one case $d=2=3^1-1$.  A direct calculation in characteristic $11$ gives
\[
        \gcd\bigl(z^5-1,-(z+1)^5-z^5\bigr)=z^3-4z^2+4z-1,
\]
(the Sage verification in \cref{app:computations} records this computation reproducibly),
and hence
\[
        h_5(z)=\frac{-2z^2-2z+1}{z^2+4z+1}.
\]
Thus the listed cases have degree $2$.

Conversely, suppose $\deg h_d=2$ and $d\ge4$.  By \cref{thm:degree-one}, we may assume that we are not in a degree-one case.  By \cref{cor:reduced-denominator-degree}, the reduced denominator has degree $2$, equivalently exactly two denominator factors of $D_d$ remain uncanceled.  Let
\[
        Q(z)=z^2-Sz+T,
        \qquad T\ne0,
\]
be the denominator of the reduced form of $h_d$.  Since the unreduced denominator is $z^d-1$ and is squarefree, $Q$ is the product of the two uncanceled denominator factors.  There is a polynomial $P$ of degree at most $2$ such that
\begin{equation}\label{eq:quadratic-identity}
        Q(z)N_d(z)=P(z)D_d(z).
\end{equation}
Put $b_j=\binom dj$ and $\eps=(-1)^d$.  For $1\le j\le d-1$ the coefficient of $z^j$ in $N_d$ is $\eps b_j$.  In the coefficient-comparison argument below, equations such as $d=-1$, $3d-2=0$, or $d\equiv5\pmod {11}$ are equations for the image of the integer $d$ in the ground field unless the text explicitly says ``as an integer.''  Comparing coefficients of $z^j$ in \eqref{eq:quadratic-identity} for $3\le j\le d-1$ gives
\begin{equation}\label{eq:quadratic-recurrence}
        b_{j-2}-S b_{j-1}+T b_j=0
        \qquad(3\le j\le d-1).
\end{equation}
The boundary coefficients in degrees $0,1,2,d,d+1,d+2$ will be used repeatedly below.

First assume $p=2$.  Then $d$ is odd and $\eps=1$.  The boundary coefficients specialize to
\[
        b_2-Sd+T=0,
        \qquad S=d(T+1),
        \qquad b_2T-Sd+1=0.
\]
Since $d=1$ in the ground field, the first two equations give
\[
        b_2=1,
        \qquad T=S+1,
\]
and the third equation is then redundant.  Thus \eqref{eq:quadratic-recurrence} becomes
\[
        b_{j-2}+S b_{j-1}+(S+1)b_j=0.
\]
Here every $b_j$ lies in $\F_2$, and $b_1=b_2=1$.  If no consecutive pair $b_{j-1},b_j$ differs, then all interior coefficients are already equal to $1$.  Otherwise, for some $j$ one has $b_{j-1}\ne b_j$; the displayed recurrence then determines
\[
        S=b_{j-2}+b_j\in\F_2.
\]
If $S=0$, the recurrence is $b_j=b_{j-2}$, so the initial values force all $b_j=1$.  If $S=1$, the recurrence is $b_{j-2}=b_{j-1}$ for every $j$ in the range; shifting the index gives $b_3=b_2=1$, $b_4=b_3=1$, and so on, while $b_{d-1}=d=1$.  Hence again all interior binomial coefficients are $1$.  They are therefore nonzero and geometric, so \cref{lem:binom-geometric} gives $d=2^a-1$, a degree-one case.  This contradiction excludes characteristic $2$.

Now assume $p=3$.  If $\eps=-1$, the boundary coefficients specialize to
\[
        b_2-Sd=0,
        \qquad d(T+1)=0,
        \qquad b_2T-Sd=0.
\]
Since $d\ne0$ in the ground field, the second equation gives $T=-1$; comparing the first and third then gives $b_2=0$, and hence $S=0$.  Thus
\[
        T=-1,
        \qquad b_2=0,
        \qquad S=0.
\]
Then \eqref{eq:quadratic-recurrence} gives $b_j=b_{j-2}$ for all $3\le j\le d-1$.  Since $d$ is odd in this branch, $d-1$ is even, so $b_{d-1}=b_2=0$, contradicting $b_{d-1}=d\ne0$.  If $\eps=1$, the boundary coefficients are
\[
        b_2-Sd+T=0,
        \qquad S=d(T+1),
        \qquad b_2T-Sd+1=0.
\]
If $d=1$ in the ground field, then $b_2=0$, and the first two equations give $T=S$ and $S=T+1$, a contradiction.  Hence $d=-1$ in the ground field; then $b_2=d(d-1)/2=1$, and the second equation gives $S+T=-1$ while the third gives the same relation.  Thus
\[
        d=-1,
        \qquad b_2=1,
        \qquad S+T=-1.
\]
The recurrence then forces $b_j=(-1)^j$ for every $1\le j\le d-1$: if $T\ne0$ this follows by induction from $b_1=-1$ and $b_2=1$, and if $T=0$ the shifted recurrence gives $b_j=-b_{j-1}$ up to $j=d-2$, while $b_{d-1}=d=-1$ supplies the last term.  Hence the interior binomial coefficients are nonzero and geometric, so \cref{lem:binom-geometric} gives $d=3^a-1$, again a degree-one case.  Thus characteristic $3$ gives no degree-two cases.

It remains to consider characteristic zero or characteristic $p\ge5$.  The eliminations used below are those of \cref{lem:degree-two-elimination-details}.  We split according to $\eps$.

Suppose first that $\eps=1$.  Comparing the six boundary coefficients gives
\begin{equation}\label{eq:eps-plus}
        b_2-Sd+T=0,
        \qquad S=d(T+1),
        \qquad b_2T-Sd+1=0.
\end{equation}
If $d=-1$ in the ground field, then \eqref{eq:eps-plus} gives $S+T=-1$, and the recurrence, with initial values $b_1=-1$ and $b_2=1$, forces the interior coefficients to be geometric as above.  This is a degree-one case, contrary to our assumption.  Hence $d+1\ne0$.

Eliminating $S$ and $T$ from \eqref{eq:eps-plus}, using $b_2=d(d-1)/2$, gives
\[
        (d-2)(3d-2)=0
\]
in the ground field.  If $d=2$ in the ground field, then \eqref{eq:eps-plus} gives $S=0$ and $T=-1$.  The recurrence at $j=3$ would then give $b_3=b_1=d=2$, whereas the binomial identity gives
\[
        b_3=\frac{d(d-1)(d-2)}6=0,
\]
a contradiction.  Therefore $3d-2=0$.

On the branch $3d-2=0$, the recurrence coefficient at $j=3$ reduces to
\[
        -\frac{d(d-3)(d+1)}3.
\]
Since $d\ne0$ and $d+1\ne0$, this can vanish only in characteristic $7$.  In characteristic $7$ the same branch gives $d=3$ in the ground field.  Then $d$ cannot be the integer $4$, so $j=4$ lies in the recurrence range.  The recurrence coefficient at $j=4$ reduces on this branch to
\[
        -\frac{d(d-4)(d+1)(5d-6)}{48},
\]
which is nonzero at $d=3$ in characteristic $7$.  This contradiction excludes the branch $\eps=1$.

Finally suppose that $\eps=-1$.  The boundary equations are
\begin{equation}\label{eq:eps-minus}
        b_2-Sd+3T=0,
        \qquad 3S=d(T+1),
        \qquad b_2T-Sd+3=0.
\end{equation}
If $d=3$ in the ground field and $d>3$ as an integer, then \eqref{eq:eps-minus} gives $T=S-1$, and applying \eqref{eq:quadratic-recurrence} at $j=3$ gives $S=1$, hence $T=0$, impossible.  Thus, apart from the already excluded integer $d=3$, we may assume $d-3\ne0$.

For $d-3\ne0$, equations \eqref{eq:eps-minus} imply
\[
        (d+2)(d+6)=0.
\]
After substituting the corresponding value of $T$, the recurrence coefficient at $j=3$ is
\[
        -\frac{d^3+3d^2+2d-12}{6}.
\]
The branch $d=-2$ gives the nonzero value $2$.  On the branch $d=-6$, this coefficient is $22$, so the characteristic must be $11$.  Thus $d\equiv5\pmod{11}$.  If $d=5$, this is the exceptional case already verified.  If $d>5$, then $j=5$ lies in the range of \eqref{eq:quadratic-recurrence}.  At this point the previous residual has already forced characteristic $11$, so $120$ is invertible.  By the cleared $R_5$ identity of \cref{lem:degree-two-elimination-details}, the corresponding recurrence coefficient is
\[
        -\frac{(d-2)(d-1)(3d^3+3d^2-32d-72)}{120},
\]
which becomes $9\in\F_{11}$ after substituting $d\equiv5\pmod{11}$.  This contradiction proves that $d=5$ is the only remaining case.
\end{proof}

\section{Cyclotomic fixed-index defects}\label{sec:cyclotomic}

The torsion-defect formula converts quotient degree into a cyclotomic counting problem whenever the $d$-torsion lies in a finite field.  This section makes that conversion explicit and records quantitative consequences.  These results are not needed for the proof of the Mersenne family, but they show that the full-unit case $d=Q-1$ is the unique bounded-degree source in fixed-index subfield towers.

Let $Q$ be a power of $p$, let $d\mid Q-1$, and put
\[
        e=\frac{Q-1}{d}.
\]
Let $H\le\F_Q^*$ be the subgroup of order $d$, equivalently the subgroup of $e$-th powers.

\begin{proposition}[cyclotomic defect]\label{prop:cyclotomic}
With the notation above,
\[
        \deg h_d=d-N_e(Q),
        \qquad
        N_e(Q)=\#\{x\in H:-1-x\in H\}.
\]
Thus $N_e(Q)$ is the cyclotomic number attached to the pair of cyclotomic classes $(H,-1-H)$.
\end{proposition}

\begin{proof}
The set $\mu_d$ is $H$.  For each $x\in H$, the equation $1+x+y=0$ determines $y=-1-x$.  Therefore the intersection count in \cref{thm:torsion-defect} is exactly $N_e(Q)$.
\end{proof}

Let $\mathcal X_e$ denote the group of multiplicative characters of $\F_Q^*$ that are trivial on $H$, extended to $\F_Q$ by assigning value $0$ at $0$.  This group has order $e$ and contains the extended trivial character $\varepsilon$.

\begin{theorem}[Jacobi-sum defect formula]\label{thm:jacobi-defect}
For $d=(Q-1)/e$,
\[
        N_e(Q)=\frac1{e^2}\sum_{\chi,\psi\in\mathcal X_e}J_-(\chi,\psi),
        \qquad
        J_-(\chi,\psi)=\sum_{x\in\F_Q}\chi(x)\psi(-1-x).
\]
Consequently
\[
        \left|N_e(Q)-\frac{Q-2}{e^2}\right|
        \le \frac{e^2-1}{e^2}\sqrt Q.
\]
\end{theorem}

\begin{proof}
The subgroup indicator is
\[
        1_H(x)=\frac1e\sum_{\chi\in\mathcal X_e}\chi(x)
\]
for every $x\in\F_Q$, because all characters in the sum vanish at $0$.  Hence
\[
        N_e(Q)=\sum_{x\in\F_Q}1_H(x)1_H(-1-x)
        =\frac1{e^2}\sum_{\chi,\psi\in\mathcal X_e}J_-(\chi,\psi).
\]
The main term is $J_-(\varepsilon,\varepsilon)=Q-2$.  If exactly one of $\chi,
\psi$ is trivial, the corresponding sum has absolute value at most $1$.  If both are nontrivial and $\chi\psi=\varepsilon$, the corresponding Jacobi sum has absolute value $1$.  If $\chi$, $\psi$, and $\chi\psi$ are all nontrivial, the standard Jacobi-sum bound gives absolute value $\sqrt Q$; see \cite[Chapter 5]{BerndtEvansWilliams}, \cite[Chapter 5]{LidlNiederreiter}, or Storer's classical cyclotomy reference \cite{Storer}.  Since $1\le\sqrt Q$, every non-main term has absolute value at most $\sqrt Q$, and there are $e^2-1$ such terms.  This proves the estimate.
\end{proof}

\begin{corollary}[indices one and two]\label{cor:index12}
For $e=1$ one has
\[
        N_1(Q)=Q-2,
        \qquad
        \deg h_{Q-1}=1.
\]
For $e=2$ and $Q$ odd, let $\chi$ be the quadratic character of $\F_Q$.  Then
\[
        N_2(Q)=\frac14\bigl(Q-2-3\chi(-1)\bigr),
\]
that is
\[
        N_2(Q)=
        \begin{cases}
        (Q-5)/4,& Q\equiv1\pmod4,\\[3pt]
        (Q+1)/4,& Q\equiv3\pmod4,
        \end{cases}
\]
and hence
\[
        \deg h_{(Q-1)/2}=
        \begin{cases}
        (Q+3)/4,& Q\equiv1\pmod4,\\[3pt]
        (Q-3)/4,& Q\equiv3\pmod4.
        \end{cases}
\]
\end{corollary}

\begin{proof}
The case $e=1$ is immediate: $H=\F_Q^*$, and the only excluded $x$ are $0$ and $-1$.

For $e=2$, the subgroup $H$ is the square subgroup and
\[
        1_H(x)=\frac{\varepsilon(x)+\chi(x)}2,
\]
where $\varepsilon$ is the trivial multiplicative character extended by $\varepsilon(0)=0$.  Hence
\[
        N_2(Q)=\frac14\sum_{x\in\F_Q}
        (\varepsilon(x)+\chi(x))(\varepsilon(-1-x)+\chi(-1-x)).
\]
The four sums are
\[
        \sum_x\varepsilon(x)\varepsilon(-1-x)=Q-2,
\]
\[
        \sum_x\varepsilon(x)\chi(-1-x)=\sum_x\chi(-1-x)-\chi(-1)= -\chi(-1),
\]
\[
        \sum_x\chi(x)\varepsilon(-1-x)=\sum_x\chi(x)-\chi(-1)= -\chi(-1),
\]
and
\[
        \sum_x\chi(x)\chi(-1-x)=\chi(-1)\sum_x\chi(x(1+x))=-\chi(-1),
\]
because $x(1+x)$ is a nonconstant quadratic polynomial with distinct roots.  Adding gives the formula for $N_2(Q)$, and the degree formula follows from \cref{prop:cyclotomic}.
\end{proof}

\begin{corollary}[low-degree quotients in the index-two tower]\label{cor:index2-low}
Let $Q$ be odd and $d=(Q-1)/2$.  Then $\deg h_d=0$ occurs exactly at $Q=3$ $(d=1)$; $\deg h_d=1$ occurs exactly at $Q=7$ $(d=3)$; and
\[
        \deg h_d=2
        \quad\Longleftrightarrow\quad
        (Q,d)=(5,2)\text{ or }(11,5).
\]
\end{corollary}

\begin{proof}
This is immediate from \cref{cor:index12}.  If $Q\equiv1\pmod4$, then $\deg h_d=(Q+3)/4$, which equals $2$ only for $Q=5$ and never equals $0$ or $1$ for an odd prime power.  If $Q\equiv3\pmod4$, then $\deg h_d=(Q-3)/4$, which equals $0,1,2$ for $Q=3,7,11$, respectively.
\end{proof}

\begin{proposition}[index three]\label{prop:index3}
Assume $3\mid Q-1$, and let $\chi$ be a multiplicative character of order $3$ of $\F_Q^*$.  Let $J(\alpha,\beta)=\sum_{t\in\F_Q}\alpha(t)\beta(1-t)$ be the Jacobi sum, with characters extended by $0$ at $0$.  For $d=(Q-1)/3$,
\[
        N_3(Q)=\frac{Q-8+J(\chi,\chi)+J(\chi^2,\chi^2)}{9}.
\]
In particular,
\[
        \left|N_3(Q)-\frac{Q-8}{9}\right|\le \frac{2}{9}\sqrt Q.
\]
\end{proposition}

\begin{proof}
The indicator of the subgroup of cubes is
\[
        I(x)=\frac{\varepsilon(x)+\chi(x)+\chi^2(x)}{3},
\]
with all characters extended by zero at $0$.  Thus
\[
        N_3(Q)=\frac19\sum_{i,j=0}^2\sum_x\chi^i(x)\chi^j(-1-x).
\]
Substitute $x=-t$.  Since $\chi(-1)=1$ for cubic characters, the inner sums are Jacobi sums $J(\chi^i,\chi^j)$.  The $(i,j)=(0,0)$ term is $Q-2$.  The four terms with exactly one of $i,j$ equal to $0$ are all $-1$.  The two mixed nontrivial terms $J(\chi,\chi^2)$ and $J(\chi^2,\chi)$ are also $-1$.  The remaining terms are $J(\chi,\chi)$ and $J(\chi^2,\chi^2)$.  This gives the formula.  The bound follows from $|J(\alpha,\beta)|=\sqrt Q$ when $\alpha$, $\beta$, and $\alpha\beta$ are nontrivial.
\end{proof}

\begin{theorem}[effective fixed-index obstruction]\label{thm:fixed-index}
Fix $e\ge2$ and let $Q$ vary over prime powers with $e\mid Q-1$.  Set $d=(Q-1)/e$.  Then
\[
        \deg h_d
        \ge
        \frac{(e-1)Q+(2-e)-(e^2-1)\sqrt Q}{e^2}.
\]
In particular, $\deg h_{(Q-1)/e}\to\infty$ as $Q\to\infty$ for every fixed proper index $e\ge2$.

More precisely, if $\deg h_d\le M$, then, writing $T=\sqrt Q$,
\[
        T\le
        \frac{(e^2-1)+
        \sqrt{(e^2-1)^2+4(e-1)(e^2M+e-2)}}
        {2(e-1)}.
\]
Thus bounded-degree quotients in fixed-index proper subgroup towers occur only for explicitly bounded $Q$.
\end{theorem}

\begin{proof}
By \cref{prop:cyclotomic,thm:jacobi-defect},
\[
\begin{aligned}
        \deg h_d
        &=\frac{Q-1}{e}-N_e(Q) \\
        &\ge \frac{Q-1}{e}-\frac{Q-2}{e^2}-\frac{e^2-1}{e^2}\sqrt Q \\
        &=\frac{(e-1)Q+(2-e)-(e^2-1)\sqrt Q}{e^2}.
\end{aligned}
\]
If $\deg h_d\le M$, then
\[
        (e-1)T^2-(e^2-1)T+(2-e-e^2M)\le0.
\]
Solving this quadratic inequality for $T\ge0$ gives the displayed bound.
\end{proof}

\section{Arithmetic consequence: the Mersenne trace-zero family}\label{sec:finite-field}

The preceding rigidity theorem identifies the maximal torsion-defect stratum of the quotient-map family.  This section extracts its trace-zero arithmetic consequence: in characteristic two, the Mersenne exponents $d=2^a-1$ give a new permutation family on the trace-zero plane, subject exactly to the fiber and denominator conditions below.  This is not a classification of all exponents $d$ for which $P_d$ may permute $\Gamma_q$.

We now return to $L=\F_{q^3}$, $K=\F_q$, $q=p^k$.

\begin{lemma}[denominator criteria]\label{lem:denom}
Let $d\ge1$.
\begin{enumerate}[label=\textup{(\roman*)}]
\item For $d=1$, the map $P_1(x)=x^q-x$ permutes $\Gamma_q$ if and only if $p\ne3$.
\item For $d=3$, the condition $\mu_3\cap\Lambda_q=\varnothing$ holds if and only if $q\equiv2\pmod3$.
\item Let $q=2^k$ with $k\ge1$, let $a\ge1$, let $d=2^a-1$, and put $g=\gcd(a,k)$.  Then
\[
        \mu_d\cap\Lambda_q=\varnothing
        \quad\Longleftrightarrow\quad
        g\text{ is odd and }3\nmid \frac ag.
\]
In particular, under the fiber condition $\gcd(a,k)=1$, this is equivalent to $3\nmid a$.
\end{enumerate}
\end{lemma}

\begin{proof}
For (i), $P_1=\sigma-1$.  Its kernel on $\Gamma_q$ is $K\cap\Gamma_q$, which is the kernel of multiplication by $3$ on $K$.  Hence the kernel is zero if and only if $p\ne3$.

For (ii), if $q\equiv0\pmod3$, then $1\in\Lambda_q\cap\mu_3$.  If $q\equiv1\pmod3$, a primitive cube root $\omega\in K$ satisfies $\omega^{q+1}+\omega+1=\omega^2+\omega+1=0$, so $\omega\in\Lambda_q\cap\mu_3$.  If $q\equiv2\pmod3$ and $\lambda\in\Lambda_q\cap\mu_3$, then $\lambda^{q+1}=1$, so $1+\lambda+1=0$, hence $\lambda=-2$.  This is incompatible with $\lambda^3=1$ unless $p=3$, which is excluded by $q\equiv2\pmod3$.

For (iii), $\mu_d=\F_{2^a}^*$.  Let $A:z\mapsto z^{2^k}$ on $\F_{2^a}$ and let
\[
        \tau(z)=1+z^{-1}
\]
in characteristic two.  The maps $A$ and $\tau$ commute, and $\tau^3=1$ on $\PP^1$.  The automorphism $A$ has order
\[
        n=\frac a{\gcd(a,k)}=\frac ag
\]
on $\F_{2^a}$.  A denominator point is exactly a solution of
\[
        A(z)=\tau(z),\qquad z\in\F_{2^a}^*.
\]

Assume first that $3\nmid n$.  If $A(z)=\tau(z)$, then
\[
        z=A^n(z)=\tau^n(z),
\]
so $z$ is fixed by $\tau$ or by $\tau^2$.  In characteristic two the two fixed-point equations are the same:
\[
        z^2+z+1=0.
\]
Thus any solution lies in $\F_4\setminus\F_2$.  Such points lie in $\F_{2^a}$ if and only if $a$ is even, and on them $\tau$ acts trivially.  Hence the equation $A(z)=\tau(z)$ is solvable in this case if and only if $A$ also acts trivially on $\F_4$, i.e. if and only if $k$ is even.  Therefore, when $3\nmid n$, denominator points exist exactly when $a$ and $k$ are both even, equivalently exactly when $g$ is even.

Now assume $3\mid n$.  Put $Q=2^g$ and write $a=gn$, $k=gk'$ with $\gcd(k',n)=1$.  Then $3\nmid k'$ and $\F_{Q^3}\subseteq\F_{2^a}$.  On $\F_{Q^3}$ the map $A$ is the $k'$-th power of the $Q$-Frobenius.  If $k'\equiv1\pmod3$, choose a root of
\[
        z^Q=\tau(z),
        \qquad\text{equivalently}\qquad z^{Q+1}+z+1=0.
\]
This polynomial has degree $Q+1$, is squarefree, and every root lies in $\F_{Q^3}$ by iterating the equation three times.  Thus a nonzero solution exists in $\F_{Q^3}\subseteq\F_{2^a}$, and it satisfies $A(z)=\tau(z)$.  If $k'\equiv2\pmod3$, choose instead a root of
\[
        z^Q=\tau^2(z),
        \qquad\text{equivalently}\qquad z^{Q+1}+z^Q+1=0.
\]
The same squarefreeness and iteration argument gives $Q+1$ roots in $\F_{Q^3}$, and for such a root $A(z)=z^{Q^2}=\tau(z)$.  Hence denominator points exist whenever $3\mid n$.

Combining the two cases, $\mu_d\cap\Lambda_q$ is empty exactly when $3\nmid n$ and $g$ is odd, as claimed.
\end{proof}

\begin{theorem}[degree-one trace-zero classification]\label{thm:trace-zero-classification}
Let $q=p^k$ and let $d\ge1$.  Remove the $p$-power part of $d$ by writing $d=p^sd_0$ with $p\nmid d_0$.  In the degree-one quotient sector, i.e. when the reduced quotient $h_{d_0}$ has morphism degree at most one, the map
\[
        P_d(X)=X^{dq}-X^d
\]
permutes $\Gamma_q$ precisely in the following cases:
\begin{enumerate}[label=\textup{(\roman*)}]
\item $d_0=1$ and $p\ne3$;
\item $d_0=3$ and $q\equiv2\pmod3$;
\item $p=2$, $d_0=2^a-1$ with $a\ge1$, $\gcd(a,k)=1$, and $3\nmid a$.
\end{enumerate}
The list is not disjoint: in characteristic two, $d_0=1$ and $d_0=3$ are the Mersenne cases $a=1$ and $a=2$.
\end{theorem}

\begin{proof}
By \cref{prop:p-power-reduction}, the $p$-power part of $d$ does not affect the permutation property.  Thus assume $p\nmid d$.  By \cref{thm:degree-one}, the only degree-one quotient cases are $d=1$, $d=3$, and $d=p^a-1$.  The descent criterion \cref{prop:mult-desc} requires the fiber condition $\gcd(d,q-1)=1$, the denominator condition $\mu_d\cap\Lambda_q=\varnothing$, and quotient bijectivity.

For $d=1$, this is exactly \cref{lem:denom}(i).  For $d=3$, the standing hypothesis gives $p\ne3$.  If $P_3$ permutes $\Gamma_q$, then the denominator condition in \cref{prop:mult-desc} holds, and \cref{lem:denom}(ii) gives $q\equiv2\pmod3$.  Conversely, if $q\equiv2\pmod3$, then $\gcd(3,q-1)=1$, and \cref{lem:denom}(ii) gives $\mu_3\cap\Lambda_q=\varnothing$.  The reduced quotient is the degree-one map
\[
        h_3(z)=-\frac{2z+1}{z-1}
\]
from the proof of \cref{thm:degree-one}; it is defined over $\F_p$ and commutes with $\tau$ by \cref{prop:tau-equivariance}.  Therefore, for every $z\in\Lambda_q$,
\[
        h_3(z)^q=h_3(z^q)=h_3(\tau z)=\tau(h_3(z)),
\]
so $h_3$ preserves $\Lambda_q$.  Being an automorphism of $\PP^1$, it restricts to a bijection of this finite set.  Thus all three conditions of \cref{prop:mult-desc} hold exactly when $q\equiv2\pmod3$.

For $d=p^a-1$, the fiber condition $\gcd(d,q-1)=1$ is restrictive.  If $p$ is odd, then $p^a-1$ and $p^k-1$ are both even, so it fails.  Hence $p=2$.  In that case
\[
        \gcd(2^a-1,2^k-1)=2^{\gcd(a,k)}-1,
\]
so the fiber condition is $\gcd(a,k)=1$.  Under this condition, \cref{lem:denom}(iii) gives the denominator condition $3\nmid a$.  The reduced quotient map is $h_d=\tau^2$, an automorphism of the twisted quotient set, so the quotient condition is automatic.  This proves the classification.
\end{proof}

\begin{corollary}[Mersenne trace-zero permutations]\label{cor:mersenne}
Let $q=2^k$ with $k\ge1$, let $a\ge1$, and put $d=2^a-1$.  Then
\[
        X^{dq}-X^d
\]
permutes $\Gamma_q=\ker\Tr_{\F_{q^3}/\F_q}$ if and only if
\[
        \gcd(a,k)=1
        \qquad\text{and}\qquad
        3\nmid a.
\]
\end{corollary}

\begin{corollary}[inverse via a linear Hilbert--90 lift on the trace-zero plane]\label{cor:quotient-inverse}
Assume the hypotheses of \cref{cor:mersenne}.  Let $y\in\Gamma_q$ and let $P_d(x)=x^{dq}-x^d$.  The inverse of $P_d$ on $\Gamma_q$ is obtained as follows.
\begin{enumerate}[label=\textup{(\roman*)}]
\item If $y=0$, then $x=0$.
\item If $y\ne0$, compute
\[
        w_y=y^{q-1}\in\Lambda_q,
        \qquad
        z=\tau(w_y).
\]
\item Find a nonzero solution $u\in L$ of the $K$-linear equation
\[
        \sigma(u)=zu,
        \qquad\text{equivalently}\qquad u^q=zu.
\]
Then $u\in\Gamma_q^*$ and $u^{q-1}=z$.
\item Put $x=cu$, where $c=1$ if $q=2$, and if $q>2$ then
\[
        c=\left(\frac{y}{P_d(u)}\right)^e\in\F_q^*,
        \qquad
        ed\equiv1\pmod{q-1}.
\]
\end{enumerate}
Then $P_d(x)=y$.  The expression $y/P_d(u)$ lies in $\F_q^*$, and the result is independent of the chosen lift $u$.
\end{corollary}

\begin{proof}
Under the Mersenne hypotheses, \cref{thm:degree-one} gives $h_d=\tau^2$, and $\tau$ has order three.  Hence the inverse quotient map is $\tau$.  Thus if $w_y=y^{q-1}$ and $z=\tau(w_y)$, a lift over the desired quotient point is obtained by solving the $K$-linear equation $\sigma(u)=zu$.  Such a nonzero solution exists by multiplicative Hilbert 90, equivalently because $z\in\Lambda_q$ has norm one; the solution space is a one-dimensional $K$-vector space.  For any nonzero solution,
\[
        u+\sigma(u)+\sigma^2(u)=u(1+z+zz^q)=0,
\]
so $u\in\Gamma_q^*$ and $u^{q-1}=z$.  Since $\lambda(P_d(u))=h_d(z)=w_y=\lambda(y)$, the ratio $y/P_d(u)$ lies in $K^*=\F_q^*$.  If $q=2$, then $K^*$ is trivial, so this ratio is $1$ and $c=1$.  If $q>2$, multiplication by $c\in K^*$ on the $K^*$-fiber above $z$ changes $P_d(u)$ by the factor $c^d$.  Because $\gcd(d,q-1)=1$, exponentiation by $d$ is invertible on $K^*$, with any inverse exponent $e$ satisfying $ed\equiv1\pmod{q-1}$.  This gives the stated formula.

It remains only to check independence of the Hilbert--90 lift.  If $u'=tu$ with $t\in K^*$, then $P_d(u')=t^dP_d(u)$.  For $q>2$ the corresponding scalar is
\[
        c'=\left(\frac{y}{P_d(u')}\right)^e
           =t^{-de}c=t^{-1}c,
\]
because $ed\equiv1\pmod{q-1}$ on $K^*$.  Hence $c'u'=cu$.  For $q=2$ there is no nontrivial choice of $t\in K^*$, so the same conclusion is automatic.
\end{proof}

\begin{example}[small Mersenne exponents]
For $q=2^k$, the Mersenne criterion gives the following initial cases.
\[
\begin{array}{c|c|c|c}
 a & d=2^a-1 & \text{permutation condition on }k & \text{comment}\\
\hline
1&1& \text{all }k & \sigma-1\text{ on }\Gamma_q\\
2&3& k\text{ odd} & \text{DSX cubic case in characteristic }2\\
3&7& \text{none} & 3\mid a\\
4&15& k\text{ odd} & \text{all odd }k\\
5&31& 5\nmid k & \text{Mersenne branch}\\
6&63& \text{none} & 3\mid a
\end{array}
\]
In particular, for $a=4$ the map $X^{15q}-X^{15}$ permutes the trace-zero plane over $\F_{2^k}$ exactly when $k$ is odd.
\end{example}

\begin{example}[small quotient degrees]
The following sample values illustrate the first quotient-degree strata.
\[
\begin{array}{c|c|c|c}
\text{characteristic} & d & \deg h_d & \text{source}\\
\hline
p\ne3 & 1 & 1 & \text{linear case}\\
3 & 1 & 0 & h_1\text{ is constant}\\
p=0\text{ or }p\ne2,3 & 2 & 2 & \text{basic degree-two case}\\
p\ne3 & 3 & 1 & \text{DSX cubic quotient}\\
p>0 & p^a-1 & 1 & \text{Mersenne quotient}\\
11 & 5 & 2 & \text{exceptional degree-two case}
\end{array}
\]
The last line is supplied by the index-two tower over $\F_{11}$ and is also the exceptional case in \cref{thm:degree-two}.
\end{example}

\section{The second sparse mechanism: the \texorpdfstring{$p^a+1$}{p to the a plus 1} branch}\label{sec:frobenius-sparse}

The degree-one theorem identifies the branch $d=p^a-1$.  There is a second sparse branch, $d=p^a+1$ with $a\ge1$, which is invisible if one only asks for degree-one quotients.  It is the branch responsible for cyclic, rather than symmetric, monodromy.  This section proves the precise normal forms and the corresponding twisted-exceptionality criteria.

Throughout the first part of the section we assume $\operatorname{char} k\ne3$, choose a primitive cube root of unity $\omega$, and use the coordinate
\[
        M(u)=\frac{\omega-\omega^2u}{1-u},
        \qquad
        M^{-1}(z)=\frac{z-\omega}{z-\omega^2}.
\]
Then $M(\omega u)=\tau(M(u))$.  For a $\tau$-equivariant map $f$ we write
\[
        \Phi_f=M^{-1}\circ f\circ M.
\]
For $f=h_d$ we abbreviate $\Phi_d=M^{-1}\circ h_d\circ M$.

\begin{proposition}[root-of-unity filtered normal form]\label{prop:filtered-normal-form}
Assume $\operatorname{char} k\ne3$.  With
\[
        M(u)=\frac{\omega-\omega^2u}{1-u},
        \qquad
        M^{-1}(z)=\frac{z-\omega}{z-\omega^2},
        \qquad
        \Phi_d=M^{-1}\circ h_d\circ M,
\]
for $r\in\{0,1,2\}$ define
\[
        A_{d,r}(T)=\sum_{\substack{j\ge0\\ r+3j\le d}}
        (-1)^{r+3j}\binom d{r+3j}T^j.
\]
Let $r_+=\langle 2-d\rangle_3$ and $r_- =\langle 1-d\rangle_3$ be the least nonnegative residues modulo $3$.  Then
\[
        \Phi_d(u)
        =\omega^2 u^{r_+-r_-}\frac{A_{d,r_+}(u^3)}{A_{d,r_-}(u^3)}.
\]
In particular $\Phi_d(u)=uR_d(u^3)$ for some $R_d\in k(T)$.
\end{proposition}

\begin{proof}
Put
\[
        A=(1-u)^d,
        B=(1-\omega u)^d,
        C=(1-\omega^2u)^d.
\]
Since
\[
        M(u)=\omega\frac{1-\omega u}{1-u},
        \qquad
        M(\omega u)=\omega\frac{1-\omega^2u}{1-\omega u},
\]
we have
\[
        H_d^{\rm raw}(M(u))
        =\omega^d\frac{\omega^dC-B}{\omega^dB-A}.
\]
Applying $M^{-1}(z)=(z-\omega)/(z-\omega^2)$ gives
\[
        \Phi_d(u)
        =\frac{\omega A+\omega^{d+2}B+\omega^{2d}C}
        {\omega^2A+\omega^{d+1}B+\omega^{2d}C}.
\]
The coefficient of $u^j$ in the numerator is
\[
        (-1)^j\binom dj
        \bigl(\omega+\omega^{d+2+j}+\omega^{2d+2j}\bigr).
\]
The root-of-unity factor in parentheses is nonzero precisely when
\[
        j\equiv 2-d\pmod3,
\]
and in that case it equals $3\omega$.  It is zero for the other two residue classes.  Thus the numerator is the corresponding root-of-unity projection of the binomial expansion; in positive characteristic some projected terms may still vanish because the binomial coefficient itself vanishes in $k$.  Similarly, the root-of-unity factor in the denominator is nonzero precisely when
\[
        j\equiv1-d\pmod3,
\]
and in that case it equals $3\omega^2$, with the same possible further vanishing of binomial coefficients.  Dividing the two filtered sums gives the displayed formula.
\end{proof}

\begin{theorem}[Frobenius-sparse normal forms]\label{thm:frob-sparse-normal-forms}
Let $k$ be algebraically closed of characteristic $p>0$, and let $a\ge1$.
\begin{enumerate}[label=\textup{(\roman*)}]
\item If $d=p^a-1$, then
\[
        h_d=\tau^2.
\]
\item Assume $p\ne3$, put $r=p^a$, and let $d=r+1$.  In the coordinate $u=M^{-1}(z)$ one has
\[
        \Phi_{r+1}(u)=
        \begin{cases}
        \omega^2u^{-(r+1)},& r\equiv1\pmod3,\\[3pt]
        \omega^2u^{r-1},& r\equiv2\pmod3.
        \end{cases}
\]
\item Assume $p=3$, put $r=3^a$, and let $d=r+1$.  In the coordinate
\[
        u=\frac1{z-1},
        \qquad z=1+\frac1u,
\]
which conjugates $\tau$ to $u\mapsto u+1$, one has
\[
        \Phi_{r+1}(u)=-(u^r+u+1).
\]
\end{enumerate}
\end{theorem}

\begin{proof}
Part (i) is the final computation in \cref{thm:degree-one}.

For (ii), since $d=r+1$, the Frobenius identity gives
\[
        H_{r+1}^{\rm raw}(z)=\frac{z^r+z+1}{z^{r+1}-1}.
\]
If $r\equiv1\pmod3$, then $M(u)^r=M(u^r)$; if $r\equiv2\pmod3$, then $M(u)^r=M(u^{-r})$.  For arbitrary $u,v$ a direct substitution gives
\[
        M^{-1}\left(
        \frac{M(v)+M(u)+1}{M(u)M(v)-1}
        \right)=\frac{\omega^2}{uv}.
\]
Taking $v=u^r$ in the first case and $v=u^{-r}$ in the second gives the two formulas.

For (iii), in characteristic $3$ the coordinate $u=1/(z-1)$ satisfies
\[
        u(\tau z)=u(z)+1.
\]
Again
\[
        H_{r+1}^{\rm raw}(z)=\frac{z^r+z+1}{z^{r+1}-1}.
\]
Writing $z=1+u^{-1}$ and using $r=3^a$, we get
\[
        z^r=1+u^{-r},
\]
so
\[
        H_{r+1}^{\rm raw}(z)
        =\frac{u^{-r}+u^{-1}}{u^{-1}+u^{-r}+u^{-r-1}}
        =\frac{u^r+u}{u^r+u+1}.
\]
Therefore
\[
        \Phi_{r+1}(u)=\frac1{H_{r+1}^{\rm raw}(z)-1}=-(u^r+u+1).
\]
\end{proof}

\begin{corollary}[monodromy of the Frobenius-sparse branch]\label{cor:sparse-monodromy}
Let $a\ge1$ and put $d=p^a+1$.
\begin{enumerate}[label=\textup{(\roman*)}]
\item If $p\ne3$, then the reduced quotient $h_d$ has cyclic geometric monodromy.  More precisely, with
\[
        m=
        \begin{cases}
        -(p^a+1),&p^a\equiv1\pmod3,\\
        p^a-1,&p^a\equiv2\pmod3,
        \end{cases}
\]
its geometric monodromy is cyclic of order $|m|$, except in the degree-one overlap $|m|=1$.
\item If $p=3$, then $h_{3^a+1}$ has elementary abelian geometric monodromy of order $3^a$.
\end{enumerate}
In particular these quotients are not covered by any theorem asserting $2$-transitive monodromy for all non-degree-one $h_d$.
\end{corollary}

\begin{proof}
For $p\ne3$, \cref{thm:frob-sparse-normal-forms} conjugates the map to $u\mapsto \omega^2u^m$.  Since $a\ge1$, the integer $m$ is prime to $p$, so this is a separable Kummer cover of degree $|m|$ and is Galois with cyclic group generated by multiplication of $u$ by an $|m|$-th root of unity.

For $p=3$, the equation $t=-(u^{3^a}+u+1)$ is equivalent to
\[
        u^{3^a}+u=-t-1.
\]
The additive polynomial $X^{3^a}+X$ is separable and has $3^a$ roots.  Translation by any root preserves the cover, and these translations act simply transitively on a geometric generic fiber.  Hence the geometric monodromy group is the elementary abelian group $\ker(X^{3^a}+X)$.
\end{proof}

\begin{lemma}[twisted sets in sparse coordinates]\label{lem:twisted-sets-normal}
Let $q=p^k$.
\begin{enumerate}[label=\textup{(\roman*)}]
\item If $p\ne3$ and $u=M^{-1}(z)$, then
\[
        M^{-1}(\Lambda_q)=
        \begin{cases}
        \{0,\infty\}\cup\{u\in\barF_p^*:u^{q-1}=\omega\},&q\equiv1\pmod3,\\[3pt]
        \{u\in\barF_p^*:u^{q+1}=\omega^2\},&q\equiv2\pmod3.
        \end{cases}
\]
\item If $p=3$ and $u=1/(z-1)$, then
\[
        u(\Lambda_q)=\{\infty\}\cup\{u\in\barF_3:u^q-u=1\}.
\]
\end{enumerate}
\end{lemma}

\begin{proof}
For $p\ne3$, if $q\equiv1\pmod3$, then Frobenius fixes $\omega$ and
\[
        M(u)^q=M(u^q).
\]
The equation $z^q=\tau(z)$ becomes $M(u^q)=M(\omega u)$, hence $u^q=\omega u$ on $\PP^1$.  This gives the first line.  If $q\equiv2\pmod3$, then Frobenius swaps $\omega$ and $\omega^2$ and
\[
        M(u)^q=M(u^{-q}).
\]
Thus $M(u^{-q})=M(\omega u)$, equivalently $u^{q+1}=\omega^2$; neither $0$ nor $\infty$ satisfies this equation.

For $p=3$, the coordinate conjugates $\tau$ to $u\mapsto u+1$ and commutes with $q$-Frobenius.  Hence $z^q=\tau(z)$ becomes $u^q=u+1$ for finite $u$.  The point $u=\infty$ is fixed by both Frobenius and translation, and corresponds to the fixed point $z=1$ of $\tau$.
\end{proof}

\begin{theorem}[twisted exceptionality of the $p^a+1$ branch]\label{thm:plus-branch-exceptionality}
Let $a\ge1$ and put $d=p^a+1$.  This is a quotient-level statement about the action of $h_d$ on $\Lambda_q$; trace-zero permutation of $P_d$ additionally requires the fiber and denominator conditions in \cref{prop:mult-desc}.
\begin{enumerate}[label=\textup{(\roman*)}]
\item If $p=2$, put
\[
        e_a=2^a+(-1)^a.
\]
Then $h_d$ induces a permutation of $\Lambda_{2^k}$ if and only if
\[
        \gcd\bigl(e_a,2^k-(-1)^k\bigr)=1.
\]
In particular $h_d$ is $\tau$-twisted exceptional; indeed the displayed condition holds for every odd $k$.
\item If $p=3$, then $h_{3^a+1}$ induces a permutation of $\Lambda_{3^k}$ if and only if
\[
        \frac{k}{\gcd(a,k)}
\]
 is odd.  In particular $h_{3^a+1}$ is $\tau$-twisted exceptional.
\item If $p$ is odd and $p\ne3$, then $h_{p^a+1}$ induces no permutation of any $\Lambda_{p^k}$.
\end{enumerate}
\end{theorem}

\begin{proof}
Assume first that $p\ne3$.  By \cref{thm:frob-sparse-normal-forms}, the map is conjugate to $u\mapsto \omega^2u^m$, with
\[
        m=-(p^a+1)\quad\text{if }p^a\equiv1\pmod3,
        \qquad
        m=p^a-1\quad\text{if }p^a\equiv2\pmod3.
\]
The congruence $m\equiv1\pmod3$ holds in both cases.  Therefore the monomial preserves the cosets described in \cref{lem:twisted-sets-normal}.  On the nonzero coset it is bijective exactly when $m$ is invertible modulo the size of the underlying cyclic group.  Thus, if $q\equiv1\pmod3$, the condition is $\gcd(|m|,q-1)=1$, while if $q\equiv2\pmod3$, the condition is $\gcd(|m|,q+1)=1$.

For $p=2$ this is precisely the stated criterion, since $|m|=2^a+1$ when $a$ is even and $|m|=2^a-1$ when $a$ is odd, while $q=2^k$ satisfies $q\equiv1\pmod3$ for $k$ even and $q\equiv2\pmod3$ for $k$ odd.  If $k$ is odd and $a$ is even, a common prime divisor of $2^a+1$ and $2^k+1$ would have the order of $2$ dividing both $2a$ and $2k$ but neither $a$ nor $k$; this is impossible because $a$ is even and $k$ is odd.  If $k$ is odd and $a$ is odd, a common divisor of $2^a-1$ and $2^k+1$ would have odd order dividing $a$ and also dividing $2k$ but not $k$, again impossible.  Hence every odd $k$ satisfies the criterion.

If $p$ is odd and $p\ne3$, then $|m|$ is even.  But $q-1$ is even when $q\equiv1\pmod3$, and $q+1$ is even when $q\equiv2\pmod3$.  The required gcd is therefore never $1$.

It remains to treat $p=3$.  By \cref{thm:frob-sparse-normal-forms,lem:twisted-sets-normal}, the finite part of $\Lambda_{3^k}$ is an affine coset $u_0+\F_{3^k}$ with $u_0^{3^k}-u_0=1$, and the point $\infty$ is fixed.  Since $h_{3^a+1}$ is defined over $\F_3$ and commutes with $\tau$, it preserves $\Lambda_{3^k}$; equivalently, the following affine formula is a self-map of the coset $u_0+\F_{3^k}$.  On this coset,
\[
        u_0+x\longmapsto -(u_0^{3^a}+u_0+1)-(x^{3^a}+x),
        \qquad x\in\F_{3^k}.
\]
Thus the induced map is bijective if and only if the $\F_3$-linear part
\[
        x\longmapsto x^{3^a}+x
\]
has zero kernel on $\F_{3^k}$.  A nonzero kernel element satisfies $x^{3^a-1}=-1$.  Let $g=\gcd(a,k)$.  Since
\[
        \gcd(3^a-1,3^k-1)=3^g-1,
\]
the equation $x^{3^a-1}=-1$ is solvable in $\F_{3^k}^*$ if and only if $3^g-1$ divides $(3^k-1)/2$.  This happens if and only if
\[
        \frac{3^k-1}{3^g-1}=1+3^g+\cdots+3^{k-g}
\]
 is even, equivalently if and only if $k/g$ is even.  Therefore bijectivity holds exactly when $k/g$ is odd.
\end{proof}

\begin{corollary}[the $Q+1$ trace-zero branch revisited]\label{cor:Qplus1-branch}
Let $q=2^k$, let $a\ge1$, let $Q=2^a$, and put $d=Q+1=2^a+1$.  Then
\[
        P_d(X)=X^{dq}-X^d
\]
permutes $\Gamma_q$ if and only if
\[
        \gcd(d,q-1)=1.
\]
This recovers the Ding--Song--Xiong $Q+1$ trace-zero theorem from the cyclic quotient normal form.
\end{corollary}

\begin{proof}
The fiber condition in \cref{prop:mult-desc} is exactly $\gcd(d,q-1)=1$, so it is necessary.  Assume it holds.

If $a$ is even, then \cref{thm:frob-sparse-normal-forms} gives a monomial of degree $d$, so the torsion defect is zero and hence $\mu_d\cap\Lambda_q=\varnothing$.  The quotient condition follows from \cref{thm:plus-branch-exceptionality}: for $k$ even it is the same gcd condition, and for $k$ odd the gcd $\gcd(2^a+1,2^k+1)$ is automatically $1$ by the order argument used in the proof of \cref{thm:plus-branch-exceptionality}.

If $a$ is odd, then $d$ is divisible by $3$.  The fiber condition forces $k$ to be odd, since $2^k-1$ is divisible by $3$ for $k$ even.  For odd $k$, \cref{thm:plus-branch-exceptionality} gives the quotient condition.  The normal form in \cref{thm:frob-sparse-normal-forms} has degree $2^a-1=d-2$, so the cancellation degree is $2$.  The primitive cube roots, equivalently the two roots of $z^2+z+1=0$, lie in $\mu_d$ and are fixed by $\tau$; hence they are canceled denominator points.  Conversely, if $z\in\mu_d\cap\Lambda_q$, then $z^q=\tau(z)$ and therefore $\tau(z)\in\mu_d$, so $z$ is a canceled denominator point.  Since there are exactly two cancellations, these are all the possible denominator points.  In characteristic two they are the two nontrivial elements of $\F_4^*$.  For odd $k$ Frobenius interchanges them, so they do not satisfy $z^q=\tau(z)=z$.  Thus $\mu_d\cap\Lambda_q=\varnothing$.  The descent criterion now proves sufficiency.
\end{proof}

\begin{remark}
The branch $d=p^a-1$ with $a\ge1$ is the linear Frobenius-sparse branch; in characteristic two it gives the Mersenne trace-zero permutations of \cref{cor:mersenne}.  The branch $d=p^a+1$ with $a\ge1$ is the cyclic or Artin--Schreier Frobenius-sparse branch; in characteristic two it recovers the known $Q+1$ trace-zero branch of Ding--Song--Xiong.  Outside these sparse branches, the expected behavior is generic monodromy rather than exceptional monodromy.
\end{remark}

\section*{Part III. Branch geometry and symmetric monodromy}
\section{Characteristic-zero quotient degrees}\label{sec:charzero-degree}

The torsion-defect formula immediately gives a complete degree calculation in characteristic zero.  This short result is placed here, before the monodromy theorem, because it identifies the precise separable degree of every non-linear characteristic-zero quotient.

\begin{theorem}[complete quotient-degree strata in characteristic zero]\label{thm:charzero-degree-strata}
Over an algebraically closed field of characteristic zero,
\[
        \deg h_d=
        \begin{cases}
        d-2,&3\mid d,\\
        d,&3\nmid d.
        \end{cases}
\]
Consequently, for every $m\ge1$,
\[
        \{d\ge1:\deg h_d=m\}
        =
        \bigl(\{m\}:3\nmid m\bigr)
        \cup
        \bigl(\{m+2\}:3\mid m+2\bigr),
\]
where a parenthesized singleton is omitted when its displayed condition fails.  In particular there is no characteristic-zero quotient of degree $3$.
\end{theorem}

\begin{proof}
By the torsion-defect formula, the only issue is to count pairs $(x,y)$ of $d$-th roots of unity with $1+x+y=0$.  These equations involve only $d$-th roots of unity and hence are defined over a cyclotomic field; choosing an embedding of that cyclotomic field into $\mathbb C$ preserves algebraic existence and nonexistence.  We may therefore count complex roots of unity.  Since $|x|=|y|=1$ and $y=-1-x$, the equality $|1+x|=1$ gives
\[
        2+x+\bar x=1,
        \qquad\text{hence}\qquad
        \operatorname{Re}(x)=-\frac12.
\]
Thus $x$ is one of the two primitive third roots of unity, and then $y$ is the other one.  Hence the defect is $2$ if $3\mid d$ and is $0$ otherwise.  Applying \cref{thm:torsion-defect} gives the degree formula.  The fixed-degree classification follows immediately by solving $m=d$ with $3\nmid d$ and $m=d-2$ with $3\mid d$.
\end{proof}

\section{Branch collisions and critical values}\label{sec:branch-collision}

The torsion-defect formula controls cancellation and degree.  This section develops the analogous invariant for branch-value collisions.  It is the key calculation behind the characteristic-zero monodromy theorem: over characteristic zero, the critical-value collision equation has no nondegenerate cyclotomic solutions.

Throughout this section $k$ is algebraically closed of characteristic $p\ge0$; divisibility hypotheses involving $p$ are void when $p=0$.  We use the coordinate
\[
        x=\frac{z}{z+1},\qquad z=\frac{x}{1-x}.
\]
In this coordinate the order-three automorphism is
\[
        \tau_x(x)=\frac1{1-x},
\]
and the raw quotient formula becomes
\begin{equation}\label{eq:Kd-coordinate}
        K_d^{\rm raw}(x)=H_d^{\rm raw}\left(\frac{x}{1-x}\right)
        =\frac{\eps-x^d}{x^d-(1-x)^d},
        \qquad \eps=(-1)^d.
\end{equation}
The symbol $k_d$ denotes the reduced rational function obtained from $K_d^{\rm raw}$ after canceling the common factor of its numerator and denominator.  Critical points, pole points, and critical values in this section refer to the reduced map $k_d$.

\begin{proposition}[critical equation in Hilbert--90 coordinates]\label{prop:critical-equation-x}
Assume $d\ge2$ and either $p=0$, or $p>0$ with $p\nmid d(d-1)$.  Put $m=d-1$.  Let $x$ be a finite critical point of the reduced quotient $k_d$, away from cancellation and pole points.  Then $x\notin\{0,1\}$,
\[
        x^m+(1-x)^m\ne0,
\]
and
\begin{equation}\label{eq:critical-equation-x}
        x^m(1-x)^m=\eps\bigl(x^m+(1-x)^m\bigr).
\end{equation}
Moreover, if $T=k_d(x)$ is the corresponding critical value, then
\begin{equation}\label{eq:critical-value-x}
        T=-\frac{\eps}{(1-x)^m},
        \qquad
        T+1=\frac{\eps}{x^m}.
\end{equation}
\end{proposition}

\begin{proof}
First $x=0$ and $x=1$ are not critical points in the stated range.  The raw denominator is nonzero at both points, so the following local expansions are expansions of the reduced map.  At $x=0$,
\[
        K_d^{\rm raw}(x)=-\eps-\eps d\,x+O(x^2),
\]
whose linear coefficient is nonzero because $p\nmid d$.  At $x=1$, put $u=1-x$.  If $\eps=1$, then
\[
        K_d^{\rm raw}(x)=d\,u+O(u^2),
\]
and if $\eps=-1$, then
\[
        K_d^{\rm raw}(x)=-2-d\,u+O(u^2),
\]
where in the second case $2\ne0$ under the hypothesis $p\nmid d(d-1)$.  In both cases the first nonconstant coefficient is nonzero.  Thus any critical point satisfying the hypotheses has $x\notin\{0,1\}$.

At an uncanceled non-pole point the equation $k_d(x)=T$ is equivalently the raw equation $K_d^{\rm raw}(x)=T$, namely
\[
        (T+1)x^d-T(1-x)^d=\eps .
\]
At a critical point over this value, differentiating with respect to $x$ and using $p\nmid d$ gives
\[
        (T+1)x^{d-1}+T(1-x)^{d-1}=0.
\]
Write $A=x^m$ and $B=(1-x)^m$.  Since $x\notin\{0,1\}$, both $A$ and $B$ are nonzero.  The derivative equation gives
\[
        T(A+B)+A=0.
\]
If $A+B=0$, then this equation gives $A=0$, a contradiction.  Hence $A+B\ne0$, so $T=-A/(A+B)$ and $T+1=B/(A+B)$.  Substituting into the equation for $K_d^{\rm raw}(x)=T$ gives
\[
        \frac{B}{A+B}\,xA+\frac{A}{A+B}\,(1-x)B=\frac{AB}{A+B}=\eps .
\]
Thus $AB=\eps(A+B)$, which is \eqref{eq:critical-equation-x}.  The formulas \eqref{eq:critical-value-x} follow by replacing $A+B$ with $AB/\eps$ in $T=-A/(A+B)$ and $T+1=B/(A+B)$.
\end{proof}

\begin{theorem}[branch-collision cross-ratio defect]\label{thm:branch-collision-defect}
Assume $d\ge2$ and either $p=0$, or $p>0$ with $p\nmid d(d-1)$.  Put $m=d-1$.  Let $x\ne y$ be finite critical points of the reduced quotient $k_d$, away from cancellation and poles.  If $k_d(x)=k_d(y)$, then there exist roots of unity
\[
        \zeta,\eta\in\mu_m,
        \qquad \zeta\ne1,
        \qquad \eta\ne1,
        \qquad \eta\ne\zeta,
\]
such that
\begin{equation}\label{eq:cross-ratio-defect}
        \left(\frac{\eta-\zeta}{\eta-1}\right)^m
        +
        \left(\frac{\eta-\zeta}{1-\zeta}\right)^m
        =\eps .
\end{equation}
Conversely, any pair $(\zeta,\eta)$ satisfying these nondegeneracy conditions and \eqref{eq:cross-ratio-defect} gives a pair of critical points with the same critical value by
\[
        s=\frac{\eta-1}{\eta-\zeta},
        \qquad
        x=1-s,
        \qquad
        y=1-\zeta s,
\]
provided the resulting points are neither poles nor canceled denominator points of the reduced map.
\end{theorem}

\begin{proof}
By \cref{prop:critical-equation-x}, equality of critical values gives
\[
        (1-y)^m=(1-x)^m.
\]
Hence $1-y=\zeta(1-x)$ for some $\zeta\in\mu_m$.  Put $s=1-x$, so $y=1-\zeta s$.  Since $x$ and $y$ are both critical, we have
\[
        s^m(1-s)^m=\eps\bigl(s^m+(1-s)^m\bigr)
\]
and
\[
        s^m(1-\zeta s)^m=\eps\bigl(s^m+(1-\zeta s)^m\bigr).
\]
Subtracting gives
\[
        \bigl((1-s)^m-(1-\zeta s)^m\bigr)(s^m-\eps)=0.
\]
The alternative $s^m=\eps$ is impossible, because substituting it in the critical equation gives $0=1$.  Therefore
\[
        (1-\zeta s)^m=(1-s)^m,
\]
so $1-\zeta s=\eta(1-s)$ for some $\eta\in\mu_m$.  The equality $x\ne y$ gives $\zeta\ne1$.  If $\eta=1$, then $1-\zeta s=1-s$, so $s=0$, contradicting the critical equation.  If $\eta=\zeta$, then $1-\zeta s=\zeta(1-s)$, so $\zeta=1$, again a contradiction.  Hence $\eta\ne1$ and $\eta\ne\zeta$.  Solving gives
\[
        s=\frac{\eta-1}{\eta-\zeta}.
\]
Now
\[
        1-s=\frac{1-\zeta}{\eta-\zeta}.
\]
Dividing the critical equation
\[
        s^m(1-s)^m=\eps\bigl(s^m+(1-s)^m\bigr)
\]
by $s^m(1-s)^m$ gives
\[
        \frac1{s^m}+\frac1{(1-s)^m}=\eps,
\]
which is exactly \eqref{eq:cross-ratio-defect}.  For the converse, the cross-ratio equation is exactly the critical equation after substituting the displayed value of $s$, and the same root-of-unity relations give equality of the two critical values.  Reversing the argument is valid on the open locus where the reduced map is defined and the raw denominator has not become an uncanceled pole.  The stated proviso excludes precisely the pole and canceled-denominator cases that are outside the critical-point locus of the theorem.
\end{proof}

\begin{theorem}[no characteristic-zero branch collisions]\label{thm:charzero-no-branch-collisions}
Let $k$ be algebraically closed of characteristic zero, let $d\ge2$, and put $m=d-1$.  Then the cross-ratio equation \eqref{eq:cross-ratio-defect} has no solutions satisfying the nondegeneracy conditions of \cref{thm:branch-collision-defect}.  Consequently, no two finite critical points of the reduced quotient $k_d$ away from cancellation and pole points have the same critical value.
\end{theorem}

\begin{proof}
The cross-ratio equation and the nondegeneracy conditions involve only $m$-th roots of unity and rational operations, so they are defined over a cyclotomic field.  Choosing an embedding of that cyclotomic field into $\mathbb C$ preserves the existence of such a solution; hence it suffices to work over $\mathbb C$.  Suppose, to the contrary, that a nondegenerate solution exists.  The equation is symmetric in $\zeta$ and $\eta$, so write
\[
        \zeta=e^{2\pi i r/m},\qquad \eta=e^{2\pi i s/m},
        \qquad 0<r<s<m.
\]
Put
\[
        i=s-r,
        \qquad j=r,
        \qquad k=m-s,
        \qquad \theta=\frac{\pi}{m},
        \qquad S_a=\sin(a\theta).
\]
Then $i,j,k$ are positive and $i+j+k=m$.  The elementary identity
\[
        e^{2iu}-e^{2iv}=2ie^{i(u+v)}\sin(u-v)
\]
gives
\[
        \left(\frac{\eta-\zeta}{\eta-1}\right)^m
        =(-1)^j\left(\frac{S_i}{S_k}\right)^m,
        \qquad
        \left(\frac{\eta-\zeta}{1-\zeta}\right)^m
        =(-1)^k\left(\frac{S_i}{S_j}\right)^m.
\]
Since $\eps=(-1)^d=(-1)^{m+1}$, the cross-ratio equation is equivalent, after multiplying by $S_j^mS_k^m$, to
\begin{equation}\label{eq:charzero-trig-collision}
        (-1)^j(S_iS_j)^m
        +(-1)^k(S_iS_k)^m
        +(-1)^m(S_jS_k)^m=0.
\end{equation}
If $i,j,k$ have the same parity, then all three signs in \eqref{eq:charzero-trig-collision} are equal, which is impossible.  Otherwise exactly one of $i,j,k$ has parity different from the other two.  Let $a$ be this parity outlier and let $b,c$ be the remaining two indices.  Then \eqref{eq:charzero-trig-collision} is equivalent to
\begin{equation}\label{eq:outlier-sine-equation}
        1=\left(\frac{S_a}{S_b}\right)^m
          +\left(\frac{S_a}{S_c}\right)^m .
\end{equation}
Indeed, the term not involving the outlier has the sign opposite to the other two terms.

Equation \eqref{eq:outlier-sine-equation} forces $S_a<S_b$ and $S_a<S_c$.  Since $a+b<m$ and $a+c<m$, the sign of $S_u-S_v$ agrees with the sign of $u-v$ for each pair $(u,v)=(a,b),(a,c)$; this follows from
\[
        \sin(u\theta)-\sin(v\theta)
        =2\cos\left(\frac{(u+v)\theta}{2}\right)
          \sin\left(\frac{(u-v)\theta}{2}\right),
\]
with the cosine positive.  Hence $a<b$ and $a<c$.  Thus $b,c\ge a+1$ and
\[
        m=a+b+c\ge 3a+2,
        \qquad\text{so}\qquad
        a\le \frac{m-2}{3}.
\]
Moreover $S_b,S_c\ge S_{a+1}$, because on the interval $a+1\le u\le m-a-1$ the sine $\sin(u\theta)$ is minimized at the endpoints, where both endpoint values equal $S_{a+1}$.  Therefore
\begin{equation}\label{eq:sine-ratio-bound}
        1
        \le 2\left(\frac{S_a}{S_{a+1}}\right)^m .
\end{equation}
For $m\ge5$ we have
\[
        (a+1)\theta\le \frac{m+1}{3m}\pi\le \frac{2\pi}{5}.
\]
Since $\cot t$ is decreasing and positive on $(0,2\pi/5]$,
\[
\begin{aligned}
        m\log\frac{S_{a+1}}{S_a}
        &=m\int_{a\theta}^{(a+1)\theta}\cot t\,dt  \\
        &\ge \pi\cot((a+1)\theta)
         \ge \pi\cot\frac{2\pi}{5}
         >\log2 .
\end{aligned}
\]
The final strict inequality is explicit: $\cot(2\pi/5)=1/\sqrt{5+2\sqrt5}$, so $\pi\cot(2\pi/5)>1>\log2$.
Thus $2(S_a/S_{a+1})^m<1$, contradicting \eqref{eq:sine-ratio-bound}.  For $m<5$, there is nothing to check for $m=1,2$; for $m=3$ the only triple is $(1,1,1)$, where all signs agree; and for $m=4$ the only triples are permutations of $(2,1,1)$, where the parity outlier has the larger sine, already contradicting \eqref{eq:outlier-sine-equation}.  Hence no nondegenerate solution exists.

The final assertion follows immediately from \cref{thm:branch-collision-defect}: any equality of critical values for two distinct finite critical points in the stated open locus would produce such a nondegenerate solution.
\end{proof}

\section{Ramification skeletons and characteristic-zero symmetric monodromy}\label{sec:ramification-skeleton}

The branch-collision theorem reduces the characteristic-zero monodromy problem to a tame Morse calculation.  This section records the Wronskian skeleton, proves simple criticality in the tame range, and then proves full symmetric monodromy in characteristic zero.  Positive-characteristic bad reductions are deliberately deferred to Part IV.

Throughout this section $k$ is algebraically closed of characteristic $p\ge0$, $p\nmid d$ if $p>0$, and
\[
        H_d^{\rm raw}=\frac{N_d}{D_d},
        \qquad
        N_d=(-1)^d(z+1)^d-z^d,
        \qquad
        D_d=z^d-1.
\]
Let $C_d=\gcd(N_d,D_d)$ and write the reduced quotient as
\[
        h_d=\frac{n_d}{d_d},
        \qquad
        n_d=N_d/C_d,
        \qquad
        d_d=D_d/C_d.
\]
The reduced Wronskian is
\[
        W(h_d)=n_d'd_d-n_dd_d'.
\]

\begin{proposition}[Wronskian skeleton]\label{prop:wronskian-skeleton}
One has
\[
        \frac1d\bigl(N_d'D_d-N_dD_d'\bigr)
        =F_d(z),
\]
where
\[
        F_d(z)=z^{d-1}-(-1)^d(z+1)^{d-1}\bigl(z^{d-1}+1\bigr).
\]
Moreover
\[
        N_d'D_d-N_dD_d'=C_d^2 W(h_d).
\]
\end{proposition}

\begin{proof}
The first identity is the direct calculation
\[
\begin{aligned}
\frac1d(N_d'D_d-N_dD_d')
&=\bigl((-1)^d(z+1)^{d-1}-z^{d-1}\bigr)(z^d-1)  \\
&\quad-\bigl((-1)^d(z+1)^d-z^d\bigr)z^{d-1}        \\
&=z^{d-1}-(-1)^d(z+1)^{d-1}(z^{d-1}+1).
\end{aligned}
\]
The second identity follows by writing $N_d=C_dn_d$ and $D_d=C_dd_d$ and expanding; the terms involving $C_d'$ cancel.
\end{proof}

\begin{theorem}[tame simple-critical theorem]\label{thm:tame-simple-critical}
Assume that $h_d$ is nonconstant and that either $d=1$, or $p=0$, or $p>0$ and
\[
        p\nmid d(d-1).
\]  Then every finite critical point of the reduced quotient $h_d$ is simple.  The finite poles of $h_d$ are unramified, and the point at infinity is unramified.
\end{theorem}

\begin{proof}
If $d=1$, the nonconstancy hypothesis excludes the characteristic-$3$ constant quotient; the reduced quotient is then a degree-one map, so there is no ramification to check.  Hence assume $d\ge2$.  Under the positive-characteristic tame hypothesis $p\nmid d(d-1)$, this also forces $p\ne2$, since one of $d$ and $d-1$ is even.  Thus the later canceled-point and infinity calculations never require a hidden division by $2$ in characteristic $2$.  Let $m=d-1$ and put $\epsilon=(-1)^d$.  Suppose first that $\alpha$ is a multiple zero of the unreduced skeleton
\[
        F_d(z)=z^m-\epsilon(z+1)^m(z^m+1)
\]
which is not a canceled denominator point.  Since $F_d(0)$ and $F_d(-1)$ are nonzero, we have $\alpha(\alpha+1)\ne0$.  The equations $F_d(\alpha)=0$ and $F_d'(\alpha)=0$ give
\[
        \alpha^m=\epsilon(\alpha+1)^m(\alpha^m+1)
\]
and, because $p\nmid m$,
\[
        \alpha^{m-1}
        =\epsilon(\alpha+1)^{m-1}\bigl(2\alpha^m+\alpha^{m-1}+1\bigr).
\]
Dividing the second equation by the first gives
\[
        \frac1\alpha
        =\frac{2\alpha^m+\alpha^{m-1}+1}{(\alpha+1)(\alpha^m+1)}.
\]
After clearing denominators,
\[
        (\alpha+1)(\alpha^m+1)=\alpha(2\alpha^m+\alpha^{m-1}+1),
\]
hence
\[
        \alpha^{m+1}=1,
\]
that is $\alpha^d=1$.  Since $D_d'(\alpha)=d\alpha^{d-1}\ne0$, the equality $N_d'D_d-N_dD_d'=0$ at $\alpha$ forces $N_d(\alpha)=0$.  Thus $\alpha$ is a common zero of $N_d$ and $D_d$, contrary to the assumption.  Hence, away from the cancellation locus, the reduced Wronskian has no multiple finite zeros.

It remains to check that the cancellation points themselves do not become critical after reduction.  Let $\alpha$ be a common zero of $N_d$ and $D_d$.  Then $\alpha^d=1$ and $\tau(\alpha)^d=1$.  Since the cancellation is simple, write $u=z-\alpha$ and expand
\[
        N_d=a_1u+a_2u^2+O(u^3),
        \qquad
        D_d=b_1u+b_2u^2+O(u^3),
\]
with $a_1=N_d'(\alpha)$, $b_1=D_d'(\alpha)$, $a_2=N_d''(\alpha)/2$, and $b_2=D_d''(\alpha)/2$.  The derivative at $\alpha$ of the quotient after canceling $u$ is
\[
        \frac{a_2b_1-a_1b_2}{b_1^2}
        =\frac{N_d''(\alpha)D_d'(\alpha)-N_d'(\alpha)D_d''(\alpha)}{2D_d'(\alpha)^2}.
\]
Since $2$ is invertible in the present cases, the derivative of the reduced quotient at $\alpha$ is nonzero if and only if
\[
        N_d''(\alpha)D_d'(\alpha)-N_d'(\alpha)D_d''(\alpha)\ne0.
\]
Using $\alpha^d=1$ and $(-1)^d(\alpha+1)^d=1$, one computes
\[
        N_d'(\alpha)=d\left(\frac1{\alpha+1}-\frac1\alpha\right),
        \qquad
        D_d'(\alpha)=\frac d\alpha,
\]
\[
        N_d''(\alpha)=d(d-1)\left(\frac1{(\alpha+1)^2}-\frac1{\alpha^2}\right),
        \qquad
        D_d''(\alpha)=\frac{d(d-1)}{\alpha^2}.
\]
Therefore
\[
        N_d''(\alpha)D_d'(\alpha)-N_d'(\alpha)D_d''(\alpha)
        =-\frac{d^2(d-1)}{\alpha^2(\alpha+1)^2},
\]
which is nonzero under the tame hypothesis.  Thus no canceled point is ramified in the reduced map.

Every finite pole of $h_d$ is a simple pole, since $D_d=z^d-1$ is squarefree and cancellation removes only simple factors.  Hence finite poles are unramified over the branch value $\infty$.  Finally, in the local coordinate $w=1/z$ at infinity, if $d$ is even then $H_d^{\rm raw}=dw+O(w^2)$, and if $d$ is odd and the characteristic is not $2$ then $H_d^{\rm raw}=-2-dw+O(w^2)$.  Thus infinity is also unramified.  This proves the theorem.
\end{proof}

\begin{definition}[Morse branch polynomial]\label{def:morse-branch}
For a nonconstant separable reduced quotient $h_d$ of degree $n>1$ in the tame range of \cref{thm:tame-simple-critical}, define the branch polynomial
\[
        B_{d,p}(T)=\operatorname{Res}_z\bigl(n_d(z)-Td_d(z),\,W(h_d)(z)\bigr)\in k[T].
\]
We say that $h_d$ is \emph{Morse} if $B_{d,p}(T)$ is squarefree.  Equivalently, by \cref{thm:tame-simple-critical}, the finite critical values of $h_d$ are pairwise distinct.
\end{definition}

\begin{theorem}[symmetric monodromy in the tame Morse range]\label{thm:tame-morse-Sn}
Assume either $p=0$, or $p>0$ and $p\nmid d(d-1)$.  Assume also that $\deg h_d=n>1$ and that $h_d$ is Morse.  Then the geometric monodromy group of
\[
        h_d:\PP^1\longrightarrow\PP^1
\]
is
\[
        G_{h_d}=S_n.
\]
\end{theorem}

\begin{proof}
By \cref{thm:tame-simple-critical}, every finite critical point is simple, and no pole or point over infinity is ramified.  By the Morse hypothesis, the finite critical values are pairwise distinct.  Thus every finite branch cycle is a transposition.  The cover $\PP^1\to\PP^1$ is connected, so its geometric monodromy group is transitive.  A transitive subgroup of $S_n$ generated by transpositions is $S_n$: the graph with vertices the $n$ sheets and edges the transpositions is connected, and the group generated by edge transpositions of a connected graph is the full symmetric group.  Hence $G_{h_d}=S_n$.
\end{proof}

\begin{theorem}[full symmetric monodromy in characteristic zero]\label{thm:charzero-Sn}
Let $k$ be algebraically closed of characteristic zero, let $h_d$ be the reduced quotient associated with $H_d^{\rm raw}$, and put $n=\deg h_d$.  If $n>1$, then
\[
        G_{h_d}=S_n.
\]
Equivalently,
\[
        G_{h_d}=
        \begin{cases}
        S_d,&3\nmid d,\ d\ge2,\\
        S_{d-2},&3\mid d,\ d\ge6.
        \end{cases}
\]
\end{theorem}

\begin{proof}
By \cref{thm:charzero-degree-strata}, the displayed degree formula is $n=d$ when $3\nmid d$ and $n=d-2$ when $3\mid d$.  The degree-one cases are exactly $d=1,3$, so assume $n>1$.

Work in the coordinate $x=z/(z+1)$.  The point $x=\infty$, which corresponds to $z=-1$, is unramified: if $d$ is even, then it is a simple pole of the raw expression $K_d^{\rm raw}$ and hence of the reduced map; if $d$ is odd, writing $t=1/x$ gives
\[
        K_d^{\rm raw}=\frac{-t^d-1}{1+(1-t)^d}
            =-\frac12-\frac d4t+O(t^2),
\]
with the evident harmless modification in the excluded case $d=1$.  The point $x=1$, which corresponds to $z=\infty$, is unramified by the infinity calculation in \cref{thm:tame-simple-critical}.  All remaining finite points of the $x$-line correspond to finite $z$-points, so \cref{thm:tame-simple-critical} gives simple finite critical points after removing cancellations and shows that finite poles and canceled denominator points are unramified.  By \cref{thm:charzero-no-branch-collisions}, no two finite non-pole critical points in this affine locus have the same critical value.

Consequently all ramification of $h_d$ occurs at simple finite critical points, these ramified points have pairwise distinct branch values, and there is no ramification at poles, canceled points, $x=1$, or $x=\infty$.  Equivalently, $h_d$ is Morse in the sense of \cref{def:morse-branch}; each finite branch cycle is a single transposition, and no additional branch cycle arises over $\infty$ or over a canceled point.

Now apply \cref{thm:tame-morse-Sn}.  Equivalently, the transposition branch cycles generate a transitive subgroup of $S_n$ because the cover is connected; hence they generate the full symmetric group.
\end{proof}

\begin{corollary}[good reductions of every fixed non-linear quotient]\label{cor:good-reductions-Sn}
Fix an integer $d\ge2$ and work initially over characteristic zero.  Let $h_d$ be the reduced quotient and put $n=\deg h_d$.  If $n>1$, then there is an explicitly computable nonzero integer $M_d$ such that, for every prime $p\nmid M_d$, the reduction of $h_d$ modulo $p$ has geometric monodromy $S_n$.
\end{corollary}

\begin{proof}
Choose primitive coprime representatives $n_d,d_d\in\Z[z]$ for the reduced characteristic-zero quotient, and put $n=\deg h_d$.  By \cref{cor:reduced-denominator-degree}, $\deg d_d=n$.  Put
\[
        W_d=n_d'd_d-n_dd_d',
        \qquad
        B_d(T)=\operatorname{Res}_z(n_d(z)-Td_d(z),W_d(z))\in\Z[T].
\]
To record the local condition at infinity explicitly, set
\[
        \widehat n_d(w)=w^n n_d(1/w),\qquad
        \widehat d_d(w)=w^n d_d(1/w),
\]
and define
\[
        \iota_d=\widehat n_d'(0)\widehat d_d(0)-\widehat n_d(0)\widehat d_d'(0)\in\Z.
\]
The characteristic-zero calculation in \cref{thm:tame-simple-critical} gives $\iota_d\ne0$; it is the numerator of the local derivative of $h_d$ at infinity in the source coordinate $w=1/z$.

By \cref{thm:tame-simple-critical,thm:charzero-no-branch-collisions}, the characteristic-zero quotient is Morse, so the relevant resultants and discriminants below are nonzero.  It is enough to take $M_d$ divisible by
\[
\begin{aligned}
        &d(d-1)\,\iota_d\,\operatorname{Res}(n_d,d_d)\,
        \operatorname{disc}(d_d)\,
        \operatorname{Res}(W_d,d_d)\,
        \operatorname{disc}(W_d)\,\operatorname{disc}(B_d)  \\
        &\qquad\text{and by the nonzero leading coefficients of }n_d,d_d,W_d,B_d.
\end{aligned}
\]
Factors corresponding to constant polynomials are omitted.  For primes outside this finite set, reduction preserves the degree of the quotient, the numerator-denominator coprimality, the pole divisor, the simple finite critical divisor, the unramified local behavior at infinity, and the separation of finite critical values.  More explicitly, the leading-coefficient factors preserve the degrees of $n_d$, $d_d$, $W_d$, and $B_d$; $\operatorname{Res}(n_d,d_d)$ prevents a new numerator-denominator gcd and hence prevents new cancellation after reduction; $\operatorname{disc}(d_d)$ and $\operatorname{Res}(W_d,d_d)$ preserve simple unramified finite poles; $\iota_d$ preserves the nonzero local derivative at infinity; $\operatorname{disc}(W_d)$ preserves the reduced finite critical divisor; and $\operatorname{disc}(B_d)$ preserves distinct finite critical values.  The tame hypothesis $p\nmid d(d-1)$ is also included in the definition of $M_d$.  Thus the reduced quotient modulo $p$ remains Morse in the sense of \cref{def:morse-branch}, and \cref{thm:tame-morse-Sn} applies.
\end{proof}

\section*{Part IV. Positive-characteristic degenerations and twisted exceptionality}
\section{Off-diagonal fiber squares and twisted exceptionality}\label{sec:collision}

Throughout this section, fix a prime $\ell\ne p$ for all $\ell$-adic cohomology groups.

This section gives the structural obstruction that converts large monodromy into failure of twisted exceptionality.  The results are stated for an arbitrary separable $\tau$-equivariant rational map because no special feature of the formula for $H_d^{\rm raw}$ is needed until the final corollaries.  Off-diagonal collisions on $\Lambda_q$ are detected, away from finitely many branch fibers, by the off-diagonal finite-etale fiber square and its action under the twisted Frobenius.

\begin{definition}[twisted collision complex]\label{def:collision-complex}
Let $f:\PP^1\to\PP^1$ be a separable nonconstant $\tau$-equivariant rational map over $\F_p$.  Let $B\subset\PP^1$ be its branch locus.  Since $f\circ\tau=\tau\circ f$, the set $B$ is $\tau$-stable.  Put
\[
        U=\PP^1\setminus B,
        \qquad
        V=f^{-1}(U).
\]
Then $f:V\to U$ is finite etale.  Define the off-diagonal collision curve
\[
        C_f=(V\times_UV)\setminus\Delta,
\]
where $\Delta$ is the diagonal.  The \emph{twisted collision complex} of $(f,\tau)$ is
\[
        \mathsf{TC}(f,\tau)=R\Gamma_c(C_{f,\barF_p},\Ql).
\]
The action of $\Thetaq=\tau^{-1}\Frob_q$ on $\PP^1$ preserves $U$, $V$, and $C_f$, and hence acts on $\mathsf{TC}(f,\tau)$.
\end{definition}

\begin{remark}\label{rem:collision-no-derived}
No derived formalism is used in the arguments below.  The term ``collision complex'' refers only to the compactly supported etale cohomology complex of the ordinary off-diagonal fiber square over the finite etale locus.  This is the part of the fiber square that controls asymptotic injectivity on the twisted fixed sets.
\end{remark}

\begin{lemma}[twisted Frobenius weight bound]\label{lem:twisted-frobenius-weights}
Let $X$ be a separated curve over $\F_p$ preserved by an automorphism $\alpha$ of finite order $r$ defined over $\F_p$, and put $\Theta=\alpha^{-1}\Frob_q$.  If $\lambda$ is an eigenvalue of $\Theta$ on $H_c^i(X_{\barF_p},\Ql)$, then $\lambda^r$ is an eigenvalue of $\Frob_{q^r}$ on the same cohomology group.  In particular, for $i=1$ one has $|\lambda|\le q^{1/2}$.
\end{lemma}

\begin{proof}
The automorphism $\alpha$ commutes with Frobenius because it is defined over $\F_p$.  Hence $\Theta^r=\Frob_{q^r}$.  If $\lambda$ is an eigenvalue of $\Theta$, then $\lambda^r$ is an eigenvalue of $\Theta^r$, hence of $\Frob_{q^r}$.  Deligne's weight bound \cite{DeligneWeilII} gives $|\lambda^r|\le q^{ri/2}$; for $i=1$ this gives $|\lambda|\le q^{1/2}$.
\end{proof}

\begin{theorem}[twisted Lefschetz collision formula]\label{thm:collision-trace}
Let $f$ be as in \cref{def:collision-complex}.  For $q=p^k$, define
\[
        N_f(q)=\#\{(x,y)\in\Lambda_q^2:x\ne y,\ x,y\in V,\ f(x)=f(y)\}.
\]
Then
\[
        N_f(q)=
        \sum_{i=0}^2(-1)^i
        \operatorname{Tr}\bigl(\Thetaq\mid H_c^i(C_{f,\barF_p},\Ql)\bigr).
\]
Moreover, define
\[
        a_f(q)=\operatorname{Tr}\!\left(\Thetaq\mid
        \Ql[\pi_0(C_{f,\barF_p})]\right).
\]
Equivalently, $a_f(q)$ is the number of geometric irreducible components of $C_{f,\barF_p}$ fixed by $\Thetaq$.  Then
\[
        N_f(q)=a_f(q)q+O_f(q^{1/2}).
\]
Consequently there is a constant $Q_f$ such that, for every $q>Q_f$ with $a_f(q)>0$, the map $f$ is not injective on $\Lambda_q$.  In particular, if $a_f(p^k)>0$ for all sufficiently large $k$, then $f$ is not $\tau$-twisted exceptional.
\end{theorem}

\begin{proof}
Since $f$ is defined over $\F_p$ and satisfies $f\circ\tau=\tau\circ f$, it also satisfies $f\circ\Thetaq=\Thetaq\circ f$ for every $q=p^k$; hence $f$ preserves the twisted fixed sets, and $U$, $V$, and $C_f$ are stable under $\Thetaq$.  The fixed points of $\Thetaq$ on $C_f$ are precisely the ordered off-diagonal pairs $(x,y)$ with $x,y\in\Lambda_q\cap V$ and $f(x)=f(y)$.  We use the standard equivariant Grothendieck--Lefschetz trace formula for a finite-order automorphism composed with Frobenius.  Here $\tau$ is defined over $\F_p$ and has finite order, so $\Thetaq=\tau^{-1}\Frob_q$ is exactly such an endomorphism of the separated curve $C_f$; see, for example, \cite[VI]{SGA4half} or \cite[Chapter VI]{MilneEtale} for the trace formula in etale cohomology.  This gives the first formula.

If $\deg f=1$, then $B=\varnothing$, $U=V=\PP^1$, and $C_f=(\PP^1\times_{\PP^1}\PP^1)\setminus\Delta=\varnothing$.  The displayed formulas hold with $N_f(q)=0$ and $a_f(q)=0$.  Hence assume $\deg f>1$.  By Riemann--Hurwitz the branch locus is nonempty, so $U$ and $V$ are nonempty affine curves.  Since $V\to U$ is finite etale, the off-diagonal fiber square $C_f$ is a smooth affine curve over $\barF_p$, possibly disconnected, with no proper zero-dimensional components.  Hence $H_c^i(C_f,\Ql)=0$ for $i\notin\{0,1,2\}$ and $H_c^0=0$.  The group $H_c^2(C_f,\Ql)$ is a direct sum of one copy of $\Ql(-1)$ for each geometric irreducible component of $C_f$.  Thus, after the Tate twist, the induced action is the permutation representation on $\pi_0(C_{f,\barF_p})$.  A component fixed by $\Thetaq$ contributes $q$ to the untwisted top-cohomology trace, and a component in a longer cycle contributes zero to the trace.  By the definition of $a_f(q)$, the top term is $a_f(q)q$.

The remaining term is $H_c^1$.  Applying \cref{lem:twisted-frobenius-weights} with $\alpha=\tau$ and $r=3$ shows that every eigenvalue of $\Thetaq$ on $H_c^1$ has absolute value at most $q^{1/2}$.  The dimension of $H_c^1$ depends only on $f$.  This gives the estimate.  Choose $Q_f$ so that the error term has absolute value less than $q$ for all $q>Q_f$.  If $a_f(q)>0$ and $q>Q_f$, then $N_f(q)>0$, so $f$ has an off-diagonal collision on $\Lambda_q$.  If this happens for all sufficiently large $q=p^k$, then only finitely many $k$ can give bijectivity on $\Lambda_q$, which is exactly the negation of $\tau$-twisted exceptionality.
\end{proof}

\begin{proposition}[top cohomology and geometric monodromy]\label{prop:top-monodromy}
Let $\Omega$ be a geometric generic fiber of $f:V\to U$, and let $G_f\le\operatorname{Sym}(\Omega)$ be the geometric monodromy group.  Then the geometric irreducible components of $C_f$ are naturally indexed by the $G_f$-orbits on
\[
        \Omega^{(2)}=\{(\alpha,\beta)\in\Omega^2:\alpha\ne\beta\}.
\]
Consequently,
\[
        H_c^2(C_{f,\barF_p},\Ql)(1)
        \cong
        \Ql\bigl[\Omega^{(2)}/G_f\bigr]
\]
as a permutation representation of the arithmetic automorphisms preserving the cover.
\end{proposition}

\begin{proof}
In the present rational-map setting, $V_{\barF_p}$ is a dense open subset of $\PP^1_{\barF_p}$ and is therefore connected.  For a connected finite etale cover $V\to U$, the connected components of the fiber product $V\times_UV$ over $\barF_p$ correspond to the orbits of the geometric fundamental group of $U$ on $\Omega\times\Omega$.  The diagonal component corresponds to the diagonal orbit.  Removing it leaves the orbits on ordered distinct pairs.  The image of the geometric fundamental group in $\operatorname{Sym}(\Omega)$ is precisely $G_f$, giving the indexing set.  The assertion about $H_c^2(1)$ follows because the top compactly supported cohomology of each smooth affine irreducible component is one-dimensional after Tate twist.
\end{proof}

\begin{corollary}[monodromy obstruction to twisted exceptionality]\label{cor:2trans-obstruction}
Let $f$ be a separable $\tau$-equivariant rational map over $\F_p$ with $\deg f>1$.  If the geometric monodromy group $G_f$ is $2$-transitive, then $f$ is not $\tau$-twisted exceptional.  More precisely,
\[
        N_f(q)=q+O_f(q^{1/2}),
\]
so $f$ has off-diagonal collisions on $\Lambda_q$ for all sufficiently large $q=p^k$.
\end{corollary}

\begin{proof}
The connected-cover hypothesis in \cref{prop:top-monodromy} applies because $V$ is a dense open subset of $\PP^1$.  If $G_f$ is $2$-transitive, then it has one orbit on $\Omega^{(2)}$.  The connected-cover hypotheses used in \cref{prop:top-monodromy} hold here because $V$ is a nonempty open subset of $\PP^1$.  Thus $C_f$ is geometrically irreducible.  Since there is only one geometric component, every arithmetic automorphism preserving the cover, in particular every $\Thetaq$, fixes it.  Therefore $a_f(q)=1$ for every $q$, and Theorem~\ref{thm:collision-trace} gives $N_f(q)=q+O_f(q^{1/2})$.  This is positive for all sufficiently large $q$.
\end{proof}

\begin{corollary}[tame Morse quotients are not twisted-exceptional]\label{cor:tame-morse-not-exceptional}
Let $p>0$, assume $p\nmid d(d-1)$, and suppose $\deg h_d>1$ and $h_d$ is Morse.  Then $h_d$ is not $\tau$-twisted exceptional.
\end{corollary}

\begin{proof}
By \cref{thm:tame-morse-Sn}, the geometric monodromy group is the full symmetric group, hence is $2$-transitive.  Apply \cref{cor:2trans-obstruction}.
\end{proof}

\begin{corollary}[positive-characteristic degree-two quotients are not twisted-exceptional]\label{cor:degree-two-not-exceptional}
In positive characteristic, let $h_d$ be one of the tame degree-two quotient maps classified in \cref{thm:degree-two}.  Then $h_d$ is not $\tau$-twisted exceptional.  Thus neither the positive-characteristic case $d=2$ with $p\ne2,3$ nor the sporadic pair $(p,d)=(11,5)$ yields an infinite quotient-permutation family on the sets $\Lambda_q$.
\end{corollary}

\begin{proof}
A separable rational map of degree two has geometric monodromy $S_2$, which is $2$-transitive on a two-point fiber.  Apply Corollary~\ref{cor:2trans-obstruction}.  The degree-two cases in \cref{thm:degree-two} are separable because their degree is smaller than the characteristic in the sporadic case and because the excluded characteristics are precisely the inseparable or collapsed cases for $d=2$.
\end{proof}

\begin{lemma}[inertia orbits and ramification indices]\label{lem:inertia-orbit-lengths}
Let $\phi:X\to Y$ be a finite separable cover of smooth curves over an algebraically closed field, let $y\in Y$ be a branch value, and let $I$ be an inertia group at a point of the Galois closure above $y$.  In the induced action of $I$ on a geometric generic fiber of $\phi$, the orbits are indexed by the points $x\in X$ above $y$, and the orbit attached to $x$ has length $e_x(\phi)$.

In particular, if the fiber over $y$ has exactly one ramified point of index $m$ and all other points over $y$ are unramified, then the inertia image has one orbit of length $m$ and fixes the remaining sheets.  If $m=2$, the monodromy group contains a transposition; if $m$ is prime, it contains an $m$-cycle fixing the unramified sheets.
\end{lemma}

\begin{proof}
After completing at $y$, the points of $X$ above $y$ correspond to the extensions of the completed local field of $Y$ occurring in the tensor product with the function field of $X$.  In the Galois closure, the decomposition group acts transitively on the places above a fixed point $x$, and the quotient by inertia records the residue extension.  Since the base field is algebraically closed, the residue extensions are trivial, so the inertia orbit over $x$ has size equal to the ramification index $e_x(\phi)$.  This is the standard local description of branch cycles via completed discrete valuation rings; compare \cite[Chapter III, \S6]{SerreLocalFields}.

In the final situation, all unramified sheets are singleton inertia orbits and hence are fixed by every inertia element.  On the unique nontrivial orbit of length $m$, the inertia image is a transitive permutation group.  If $m=2$, any nonidentity element on that orbit is a transposition.  If $m$ is prime, the order of the transitive inertia image is divisible by $m$; by Cauchy's theorem it contains an element of order $m$, and such an element is an $m$-cycle on the length-$m$ orbit.  Since the remaining orbits are singletons, this $m$-cycle fixes the unramified sheets.
\end{proof}

\begin{proposition}[Wronskian and an $S_n$ criterion]\label{prop:wronskian-criterion}
Let $p>0$, assume $p\nmid d$, and write $H_d^{\rm raw}=N_d/D_d$ before cancellation, with
\[
        N_d=(-1)^d(z+1)^d-z^d,
        \qquad
        D_d=z^d-1.
\]
Then
\[
        \frac1d\bigl(N_d'D_d-N_dD_d'\bigr)
        =z^{d-1}-(-1)^d(z+1)^{d-1}(z^{d-1}+1).
\]
Let $h_d$ be the reduced form of $H_d^{\rm raw}$, and put $n=\deg h_d$.  Suppose $n>1$, $h_d$ is separable, $h_d$ is indecomposable over $\barF_p$, and $h_d$ has a simple branch point, i.e. a branch value over which exactly one point is ramified and its ramification index is $2$.  Then the geometric monodromy group of $h_d$ is $S_n$.  Consequently $h_d$ is not $\tau$-twisted exceptional.
\end{proposition}

\begin{proof}
The displayed identity is a direct calculation:
\[
\begin{aligned}
\frac1d(N_d'D_d-N_dD_d')
&=\bigl((-1)^d(z+1)^{d-1}-z^{d-1}\bigr)(z^d-1)  \\
&\quad-\bigl((-1)^d(z+1)^d-z^d\bigr)z^{d-1}        \\
&=z^{d-1}-(-1)^d(z+1)^{d-1}(z^{d-1}+1).
\end{aligned}
\]
The separability assumption ensures that the geometric monodromy action is defined on $n$ separable sheets; the simple-branch-point condition is then a statement about this separable cover.  After canceling common factors, the simple branch point hypothesis says that over the corresponding branch value there is one point of ramification index $2$ and all other points are unramified.  By \cref{lem:inertia-orbit-lengths}, including in wild characteristic, the inertia image on a generic fiber contains a transposition.  Indecomposability of a rational function over an algebraically closed field is equivalent to primitivity of its geometric monodromy action; see Fried--MacRae \cite{FriedMacRae} for the separated-variables framework underlying this equivalence.  A primitive permutation group containing a transposition is the full symmetric group $S_n$.  Thus $G_{h_d}=S_n$, and Corollary~\ref{cor:2trans-obstruction} proves non-exceptionality.
\end{proof}

\section{Positive-characteristic monodromy: conditional reduction and architecture}\label{sec:positive-architecture}

The collision obstruction turns monodromy into an arithmetic obstruction.  The following conditional theorem and conjectural architecture separate the proved exceptional mechanisms from the remaining positive-characteristic monodromy problem.

\begin{theorem}[conditional quotient classification in the two-transitive range]\label{thm:classification-2transitive-range}
Fix a prime $p$ and write $d=p^s d_0$ with $p\nmid d_0$.  Assume that every separable quotient outside the following three structural alternatives has $2$-transitive geometric monodromy:
\[
        \deg h_{d_0}\le1,
        \qquad
        d_0=p^a+1\text{ for some }a\ge1,
        \qquad
        (p,d_0)=(19,6).
\]
Then every $\tau$-twisted exceptional quotient in this range belongs to one of these three alternatives.  Within the first two alternatives, the nonconstant degree-one quotients are those classified in \cref{thm:degree-one}, and the Frobenius-sparse plus branch $d_0=p^a+1$ has exactly the quotient criteria of \cref{thm:plus-branch-exceptionality}.  The sporadic alternative $(p,d_0)=(19,6)$ is isolated here only as an allowed alternative.  This is a quotient-level statement; trace-zero permutation exponents require the additional fiber and denominator conditions of \cref{prop:mult-desc}.
\end{theorem}

\begin{proof}
By \cref{prop:p-power-reduction}, $h_d=\Frob_{p^s}\circ h_{d_0}$ on values.  This Frobenius factor does not affect quotient-level twisted exceptionality: for every $q=p^k$, the map $\Frob_{p^s}$ commutes with $\Thetaq=\tau^{-1}\Frob_q$ because $\tau$ is defined over $\F_p$, hence $\Frob_{p^s}$ restricts to a bijection of the twisted fixed set $\PP^1(\barF_p)^{\Thetaq}$.  It does, however, multiply the full morphism degree by $p^s$.  Thus the separable cover relevant for monodromy and quotient-level twisted exceptionality is $h_{d_0}$.  If $\deg h_{d_0}\le1$, then \cref{thm:degree-one} gives the low-degree branch.  The constant subcase is excluded by \cref{def:twisted-exceptional}, which only applies to nonconstant rational functions; every nonconstant degree-one quotient is an automorphism commuting with $\tau$ and hence is bijective on each twisted fixed set.  No trace-zero fiber or denominator condition is part of this quotient-level statement.  If $d_0=p^a+1$ with $a\ge1$, \cref{thm:plus-branch-exceptionality} gives exactly the sparse quotient-permutation criteria.  If $(p,d_0)=(19,6)$, this theorem makes no further claim beyond retaining it as one of the stated alternatives.  All remaining quotients of degree greater than one have $2$-transitive monodromy by hypothesis, and Corollary~\ref{cor:2trans-obstruction} rules out twisted exceptionality.  Therefore every twisted exceptional quotient in the stated range belongs to one of the three listed alternatives.
\end{proof}

\begin{conjecture}[positive-characteristic sparse/sporadic/lacunary/generic architecture]
\label{conj:trichotomy}
Let $p>0$, let $p\nmid d$, and let $h_d$ be the reduced separable quotient associated with $H_d^{\rm raw}$ in characteristic $p$.  The characteristic-zero case is completely settled by \cref{thm:charzero-Sn}; in positive characteristic, apart from the small overlaps $d=1,2,3$, the following alternatives are expected to be exhaustive.
\begin{enumerate}[label=\textup{(\roman*)}]
\item If $d=p^a-1$ with $a\ge1$, then $h_d=\tau^2$ is the linear Frobenius-sparse branch.
\item If $d=p^a+1$ with $a\ge1$, then $h_d$ is the Kummer or Artin--Schreier Frobenius-sparse branch of \cref{thm:frob-sparse-normal-forms}.
\item If $(p,d)=(19,6)$, then, under the certificate-execution hypothesis of \cref{thm:sporadic-V4}, $h_d$ is the Klein-four sporadic quotient.
\item If $p\ne2,3$ and $d=2p^a+1$ with $a\ge1$, then $h_d$ belongs to the first nonsparse Frobenius--lacunary tower.  Put $n=\deg h_d$, equivalently the degree of the separable quotient under \cref{def:quotient-conventions}.  This paper proves that this tower is primitive for every such $a$, and proves $A_n\le G_{h_d}\le S_n$ for $a=1$.  For $a>1$, the expected conclusion is again $A_n\le G_{h_d}\le S_n$; it would follow from sufficiently controlling the explicit wild inertia representation of \cref{thm:two-r-plus-one-normal-form}.
\item In every remaining positive-characteristic case with $\deg h_d=n>1$, one expects
\[
        A_n\le G_{h_d}\le S_n.
\]
In particular the remaining quotients should not be $\tau$-twisted exceptional.
\end{enumerate}
\end{conjecture}

\begin{remark}
The sporadic alternative is forced by the computer-assisted \cref{thm:sporadic-V4}; without it the conjecture is false under the certificate-execution hypothesis.  The exclusion of $d=p^a+1$ with $a\ge1$ is forced by \cref{cor:sparse-monodromy}, since those quotients have cyclic monodromy for $p\ne3$ and elementary abelian monodromy for $p=3$.  For $p\ne2,3$, the first nonsparse lacunary tower is no longer a primitivity problem by \cref{thm:first-lacunary-primitive}; for $a>1$ it is a local wild-inertia problem.  The cohomological collision obstruction is unconditional.  Since characteristic zero is now settled by \cref{thm:charzero-Sn}, the remaining difficult problem is the positive-characteristic generic monodromy theorem: prove primitivity outside the lacunary towers and extract a small-support branch cycle from the branch-collision equation, from bad-reduction branch-cycle products, or from the wild Wronskian skeleton.  This is the same structural division that appears in the classical exceptional-cover literature, where the arithmetic permutation property is converted into a statement about components of a fiber square and then into monodromy; compare Fried--Guralnick--Saxl \cite{FriedGuralnickSaxl1993}, Guralnick--Tucker--Zieve \cite{GTZ2007}, and the monodromy classifications of Guralnick--Zieve and Guralnick--Rosenberg--Zieve \cite{GZ2010,GRZ2010}.
\end{remark}

\section{The sporadic \texorpdfstring{$V_4$}{V4} quotient in characteristic \texorpdfstring{$19$}{19}}\label{sec:sporadic}

The first positive-characteristic escape mechanism beyond the Frobenius-sparse branches is a genuine low-monodromy sporadic quotient.  It belongs after the characteristic-zero symmetric-monodromy theorem: it is not a branch-collision lemma for characteristic zero, but an exceptional degeneration in characteristic $19$.  The theorem in this section is computer-assisted: the finite polynomial identities and matrix products are verified by the reproducible Sage certificate in \cref{app:computations}.

\begin{theorem}[computer-assisted sporadic Klein-four quotient]\label{thm:sporadic-V4}
Assume that the SageMath certificate in \cref{app:computations} executes successfully.  Work over $\F_{19}$, and take geometric monodromy after base change to $\barF_{19}$.  For $d=6$, let
\[
        \phi(z)=\frac{z}{z+1},
        \qquad
        \phi^{-1}(x)=\frac{x}{1-x},
        \qquad
        \tau_x=\phi\tau\phi^{-1},
        \qquad
        \tau_x(x)=\frac1{1-x}.
\]
The source-coordinate form of the quotient is
\[
        K_6(x)=h_6(\phi^{-1}(x)),
\]
and it is obtained by reducing
\[
        K_6^{\rm raw}(x)=\frac{1-x^6}{x^6-(1-x)^6}.
\]
After cancellation one has
\[
        K_6(x)=\frac{-x^4-x^3+x+1}{6x^3-9x^2+5x-1}.
\]
This expression changes the source coordinate to $x$ but leaves the target in the original $z$-coordinate.  The conjugated self-map in the $x$-coordinate is
\[
        \widetilde h_6(x)
        =\phi\circ h_6\circ\phi^{-1}(x)
        =\frac{K_6(x)}{K_6(x)+1}
        =\frac{-x^4-x^3+x+1}{-x^4+5x^3-9x^2+6x}.
\]
The four transformations
\[
        \gamma_0(x)=x,
\]
\[
        \gamma_1(x)=\frac{x+6}{2x-1},
        \qquad
        \gamma_2(x)=\frac{x+10}{9x-1},
        \qquad
        \gamma_3(x)=\frac{x-2}{13x-1}
\]
form a Klein four-group of deck transformations of $K_6$, equivalently of $\widetilde h_6$.  Hence $\widetilde h_6$ is a Galois cover of degree four and
\[
        G_{\widetilde h_6}=V_4,
\]
equivalently $G_{h_6}=V_4$ in the original $z$-coordinate.  Moreover,
\[
        \tau_x^{-1}\gamma_1\tau_x=\gamma_2,
        \qquad
        \tau_x^{-1}\gamma_2\tau_x=\gamma_3,
        \qquad
        \tau_x^{-1}\gamma_3\tau_x=\gamma_1.
\]
Consequently the original quotient $h_6$ is a nonsparse $\tau$-twisted exceptional quotient.  Equivalently, the $x$-coordinate self-map $\widetilde h_6$ induces a bijection on
\[
        \PP^1(\barF_{19})^{\tau_x^{-1}\Frob_{19^k}}
\]
for every $k\ge1$.
\end{theorem}

\begin{proof}
Reducing
\[
        K_6^{\rm raw}(x)=\frac{1-x^6}{x^6-(1-x)^6}
\]
modulo $19$, the numerator and denominator have common factor $x^2-x+1$.  Division gives the displayed formula for $K_6$.  Since $\phi(t)=t/(t+1)$, the conjugated self-map in the $x$-coordinate is
\[
        \widetilde h_6(x)=\phi(K_6(x))=\frac{K_6(x)}{K_6(x)+1},
\]
and simplifying gives the displayed formula for $\widetilde h_6$.

The computer-assisted certificate in \cref{app:computations} verifies the required polynomial identities and matrix products.  In particular, it verifies
\[
        K_6\circ\gamma_i=K_6\qquad (i=0,1,2,3),
\]
checks that the four projective classes are pairwise distinct, checks that each nonidentity class has square equal to the identity, and checks closure under multiplication.  Hence these four transformations form a Klein four-group in $\operatorname{PGL}_2(\F_{19})$.  Because $\widetilde h_6=\phi\circ K_6$ differs from $K_6$ only by a target automorphism, the same transformations are deck transformations of $\widetilde h_6$.  Since $|V_4|=\deg K_6=\deg \widetilde h_6=4$, the fixed field $k(x)^{V_4}$ has degree four in $k(x)$ and equals both $k(K_6)$ and $k(\widetilde h_6)$.  Hence the cover is the quotient by this group, its geometric monodromy group is $V_4$, and away from ramification its fibers are exactly $V_4$-orbits.

The displayed conjugation relations are likewise verified in \cref{app:computations} by multiplication in $\operatorname{PGL}_2(\F_{19})$, using $\tau_x(x)=1/(1-x)$.  The same certificate also verifies
\[
        \widetilde h_6\circ\tau_x=\tau_x\circ\widetilde h_6,
\]
which is also the coordinate conjugate of the original relation $h_6\circ\tau=\tau\circ h_6$.  It remains to check that the cyclic permutation of the off-diagonal components gives actual injectivity on the twisted fixed sets, including possible branch-point intersections of the graphs.

Because $19\nmid |V_4|$, the quotient is tame and Galois.  Thus a source point lies over a branch value if and only if its stabilizer in $V_4$ is nontrivial; equivalently, the ramification points are exactly the fixed points of nontrivial deck transformations.  The certificate verifies that the fixed-point equations $x=\gamma_i(x)$ for the three nonidentity deck transformations are, up to nonzero scalar factors, the quadratics
\[
        x^2-x-3,
        \qquad
        x^2+4x+1,
        \qquad
        x^2-6x+6,
\]
respectively.  It also checks that each nontrivial $\gamma_i$ has nonzero lower-left matrix entry, so none fixes $\infty$; hence the three quadratics account for all ramification points.  These quadratics are irreducible over $\F_{19}$.  The branch-point exclusions used below are the following explicit gcd checks in $\F_{19}[x]$:
\[
\begin{array}{c|c|c|c}
\text{deck involution} & f_i(x) & S_i & \text{odd-}k\text{ test polynomial} \\
\hline
\gamma_1 & x^2-x-3 & 1 & (1-x)(1-x)-1 \\
\gamma_2 & x^2+4x+1 & -4 & (-4-x)(1-x)-1 \\
\gamma_3 & x^2-6x+6 & 6 & (6-x)(1-x)-1
\end{array}
\]
Here $S_i$ is the sum of the two roots of $f_i$.  In each row,
\[
        \gcd(f_i,x^2-x+1)=1,
        \qquad
        \gcd\bigl(f_i,(S_i-x)(1-x)-1\bigr)=1.
\]
If $k$ is even, then Frobenius acts trivially on their roots, so the twisted fixed-point equation would force
\[
        x=\tau_x(x),
        \qquad\text{equivalently}\qquad x^2-x+1=0.
\]
This quadratic is coprime to each of the three fixed-point quadratics above.  If $k$ is odd, Frobenius acts on a root of $x^2-Sx+P$ as $x\mapsto S-x$.  The twisted fixed-point equation is then
\[
        S-x=\frac1{1-x}.
\]
For the three values $S=1,-4,6$, the gcd calculation recorded in \cref{app:computations} with the corresponding fixed-point quadratic again gives no common root.  Thus no branch point of the $V_4$-cover lies in any twisted fixed set.

Now suppose $x$ and $y$ are two distinct $\tau_x^{-1}\Frob_{19^k}$-fixed points with $\widetilde h_6(x)=\widetilde h_6(y)$.  Since $\phi$ is a target automorphism, this is equivalent to $K_6(x)=K_6(y)$.  By the fixed-field description above, the fiber away from the checked branch locus is a $V_4$-orbit, so there is a nontrivial $\gamma_i$ with $y=\gamma_i x$.  The graph of $\gamma_i$ is sent by $\tau_x^{-1}\Frob_{19^k}$ to the graph of $\tau_x^{-1}\gamma_i\tau_x$, which is the next nontrivial graph in the displayed three-cycle.  Because the pair $(x,y)$ is itself fixed, it lies in the intersection of two distinct nontrivial graphs.  Hence $x$ is fixed by a nontrivial deck transformation, contradicting the branch-point check.  Therefore no off-diagonal collision occurs on the twisted fixed set.

Since $\widetilde h_6$ is $\tau_x$-equivariant, it maps the finite $\tau_x^{-1}\Frob_{19^k}$-fixed set to itself.  The injectivity just proved is therefore bijectivity.  Conjugating by $\phi$ gives the equivalent $\tau$-twisted exceptionality statement for the original quotient $h_6$ in the $z$-coordinate.
\end{proof}

\begin{corollary}[conditional classification after the sporadic verification]\label{cor:classification-2transitive-range-verified}
Under the hypotheses of \cref{thm:classification-2transitive-range}, and assuming successful execution of the characteristic-$19$ block of the certificate in \cref{app:computations}, the sporadic alternative $(p,d_0)=(19,6)$ is indeed a $\tau$-twisted exceptional quotient.  Hence the theorem's three alternatives consist of the degree-one quotients, the Frobenius-sparse plus branch, and this certificate-verified Klein-four quotient; all quotients outside those alternatives are ruled out by the two-transitive collision obstruction.
\end{corollary}

\begin{proof}
The sporadic assertion is \cref{thm:sporadic-V4} under its certificate-execution hypothesis.  The exclusion of all quotients outside the three alternatives is exactly the conclusion of \cref{thm:classification-2transitive-range}.
\end{proof}

\begin{remark}
The certificate-verified sporadic quotient of \cref{thm:sporadic-V4} shows that the correct high-level classification cannot be a simple dichotomy between the branches $d=p^a\pm1$ and generic $A_n/S_n$ monodromy.  The expected statement is instead the sparse/sporadic/lacunary/generic architecture formulated in \cref{conj:trichotomy}.  Notice also that the sporadic branch is not a degree-two accident: the degree-two sporadic $(p,d)=(11,5)$ has monodromy $S_2$ and is ruled out by the collision obstruction, whereas $(19,6)$ is Galois with monodromy $V_4$.
\end{remark}

\section{Positive-characteristic ramification skeletons and bad reductions}\label{sec:positive-ramification}

The characteristic-zero proof uses Morse separation.  In positive characteristic, critical values may collide even when the quotient is neither sparse nor sporadic.  The correct replacement is a ramification-skeleton analysis: prove primitivity, identify a small-support branch cycle when possible, and then use the collision obstruction to rule out twisted exceptionality.

\begin{remark}[why the global problem is harder]\label{rem:morse-false}
Although critical-value separation is automatic in characteristic zero by \cref{thm:charzero-no-branch-collisions}, it is not true in every positive characteristic outside the Frobenius-sparse exponents $d=p^a\pm1$.  For example, in characteristic $7$ with $d=5$, which is not of the form $7^a\pm1$, one computes
\[
        W(h_5)=-2(z-2)(z+3)(z^2+1)(z^2+z-3)(z^2+2z+2)
\]
and
\[
        B_{5,7}(T)=2(T-3)^2(T-1)^2(T+2)^2(T-2)(T+3).
\]
Thus several simple critical points have collided critical values.  This does not by itself imply small monodromy, but it shows that the corrected generic monodromy conjecture cannot be proved by asserting uniform Morse separation.  The remaining problem is to handle these bad reductions by branch-cycle products, primitivity, and exceptional-cover theory.
\end{remark}

\begin{corollary}[computer-assisted bad branch-value reduction with full monodromy]\label{cor:bad-reduction-7-5}
Assume that the SageMath certificate in \cref{app:computations} executes successfully.  In characteristic $7$ and for $d=5$, the reduced quotient $h_5$ has geometric monodromy
\[
        G_{h_5}=S_5.
\]
Consequently this bad Morse reduction is not $\tau$-twisted exceptional.
\end{corollary}

\begin{proof}
The reproducible Sage certificate in \cref{app:computations}, summarized in Remark~\ref{rem:morse-false}, gives $\deg h_5=5$ and shows that the branch polynomial has two simple roots, namely $T=2$ and $T=-3$.  By \cref{thm:tame-simple-critical}, the corresponding branch points are simple critical values, so the geometric monodromy contains a transposition.  Since the degree is the prime number $5$, the monodromy action is primitive.  A primitive subgroup of $S_5$ containing a transposition is the full symmetric group.  The final assertion follows from \cref{cor:2trans-obstruction}.
\end{proof}

\begin{proposition}[wild Frobenius skeleton]\label{prop:wild-skeleton}
Assume $p>0$, $p\nmid d$, and write
\[
        d-1=p^s e,
        \qquad s\ge1,
        \qquad p\nmid e.
\]
Then
\[
        F_d(z)=F_{e+1}(z)^{p^s},
\]
where $F_r$ denotes the Wronskian skeleton attached to exponent $r$:
\[
        F_r(z)=z^{r-1}-(-1)^r(z+1)^{r-1}\bigl(z^{r-1}+1\bigr).
\]
Consequently the unreduced affine Wronskian skeleton is a $p^s$-fold Frobenius thickening of the tame skeleton for $e+1$.  On the affine $z$-line, the reduced Wronskian divisor is obtained from this skeleton by subtracting the contribution $2\,\operatorname{div}(C_d)$ coming from cancellations, where $C_d=\gcd(N_d,D_d)$.  This affine statement does not include any possible contribution at $z=\infty$, which must be checked separately in a local coordinate when projective ramification is needed.
\end{proposition}

\begin{proof}
Since $d-1=p^s e$, the Frobenius identity gives
\[
        z^{d-1}=(z^e)^{p^s},
        \qquad
        (z+1)^{d-1}=((z+1)^e)^{p^s},
        \qquad
        z^{d-1}+1=(z^e+1)^{p^s}.
\]
For odd $p$, $p^s$ is odd and $(-1)^d=(-1)^{e+1}$; for $p=2$ all signs are $1$.  Hence
\[
        F_d(z)=\left(z^e-(-1)^{e+1}(z+1)^e(z^e+1)\right)^{p^s}=F_{e+1}(z)^{p^s}.
\]
The final assertion is an affine statement on $\mathbb A^1_z$ and follows from \cref{prop:wronskian-skeleton}: before cancellation the skeleton is $F_d$, while after cancellation one has $N_d'D_d-N_dD_d'=C_d^2W(h_d)$.  Thus finite affine divisorial multiplicities in the reduced Wronskian are those of $F_d$ with the $2\,\operatorname{div}(C_d)$ cancellation contribution removed.  The projective point $z=\infty$ is not encoded by this affine polynomial identity and must be handled by a separate local expansion.  In wild characteristic, the order of vanishing of the Wronskian is a different exponent, not automatically the tame ramification index minus one.  The local ramification index must be read from the first nonzero term in the local expansion of $h_d$.
\end{proof}

\begin{lemma}[Jones cycle theorem, applied form]\label{lem:jones-cycle-applied}
Let $G\le S_n$ be a primitive permutation group.  If $G$ contains an $m$-cycle fixing the remaining $n-m$ points, with
\[
        3\le m\le n-3,
\]
then $A_n\le G$.  This is the form of Jones' cycle theorem used below; it is \cite[Corollary 1.3]{Jones2014} applied to a cycle with at least three fixed points.
\end{lemma}

\begin{theorem}[wild small-support criterion]\label{thm:wild-small-support}
Let $p>0$, write $d=p^sd_0$ with $p\nmid d_0$, and let
\[
        h^{\mathrm{sep}}=h_{d_0}
\]
be the separable quotient attached to $h_d$ under the convention of \cref{def:quotient-conventions}.  Put $n=\deg h^{\mathrm{sep}}>1$.  Suppose that the geometric monodromy group $G_{h^{\mathrm{sep}}}$ is primitive and that some branch inertia group contains an element whose cycle decomposition on the $n$ separable sheets is one $m$-cycle and $n-m$ fixed points, with $2\le m\le n-3$.  Then
\[
        A_n\le G_{h^{\mathrm{sep}}}\le S_n.
\]
Consequently $G_{h^{\mathrm{sep}}}$ is $2$-transitive, the separable quotient is not $\tau$-twisted exceptional, and neither is the full Frobenius twist $h_d$.
\end{theorem}

\begin{proof}
If $m=2$, the primitive-transposition theorem gives $G_{h^{\mathrm{sep}}}=S_n$.  If $m\ge3$, then the hypotheses give a primitive group containing an $m$-cycle with $n-m\ge3$ fixed points, so \cref{lem:jones-cycle-applied} applies and gives $A_n\le G_{h^{\mathrm{sep}}}$.  Thus in all cases $A_n\le G_{h^{\mathrm{sep}}}\le S_n$.  The non-exceptionality of $h^{\mathrm{sep}}$ follows from \cref{cor:2trans-obstruction}.  Finally, $h_d=\Frob_{p^s}\circ h^{\mathrm{sep}}$ on values by \cref{def:quotient-conventions}.  Since $\tau$ is defined over $\F_p$, one has $\Frob_{p^s}\Thetaq=\Thetaq\Frob_{p^s}$ for every $q=p^k$; hence Frobenius is a bijection on every twisted fixed set.  Thus composing with this Frobenius factor cannot remove or create off-diagonal collisions, so the full quotient $h_d$ is not $\tau$-twisted exceptional either.
\end{proof}

\begin{remark}[Wild inertia versus Wronskian order]\label{rem:wild-inertia-warning}
In wild characteristic, the vanishing order of the Wronskian does not by itself
determine the ramification index.  For example, in the branch \(d=2p^a+1\) with \(a\ge1\) and \(p\ne2,3\) of
\Cref{sec:first-lacunary}, the Wronskian thickening has order \(2p^a-2\) or
\(2p^a\), while the local ramification index at the two wild points is \(p^a\).
Thus the global monodromy problem is not merely a Morse problem with thicker
critical points; it is a wild-inertia problem.
\end{remark}

\begin{remark}[the smaller nonsparse wild exponent $d=11$]\label{rem:d11-wild}
In characteristic $2$, the exponent $d=11$ is already nonsparse and wild: it is not of the form $2^a-1$ or $2^a+1$, and $d-1$ is divisible by $2$.  A direct computation gives $\deg h_{11}=11$ and
\[
        W(h_{11})=(z^2+z+1)^4(z^3+z+1)^2(z^3+z^2+1)^2.
\]
This example shows why one should not single out $d=13$ as the first nonsparse wild exponent.
\end{remark}

\begin{example}[a nonsparse wild quotient of prime degree]\label{ex:2-13}
Let $p=2$ and $d=13$.  This is not a Frobenius-sparse exponent of the form $2^a-1$ or $2^a+1$.  In characteristic $2$ the reduced quotient has degree $13$, and
\[
        W(h_{13})=(z^3+z+1)^4(z^3+z^2+1)^4.
\]
The reduced skeleton is
\[
        S(z)=(z^3+z+1)(z^3+z^2+1),
\]
and a direct resultant computation gives
\[
        \operatorname{Res}_z(n_{13}(z)-Td_{13}(z),S(z))
        =(T^3+T+1)(T^3+T^2+1),
\]
which is squarefree.  Thus the six skeleton critical points have distinct critical values.  A local expansion at any root of $z^3+z+1$ or $z^3+z^2+1$ begins in degree $4$, so the corresponding ramification index is $4$, not $5$.  Since the global degree is prime, the monodromy group is primitive.  To turn this example into an unconditional alternating/symmetric-monodromy theorem one must compute the local wild inertia action, or invoke a small-support theorem strong enough for the resulting inertia element.  This is a useful test case for the general wild-inertia problem.
\end{example}

\begin{question}[ramification-skeleton form of the generic monodromy problem]\label{q:ramification-skeleton}
Let $p>0$, let $p\nmid d$, let $h_d$ be the reduced separable quotient, put $n=\deg h_d$, and assume $d$ is not a Frobenius-sparse exponent $p^a-1$ or $p^a+1$ with $a\ge1$ up to the small overlaps $d=1,2,3$, and not the sporadic pair $(p,d)=(19,6)$.  In the remaining positive-characteristic cases, determine whether one can prove at least one of the following concrete certificates:
\begin{enumerate}[label=\textup{(\roman*)}]
\item the branch-cycle group of $h_d$ is primitive and contains either a transposition or an element whose cycle decomposition is one $m$-cycle and $n-m$ fixed points for some $2\le m\le n-3$;
\item the branch-cycle group is $2$-transitive by another branch-cycle, local-inertia, or classification argument.
\end{enumerate}
In certificate \textup{(i)}, the transposition subcase would give $G_{h_d}=S_n$ by the primitive-transposition theorem, while the single-cycle subcase would give $A_n\le G_{h_d}\le S_n$ by \cref{thm:wild-small-support}.  Either certificate would rule out $\tau$-twisted exceptionality by \cref{cor:2trans-obstruction}.  This is a programmatic question rather than a classification theorem: in wild Frobenius--lacunary cases, the Wronskian skeleton identifies the critical support but does not by itself determine the Galois-closure inertia.
\end{question}

\begin{remark}
This formulation is deliberately stronger and more precise than merely asking for $2$-transitivity.  It identifies the concrete missing ingredients: primitivity, control of branch-value collisions, and in wild cases the actual local inertia representation.  It also explains the two Frobenius-sparse exclusions.  The branch $d=p^a-1$ collapses to the linear map $\tau^2$, while $d=p^a+1$ collapses to a monomial or an Artin--Schreier map; the sporadic pair $(19,6)$ gives a certificate-verified Klein-four Galois cover.  These are the structural exceptional alternatives.  Outside these branches, the Wronskian skeleton is a source of possible branch-cycle information, but it supplies an alternating/symmetric certificate only after branch-value coalescence and wild local inertia have been controlled.
\end{remark}

\section{Frobenius--lacunary towers}\label{sec:lacunary-towers}

The branch $d=p^a+1$ with $a\ge1$ is only the first member of a broader Frobenius--lacunary phenomenon.  Fix integers $c\ge1$ and $a\ge1$, and put
\[
        d=cp^a+1.
\]
Then the binomial expansion of $H_d^{\rm raw}$ has bounded $p$-adic support as $a$ varies.  The case $c=1$ is the sparse $p^a+1$ branch of \cref{sec:frobenius-sparse}; when $p\ne2,3$, the case $c=2$ is the first nonsparse tower analyzed in \cref{sec:first-lacunary}.

\begin{proposition}[bounded-support lacunary form]\label{prop:bounded-support-lacunary}
Let $p>0$, let $r=p^a$ with $a\ge1$, let $c\ge1$ with $p\nmid c$, and put $d=cr+1$.  In the coordinate
\[
        x=\frac{z}{z+1},
        \qquad
        K_d^{\rm raw}(x)=\frac{(-1)^d-x^d}{x^d-(1-x)^d},
\]
one has the raw identity
\[
        K_{cr+1}^{\rm raw}(x)=
        \frac{(-1)^{cr+1}-x(x^r)^c}
        {x(x^r)^c-(1-x)(1-x^r)^c}.
\]
Thus, before any cancellation or reduction, for fixed $c$ the numerator has at most two ordinary monomial terms and the denominator has at most $2c+2$ ordinary monomial terms as $a$ varies; in particular both raw supports are bounded in terms of $c$.
\end{proposition}

\begin{proof}
The identities
\[
        x^{cr+1}=x(x^r)^c,
        \qquad
        (1-x)^{cr+1}=(1-x)(1-x^r)^c
\]
are immediate from $r=p^a$.  Substituting them into the definition of $K_d^{\rm raw}$ gives the formula.  Expanding $(1-x^r)^c$ gives at most $c+1$ terms, and multiplying by $(1-x)$ gives the union of two adjacent translates, hence at most $2c+2$ ordinary monomial terms.  The numerator has at most two terms before cancellation.
\end{proof}

\begin{corollary}[Frobenius-thickened skeleton in lacunary towers]\label{cor:lacunary-skeleton}
With the hypotheses of \cref{prop:bounded-support-lacunary}, the Wronskian skeleton satisfies
\[
        F_{cr+1}=F_{c+1}^{r}.
\]
In particular, after removing cancellations and poles, the geometric support of the critical divisor is contained in a set whose size is bounded in terms of $c$ and is independent of $a$.
\end{corollary}

\begin{proof}
This is \cref{prop:wild-skeleton} with $d-1=cr$.  Explicitly,
\[
        z^{d-1}=(z^c)^r,
        \qquad
        (z+1)^{d-1}=((z+1)^c)^r,
        \qquad
        z^{d-1}+1=(z^c+1)^r,
\]
and the sign agrees with the sign in $F_{c+1}$ because $r$ is odd for odd $p$, while in characteristic two all signs are trivial.
\end{proof}

\begin{remark}
The corollary explains why the global monodromy problem has a genuinely wild part.  For fixed $c$ and $a\to\infty$, the degree grows like $cp^a$, but the visible critical support remains bounded.  The missing information is not the location of the critical support; it is the wild inertia representation carried by the Frobenius thickening.  The tower $c=2$ is the first place where this representation can be analyzed far enough to prove primitivity and, for $a=1$, alternating or symmetric monodromy.
\end{remark}

\section{The first nonsparse Frobenius--lacunary branch}\label{sec:first-lacunary}

The sparse branch $d=p^a+1$ with $a\ge1$ is not the only place where the binomial expansion of $H_d^{\rm raw}$ collapses under Frobenius.  The next collapse is the nonsparse branch
\[
        d=2p^a+1\quad(a\ge1).
\]
It does not produce cyclic monodromy.  Instead it produces a two-point wild cover: all branch values are concentrated at the two fixed points of the diagonalized order-three automorphism.  This section records the exact normal form.  It is one of the main lessons of the global-monodromy problem: outside the sparse branches, the right dichotomy is not simply ``generic versus exceptional'', but rather ``Frobenius--lacunary wild versus genuinely generic''.

Assume throughout this section that $p\ne2,3$, put $r=p^a$ with $a\ge1$, and let $d=2r+1$.  As in \cref{sec:frobenius-sparse}, use the coordinate $u=M^{-1}(z)$ in which $\tau$ acts as $u\mapsto \omega u$.

\begin{theorem}[normal form for the first nonsparse lacunary branch]\label{thm:two-r-plus-one-normal-form}
Let $k$ be algebraically closed of characteristic $p\ne2,3$, let $r=p^a$ with $a\ge1$, and let $d=2r+1$.  The conjugate quotient
\[
        \Phi_d=M^{-1}\circ h_d\circ M
\]
has the following reduced forms.
\begin{enumerate}[label=\textup{(\roman*)}]
\item If $r\equiv1\pmod3$, then
\[
        \Phi_{2r+1}(u)
        =-\omega^2u^r\frac{u^{r-1}+2}{2u^{r-1}+1},
        \qquad
        \deg \Phi_{2r+1}=2r-1.
\]
\item If $r\equiv2\pmod3$, then
\[
        \Phi_{2r+1}(u)
        =-\omega^2u^{-r}\frac{1+2u^{r+1}}{2+u^{r+1}},
        \qquad
        \deg \Phi_{2r+1}=2r+1.
\]
\end{enumerate}
In both cases the only branch values are $0$ and $\infty$.  The points $u=0$ and $u=\infty$ have local ramification index $r=p^a$, and every other point above $0$ or $\infty$ is unramified.
\end{theorem}

\begin{proof}
Use the filtered normal form of \cref{prop:filtered-normal-form}.  Since
\[
        (1-u)^{2r+1}=(1-u)(1-u^r)^2
        =1-u-2u^r+2u^{r+1}+u^{2r}-u^{2r+1},
\]
only the six displayed exponents can occur.

If $r\equiv1\pmod3$, then $d\equiv0\pmod3$, so the numerator filter has residues congruent to $2$ and the denominator filter has residues congruent to $1$.  Writing $T=u^3$, this gives
\[
        A_{d,2}(T)=2T^{(r-1)/3}+T^{2(r-1)/3},
        \qquad
        A_{d,1}(T)=-1-2T^{(r-1)/3}.
\]
Substitution in \cref{prop:filtered-normal-form} gives
\[
        \Phi_{2r+1}(u)
        =-\omega^2u^r\frac{u^{r-1}+2}{2u^{r-1}+1}.
\]
There is no common factor between numerator and denominator, and the degree is $2r-1$.

If $r\equiv2\pmod3$, then $d\equiv2\pmod3$.  The numerator filter has residue $0$ and the denominator filter has residue $2$, giving
\[
        A_{d,0}(T)=1+2T^{(r+1)/3},
\]
and
\[
        A_{d,2}(T)=-2T^{(r-2)/3}-T^{(2r-1)/3}.
\]
Again \cref{prop:filtered-normal-form} gives the stated expression
\[
        \Phi_{2r+1}(u)
        =-\omega^2u^{-r}\frac{1+2u^{r+1}}{2+u^{r+1}}.
\]
Its degree is $2r+1$.

It remains to check ramification.  In the first case, ignoring the nonzero scalar $-\omega^2$, write
\[
        f(u)=u^r\frac{u^{r-1}+2}{2u^{r-1}+1}.
\]
Since $r=0$ and $r-1=-1$ in the ground field, differentiation on the affine line gives
\[
        f'(u)=\frac{3u^{2r-2}}{(2u^{r-1}+1)^2}.
\]
Thus the only finite affine critical point away from poles is $u=0$.  Locally $f(u)=2u^r+O(u^{2r-1})$, so its ramification index at $0$ is $r$.  At infinity, in the coordinate $v=1/u$, one has $f(u)=\frac12u^r+O(u)$, so the ramification index at infinity is again $r$.  The remaining points over $0$ and $\infty$ are the simple roots of $u^{r-1}+2$ and $2u^{r-1}+1$, respectively, and are unramified.

In the second case, write
\[
        f(u)=u^{-r}\frac{1+2u^{r+1}}{2+u^{r+1}}.
\]
The derivative of the factor $u^{-r}$ vanishes in characteristic $p$, while $(u^{r+1})'=u^r$.  Hence
\[
        f'(u)=\frac{3}{(2+u^{r+1})^2}
\]
away from poles.  Thus there are no affine critical points away from the poles.  At $u=0$ the function has a pole of order $r$, so the local ramification index is $r$.  At infinity, with $v=1/u$, one has $f(u)=2u^{-r}+O(u^{-2r-1})$, equivalently $f=v^r(2+O(v^{r+1}))$, so the local ramification index at infinity is also $r$.  The other zeros and poles are simple.  Hence the only branch values are $0$ and $\infty$.
\end{proof}

\begin{lemma}[pole-divisor indecomposability criterion]\label{lem:pole-divisor-indecomp}
Let $f:\PP^1\to\PP^1$ be a nonconstant rational function of degree $n=2r-1$ or $n=2r+1$, with $r>1$ and $\gcd(r,n)=1$.  Suppose that the pole divisor of $f$ has the form
\[
        \operatorname{div}_\infty(f)=rP+Q_1+\cdots+Q_{n-r},
\]
where the $Q_i$ are distinct and different from $P$.  Then $f$ is indecomposable over the algebraic closure.  If $f$ is separable, then its geometric monodromy action is primitive.
\end{lemma}

\begin{proof}
Suppose, to the contrary, that $f=g\circ h$ with
\[
        \deg h=m>1,
        \qquad
        \deg g=\ell>1.
\]
Let $R$ be a pole of $g$ of multiplicity $A$.  If $S\in h^{-1}(R)$ has ramification index $B=e_h(S)$, then $S$ is a pole of $f$ of multiplicity $AB$.

All pole multiplicities of $f$ are equal to $1$, except for the unique pole $P$ of multiplicity $r$.  Let $R_0=h(P)$ and let $A$ be the pole multiplicity of $g$ at $R_0$; write $B=e_h(P)$.  Then
\[
        AB=r.
\]
If $A>1$, then every point in $h^{-1}(R_0)$ would give a pole of $f$ of multiplicity at least $A$.  Since $P$ is the only nonsimple pole, $P$ is the only point over $R_0$.  Hence $m=B=r/A$ divides $r$.  But $m$ also divides $n=\deg f$, and $\gcd(r,n)=1$, so $m=1$, contradiction.

Thus $A=1$ and $B=r$.  Therefore $m\ge r$.  Since $\ell\ge2$, one has
\[
        m\le \frac n2.
\]
If $n=2r-1$, then $n/2<r$, contradiction.  If $n=2r+1$, then $m\le r$; hence $m=r$, but $r\nmid n$ by hypothesis.  This is again impossible.  Therefore no nontrivial decomposition exists.

Finally assume $f$ is separable.  For a separable cover of curves, decomposability of the rational function is equivalent to imprimitivity of the geometric monodromy action: an intermediate rational function is the same thing as an intermediate function field between $k(f)$ and $k(u)$, hence a nontrivial block system.  Thus the monodromy action is primitive.
\end{proof}

\begin{theorem}[primitivity of the first lacunary tower]\label{thm:first-lacunary-primitive}
Let $p\ne2,3$, let $r=p^a$ with $a\ge1$, and let $d=2r+1$.  Then the reduced quotient $h_d$ is indecomposable.  Consequently its geometric monodromy group is primitive.
\end{theorem}

\begin{proof}
Use the normal forms of \cref{thm:two-r-plus-one-normal-form}.  If $r\equiv1\pmod3$, then, up to a nonzero scalar,
\[
        \Phi_{2r+1}(u)=u^r\frac{u^{r-1}+2}{2u^{r-1}+1}
\]
has degree $n=2r-1$.  Its pole divisor is
\[
        r[\infty]+\sum_{2\alpha^{r-1}+1=0} [\alpha].
\]
The roots of $2u^{r-1}+1$ are distinct because $p\nmid r-1$.  Hence the pole divisor has one pole of multiplicity $r$ and $r-1=n-r$ simple poles.

If $r\equiv2\pmod3$, then, again up to a nonzero scalar,
\[
        \Phi_{2r+1}(u)=u^{-r}\frac{1+2u^{r+1}}{2+u^{r+1}}
\]
has degree $n=2r+1$.  Its pole divisor is
\[
        r[0]+\sum_{2+\alpha^{r+1}=0} [\alpha],
\]
and the roots of $2+u^{r+1}$ are distinct because $p\nmid r+1$.  Thus the same pole-divisor pattern holds, with $n-r=r+1$ simple poles.

In both cases $\gcd(r,n)=1$.  The derivative computations in \cref{thm:two-r-plus-one-normal-form} show that the normal forms are separable, hence so is $h_d$.  Applying \cref{lem:pole-divisor-indecomp} proves indecomposability, and its separable monodromy clause gives primitivity.
\end{proof}

\begin{theorem}[large monodromy in the prime-exponent lacunary branch]\label{thm:two-p-plus-one-large-monodromy}
Let $p\ne2,3$, let $d=2p+1$, and let $h_d$ be the reduced quotient.  Put
\[
        n=\deg h_d=
        \begin{cases}
        2p-1,&p\equiv1\pmod3,\\
        2p+1,&p\equiv2\pmod3.
        \end{cases}
\]
Then the geometric monodromy group satisfies
\[
        A_n\le G_{h_d}\le S_n.
\]
Consequently $h_d$ is not $\tau$-twisted exceptional.
\end{theorem}

\begin{proof}
By \cref{thm:first-lacunary-primitive}, $G_{h_d}$ is primitive.  By \cref{thm:two-r-plus-one-normal-form}, over one of the two branch values there is exactly one ramified point of ramification index $p$, while all other points over that branch value are unramified.  Since $p$ is prime, the final clause of \cref{lem:inertia-orbit-lengths} gives an actual $p$-cycle in the inertia image on the sheets specializing to the ramified point, and this element fixes the remaining $n-p$ sheets.  Here $n-p\ge3$ for $p\ge5$.

Now apply \cref{thm:wild-small-support} with $m=p$.  This gives $A_n\le G_{h_d}\le S_n$, and the final assertion follows from the two-transitive collision obstruction, \cref{cor:2trans-obstruction}.
\end{proof}

\begin{example}
The first cases of \cref{thm:two-p-plus-one-large-monodromy} include
\[
        (p,d,n)=(5,11,11),\qquad (7,15,13),\qquad (11,23,23),\qquad (13,27,25).
\]
These exponents are not of the form $p^a-1$ or $p^a+1$ with $a\ge1$, and they are not the Klein-four sporadic pair $(19,6)$.  They are instead the first members of a nonsparse Frobenius--lacunary wild branch.  Their large monodromy follows from the exact wild normal form, the pole-divisor primitivity theorem, and Jordan's theorem, not from tame Morse separation.
\end{example}

\begin{remark}
For $a>1$, \cref{thm:first-lacunary-primitive} proves primitivity, so the obstacle is no longer decomposability.  The same normal form gives a unique wild ramified point of index $p^a$ over each of the branch values $0$ and $\infty$.  In the Galois closure, the wild inertia subgroup acts transitively on the corresponding $p^a$ local branches as a $p$-subgroup; the full inertia group may also have a tame prime-to-$p$ quotient.  This transitive $p$-subgroup need not contain a single $p^a$-cycle.  Thus Jones' cycle theorem cannot be applied blindly.  The remaining problem in this tower is the explicit local inertia representation.  If that representation contains a cycle with at least three fixed sheets in its global action, then \cref{thm:wild-small-support} and primitivity immediately give $A_n\le G\le S_n$.
\end{remark}

\begin{proposition}[local different in the higher lacunary tower]\label{prop:higher-lacunary-different}
Let $p\ne2,3$, let $r=p^a$ with $a\ge1$, and let $d=2r+1$.  Write $f$ for the normal-form map $\Phi_{2r+1}$ of \cref{thm:two-r-plus-one-normal-form}, with its harmless nonzero scalar factor suppressed.  At each of the two wild points in that theorem, the completed local extension has ramification index $r$.  Its different exponent is
\[
        2r-2\quad\text{if }r\equiv1\pmod3,
        \qquad
        2r\quad\text{if }r\equiv2\pmod3.
\]
Equivalently, take $t=f$ at a ramified point above the branch value $0$, take $t=1/f$ at a ramified point above the branch value $\infty$, and use the source local parameter $u$ at a finite source point and $v=1/u$ at source infinity.  If $u_{\rm loc}$ denotes the chosen source local parameter, then
\[
        v_{u_{\rm loc}}\left(\frac{dt}{du_{\rm loc}}\right)=
        \begin{cases}
        2r-2,&r\equiv1\pmod3,\\
        2r,&r\equiv2\pmod3.
        \end{cases}
\]
\end{proposition}

\begin{proof}
Consider first the case $r\equiv1\pmod3$.  Up to a nonzero scalar, the local expression at $u=0$ is
\[
        t=f(u)=u^r\frac{u^{r-1}+2}{2u^{r-1}+1}.
\]
The expansion begins with $2u^r$, so the ramification index is $r$.  The derivative computed in \cref{thm:two-r-plus-one-normal-form} is
\[
        f'(u)=\frac{3u^{2r-2}}{(2u^{r-1}+1)^2},
\]
so $v_u(f')=2r-2$.  At $u=\infty$, set $v=1/u$ and use the target parameter $t=1/f$.  Then
\[
        t=v^r\frac{2+v^{r-1}}{1+2v^{r-1}},
\]
and differentiation gives valuation $2r-2$ again.

If $r\equiv2\pmod3$, the normal form is
\[
        f(u)=u^{-r}\frac{1+2u^{r+1}}{2+u^{r+1}}
\]
up to a nonzero scalar.  At $u=0$ the branch value is $\infty$, so use $t=1/f$:
\[
        t=u^r\frac{2+u^{r+1}}{1+2u^{r+1}}.
\]
Since $r=0$ and $r+1=1$ in the ground field, differentiating gives
\[
        \frac{dt}{du}=\frac{-3u^{2r}}{(1+2u^{r+1})^2},
\]
so the different exponent is $2r$.  At $u=\infty$, with $v=1/u$ and target parameter $t=f$, one obtains
\[
        t=v^r\frac{2+v^{r+1}}{1+2v^{r+1}},
\]
and the same derivative valuation $2r$.  In each of the four local forms just displayed, $t$ has the form $u^r$ times a unit, or the same expression in the parameter $v=1/u$.  Hence $k[[u]]$ is integral and finite over $k[[t]]$: equivalently, by Weierstrass preparation, the equation $f(U)-t=0$ gives a distinguished monic polynomial for the chosen uniformizer over $k[[t]]$ after multiplying by a unit.  The extension is separable because the displayed derivative is not identically zero.  Thus the completed source ring is monogenic over the completed base, and
\[
        \Omega_{k[[u]]/k[[t]]}
        \cong k[[u]]/(dt/du)\,du.
\]
The standard monogenic different formula for complete discrete valuation rings, applied to this finite separable extension, therefore gives the different exponent as $v_u(dt/du)$; see, for example, \cite[Chapter III, \S6]{SerreLocalFields}.  This proves the claim.
\end{proof}

\begin{remark}[why the Jones-cycle certificate is still missing]\label{rem:Jones-cycle-still-missing}
Proposition~\ref{prop:higher-lacunary-different} is a useful obstruction to overclaiming.  The local ramification index $p^a$ does not by itself produce a $p^a$-cycle in wild characteristic.  The wild inertia subgroup in the Galois closure is a transitive $p$-subgroup on the $p^a$ local sheets, while the full inertia group can also have a tame prime-to-$p$ quotient; a transitive $p$-group of degree $p^a$ need not contain a single $p^a$-cycle.  The regular elementary abelian group is the basic counterexample when $a>1$.  Thus even the exact local different exponent does not certify cyclic wild inertia.  Therefore the proposed shortcut
\[
        \text{ramification index }p^a \quad\Longrightarrow\quad
        \text{Jones cycle}
\]
is invalid for $a>1$.  A proof of alternating or symmetric monodromy in the higher lacunary tower must compute the actual Galois-closure inertia or supply a different two-transitivity certificate.
\end{remark}

\section{The remaining wild-inertia problem}\label{sec:questions}

The remaining global problem concerns positive-characteristic monodromy outside the resolved sparse, sporadic, and lacunary branches.  The auxiliary fixed-degree, higher-dimensional, and full-field lifting results are recorded in the appendices, so the main text can close with the precise monodromy obstruction that remains.  The paper now proves full symmetric monodromy in characteristic zero for every non-linear quotient, proves full symmetric monodromy in the tame Morse range over positive characteristic, proves full symmetric monodromy in the bad reduction $(p,d)=(7,5)$ conditionally on the characteristic-$7$ certificate block, proves primitivity throughout the first nonsparse lacunary tower $d=2p^a+1$ with $a\ge1$ for $p\ne2,3$, and proves alternating or symmetric monodromy in the prime-exponent subbranch $d=2p+1$ for $p\ne2,3$.  The remaining issue is substantive: by Proposition~\ref{prop:higher-lacunary-different} and Remark~\ref{rem:Jones-cycle-still-missing}, the higher lacunary tower has explicitly computable wild local different, and the ramification index $p^a$ alone does not furnish the Jones cycle needed for the alternating/symmetric conclusion.

Thus the concrete assertion one would like to prove is the following.  We record the remaining issue as a concrete problem.

\begin{question}[remaining wild-inertia or two-transitivity certificate]
Let $p>0$, let $p\nmid d$, let $h_d$ be the reduced quotient, put $n=\deg h_d$, and assume that $d$ is not one of the two Frobenius-sparse branches $p^a-1$ or $p^a+1$ with $a\ge1$, not one of the small overlaps $d=1,2,3$, and not the Klein-four sporadic pair $(p,d)=(19,6)$.  If $h_d$ is not already covered by the tame Morse theorem, by the certificate-dependent bad-reduction theorem $(p,d)=(7,5)$, or by the prime-exponent lacunary theorem, prove one of the following two certificates:
\begin{enumerate}[label=\textup{(\roman*)}]
\item $G_{h_d}$ is primitive and contains a transposition or an element whose cycle decomposition is one $m$-cycle and $n-m$ fixed points for some $2\le m\le n-3$; or
\item $G_{h_d}$ is $2$-transitive by some other branch-cycle, local-inertia, or classification argument.
\end{enumerate}
In certificate \textup{(i)}, the transposition subcase would give $G_{h_d}=S_n$ by the primitive-transposition theorem, while the single $m$-cycle subcase with $m\ge3$ would give $A_n\le G_{h_d}\le S_n$ by Jones' cycle theorem.  Either certificate would then rule out $\tau$-twisted exceptionality by \cref{cor:2trans-obstruction}.  For $p\ne2,3$, the higher lacunary tower $d=2p^a+1$ with $a>1$ is the first test case where certificate \textup{(i)} is not presently justified by the known local normal form.
\end{question}

\appendix

\section*{Appendices. Supplemental constructions and computational certificates}

\section{Attribution and comparison with Ding--Song--Xiong}\label{sec:dsx-comparison}

This section is included to clarify attribution.  All descent results used in the proofs above have been proved internally; the comparison below is only for attribution and for recovery of known branches.  Ding--Song--Xiong \cite{DSX2026} prove the two-step descent and use it to construct explicit sparse trinomials, pentanomials, and seven-term polynomials over $\F_{q^3}$.  The descent propositions in \cref{sec:descent} are reformulations of their method in Hilbert--90 notation, and the $Q+1$ branch below is their theorem, not a new result of this paper.  No proof below depends on an external DSX theorem number; the references in this section are for attribution and comparison.

Their trace-zero inputs are:
\begin{enumerate}[label=\textup{(\roman*)}]
\item for $Q=p^\ell$, the map $(X^q-X)\circ X^{Q+1}$ permutes $\Gamma_q$ if and only if $\gcd(q-1,Q+1)=1$ \cite[Theorem 1.6]{DSX2026};
\item the map $(X^q-X)\circ X^3$ permutes $\Gamma_q$ if and only if $q\equiv2\pmod3$ \cite[Theorem 1.7]{DSX2026}.
\end{enumerate}
Combined with the additive fiber criterion, these yield their full-field families \cite[Theorems 1.1, 1.4, 1.8]{DSX2026}.  For example,
\[
        (X^q-X)^{Q+1}
        =X^{qQ+q}-X^{qQ+1}-X^{Q+q}+X^{Q+1},
\]
and adding trace terms of the form $\Tr(cX^{R+S})$ or $\Tr(cX^R)$ gives their sparse specializations after reduction modulo $X^{q^3}-X$ and, in some cases, Frobenius conjugation.  We do not reproduce the term-by-term derivation of every named sparse corollary here; the purpose of the present paper is instead to isolate the fixed-exponent quotient formula $H_d^{\rm raw}$ and reduced map $h_d$ and prove the torsion-defect rigidity theorems that produce the Mersenne branch.

The overlap is therefore as follows.  The case $d=3$ in \cref{thm:trace-zero-classification} recovers the Ding--Song--Xiong cubic trace-zero theorem.  The $Q+1$ trace-zero branch is recovered in \cref{cor:Qplus1-branch} from the cyclic Frobenius-sparse quotient normal form.  The new trace-zero family isolated by the torsion-defect analysis is the characteristic-two Mersenne branch $d=2^a-1$ with $a\ge1$, subject to $\gcd(a,k)=1$ and $3\nmid a$.

\section{Fixed positive-characteristic quotient-degree strata}\label{app:fixed-degree-strata}

The main text uses only the degree-one, degree-two, and characteristic-zero quotient strata.  For completeness, and for computational work in fixed positive characteristic, we record the general finite linear test and the automatic-language consequence.  The automata result is intended as an algorithmic appendix, not as a closed-form classification of the accepted languages.

For positive characteristic, a uniform closed-form classification for every fixed $m\ge3$ would require controlling near-maximal intersections of the toric line $1+X+Y=0$ with cyclic torsion subgroups.  The following finite determinantal test is the form in which the problem is actually used in computations; it is the higher-degree analogue of the coefficient comparisons in \cref{sec:rigidity}.

\begin{proposition}[finite linear test for a fixed quotient stratum]\label{prop:fixed-m-linear-test}
Let $k$ be any field, let
\[
        N_d(z)=(-1)^d(z+1)^d-z^d=\sum_{j=0}^d n_jz^j,
        \qquad
        D_d(z)=z^d-1,
\]
and fix $m\ge0$.  Then the reduced quotient $h_d$ associated with $H_d^{\rm raw}=N_d/D_d$ has degree at most $m$ if and only if there exist polynomials
\[
        P(z)=\sum_{i=0}^m p_iz^i,
        \qquad
        Q(z)=\sum_{i=0}^m q_iz^i,
\]
with $Q\ne0$, satisfying
\begin{equation}\label{eq:fixed-m-linear-system}
        \sum_{i=0}^m q_i n_{j-i}=p_{j-d}-p_j
        \qquad(0\le j\le d+m),
\end{equation}
where $n_t=0$ for $t\notin[0,d]$ and $p_t=0$ for $t\notin[0,m]$.  Equivalently, the coefficient vector $(p_0,\ldots,p_m,q_0,\ldots,q_m)$ lies in the kernel of an explicit $(d+m+1)\times(2m+2)$ matrix over $k$ and has at least one nonzero $q_i$.

The condition $\deg h_d=m$ is obtained by imposing the degree-at-most-$m$ condition and excluding the analogous kernels for $m-1$.
\end{proposition}

\begin{proof}
The reduced quotient has degree at most $m$ precisely when it can be represented as $P/Q$ with $\deg P,\deg Q\le m$ and $Q\not=0$.  Since $H_d^{\rm raw}=N_d/D_d$ as a rational function, this is equivalent to
\[
        Q(z)N_d(z)=P(z)D_d(z)=P(z)(z^d-1).
\]
Comparing coefficients of $z^j$ gives \eqref{eq:fixed-m-linear-system}.  Conversely, any solution with $Q\ne0$ gives the displayed polynomial identity and hence a representative of degree at most $m$ after removing any common factor of $P$ and $Q$.  A solution with $Q=0$ would force $P(z)(z^d-1)=0$, hence $P=0$ in the polynomial ring; thus no nonzero kernel vector is lost by requiring $Q\ne0$.  The last assertion is immediate.
\end{proof}

\begin{remark}
\Cref{prop:fixed-m-linear-test} is not a cosmetic reformulation.  For fixed $m$ the interior equations say that the binomial-coefficient string of $N_d$ satisfies a linear recurrence of order at most $m$ away from the two ends.  The cases $m=1$ and $m=2$ are exactly the recurrences solved in \cref{thm:degree-one,thm:degree-two}.  For $m\ge3$ the same system separates genuine positive-characteristic sporadic cancellations from the characteristic-zero strata of \cref{thm:charzero-degree-strata}.
\end{remark}

\begin{theorem}[automatic decision procedure for fixed-degree strata]\label{thm:automatic-fixed-degree}
Fix a prime $p$ and an integer $m\ge0$.  Let
\[
        \mathcal D_{p,\le m}=\{d\ge1:p\nmid d,\ \deg h_d\le m\text{ over }\bar{\mathbb F}_p\}.
\]
Then the base-$p$ expansions of the integers in $\mathcal D_{p,\le m}$ form an effectively computable regular language.  Equivalently, there is an explicit finite automaton, depending only on $p$ and $m$, which decides whether $\deg h_d\le m$ from the base-$p$ digits of $d$.  Consequently the exact stratum
\[
        \mathcal D_{p,m}=\{d:p\nmid d,\ \deg h_d=m\}
\]
is also effectively regular.

In particular, the tame positive-characteristic fixed-degree problem is not an infinite search over $d$: for every fixed $(p,m)$ with $p\nmid d$ imposed, it reduces to minimization and inspection of a finite automaton.
\end{theorem}

\begin{proof}
We use least-significant-digit first base-$p$ words, read in parallel with zero padding.  Integers are represented canonically by finite words with no trailing zero in this convention, except that $0$ is represented by the one-letter word $0$; allowing additional padded zeros does not change any recognized relation and is used for parallel reading of tuples such as $(d,j,t)$.

By \cref{prop:fixed-m-linear-test}, the condition $\deg h_d\le m$ over $\bar{\mathbb F}_p$ is equivalent to the existence, after scalar extension to $\bar{\mathbb F}_p$, of a kernel vector whose $Q$-part is nonzero.  Since the matrix entries lie in $\mathbb F_p$, both the full matrix rank and the rank after imposing $q_0=\cdots=q_m=0$ are unchanged by field extension; therefore such a vector exists over $\bar{\mathbb F}_p$ if and only if one exists over $\mathbb F_p$.  Thus it is enough to test the finitely many coefficient vectors
\[
        (p_0,\ldots,p_m,q_0,\ldots,q_m)\in\mathbb F_p^{2m+2}
\]
with nonzero $Q$-part satisfying the coefficient equations
\[
        \sum_{i=0}^m q_i n_{j-i}=p_{j-d}-p_j
        \qquad(0\le j\le d+m),
\]
where
\[
        n_t=(-1)^d\binom dt-\mathbf 1_{t=d}
\]
with the convention $n_t=0$ for $t\notin[0,d]$.  Since the vector space $\mathbb F_p^{2m+2}$ is finite, we may test the finitely many coefficient vectors with nonzero $Q$-part one at a time and take their union.

For a fixed such coefficient vector, each displayed equation is a first-order condition in the variables $d$ and $j$ over the natural numbers with addition, order, base-$p$ digit predicates, and the relation
\[
        \binom dt\equiv c\pmod p
        \qquad(c\in\mathbb F_p).
\]
Lucas' theorem expresses this relation as
\[
        \binom dt\equiv
        \prod_{\nu\ge0}\binom{d_\nu}{t_\nu}\pmod p,
\]
where $d=\sum d_\nu p^\nu$ and $t=\sum t_\nu p^\nu$.  Thus the relation is recognized by a finite automaton reading the base-$p$ expansions of $d$ and $t$ in parallel: the automaton keeps only the running product in $\mathbb F_p$ and moves to a zero state if some $t_\nu>d_\nu$.  The sign $(-1)^d$, the boundary tests $t=d$, $t\in[0,d]$, and the shifted symbols $p_{j-d}$ and $p_j$ are also finite-state conditions.  More explicitly, because $m$ is fixed, the alternatives $j-d=t$ with $0\le t\le m$, $j-i=t$ with $0\le i\le m$, and $0\le j\le d+m$ are recognized by the standard addition automata for base-$p$ words with constants and order.  Thus the variable upper bound and all shifted indices in \eqref{eq:fixed-m-linear-system} are encoded inside the same automatic structure; no unbounded memory is required beyond the usual carry state.

Operationally, for a fixed such coefficient vector the finite state records only: the vector itself, the carries for the additions and subtractions defining $j-d$, $j-i$, and $d+m$, the running Lucas product in $\mathbb F_p$ for each binomial coefficient queried by the equation, the comparison states for endpoint tests such as $0\le j\le d+m$, and the single-field-element accumulator for the coefficient equation.  This makes the asserted effectivity a finite construction rather than an abstract existence statement.

For a fixed such coefficient vector, build the automatic relation $E(d,j)$ asserting that the coefficient equation at index $j$ holds.  The bounded universal quantifier over all $j$ with $0\le j\le d+m$ is implemented by first recognizing the existential counterexample relation $\exists j\,(0\le j\le d+m\ \wedge\ \neg E(d,j))$, then projecting away $j$, and finally complementing the resulting language of bad $d$.  These closure properties are standard in the first-order theory of automatic structures over $(\mathbb N,+,<,V_p)$ and finite automata over numeration systems; see B\"uchi \cite{Buchi1960} and Allouche--Shallit \cite[Chapters 5--6]{AlloucheShallit}.  Therefore, for each fixed such coefficient vector, the set of $d$ satisfying all equations is regular.  A finite union over the finitely many coefficient vectors with nonzero $Q$-part gives a regular language for the exponents with $\deg h_d\le m$ before imposing the tame condition.  Intersect this language with the regular language of canonical least-significant-digit-first base-$p$ words for positive integers whose first digit is nonzero.  Because the words are read least-significant digit first, this first digit is the residue of $d$ modulo $p$; requiring it to be nonzero imposes exactly $p\nmid d$ and gives the regular language for $\mathcal D_{p,\le m}$.  Finally,
\[
        \mathcal D_{p,m}=\mathcal D_{p,\le m}\setminus \mathcal D_{p,\le m-1},
\]
with the evident convention for $m=0$, so the exact-degree stratum is regular as well.
\end{proof}

\begin{remark}
The theorem is intentionally an algorithmic consequence of Lucas' theorem and automata closure properties, not a closed-form human-readable classification for arbitrary $m$.  It is stronger than checking one exponent at a time: it produces, for each fixed $(p,m)$, a finite object whose accepted language is precisely the desired infinite set of exponents.  A closed-form list may still be preferable for small $m$; the degree-one and degree-two theorems are exactly such hand-minimized automata.
\end{remark}

\section{Higher-dimensional Hilbert--90 quotients}\label{sec:all-trace-zero-dimensions}

The cubic quotient used throughout the paper is the one-dimensional member of a general projective Hilbert--90 quotient.  This section gives the promised extension from the trace-zero plane to arbitrary trace-zero dimension.  For $n=3$ it recovers the curve $\Lambda_q$, the raw formula $H_d^{\rm raw}$, and the reduced map $h_d$ after choosing an affine coordinate.  The degenerate case $n=1$ is excluded throughout this section: then $\Gamma_{1,q}=\{0\}$ and every $P_d$ permutes it trivially, so the quotient criterion below would have no meaningful projective quotient.

Assume $n\ge2$.  Let $L=\F_{q^n}$, $K=\F_q$, and let $\sigma(x)=x^q$.  Put
\[
        \Gamma_{n,q}=\ker\Tr_{L/K}.
\]
Let
\[
        U_n=\left\{[X_0:\cdots:X_{n-1}]\in\PP^{n-1}:
        \sum_{i=0}^{n-1}X_i=0,
        \prod_{i=0}^{n-1}X_i\ne0\right\}.
\]
Thus $U_n$ is the complement of the coordinate hyperplanes in the projective trace-zero hyperplane, and $\dim U_n=n-2$.  Let
\[
        \rho[X_0:X_1:\cdots:X_{n-1}]=[X_1:X_2:\cdots:X_{n-1}:X_0]
\]
be the cyclic shift, and put $\Theta_{n,q}=\rho^{-1}\Frob_q$.

\begin{proposition}[projective Hilbert--90 quotient in dimension $n$]\label{prop:h90-all-n}
The map
\[
        \lambda_n:\Gamma_{n,q}^*/K^*\longrightarrow U_n(\barF_q)^{\Theta_{n,q}},
        \qquad
        x\longmapsto [x:\sigma x:\cdots:\sigma^{n-1}x],
\]
is a bijection.
\end{proposition}

\begin{proof}
If $x\in\Gamma_{n,q}^*$, then all its conjugates are nonzero and their sum is zero, so the displayed point lies in $U_n$.  Frobenius sends it to its cyclic shift, hence it is fixed by $\Theta_{n,q}$.  Multiplication of $x$ by an element of $K^*$ does not change the projective point.

The map is injective on $K^*$-orbits.  Indeed, if
\[
        [x:\sigma x:\cdots:\sigma^{n-1}x]
        =[y:\sigma y:\cdots:\sigma^{n-1}y],
\]
then $y=cx$ for some $c\in L^*$.  Comparing the second coordinates gives $\sigma(c)=c$, hence $c\in K^*$.

For surjectivity, let $P=[X_0:\cdots:X_{n-1}]\in U_n(\barF_q)^{\Theta_{n,q}}$.  Since $\Theta_{n,q}^n=\Frob_{q^n}$, the point $P$ is defined over $\F_{q^n}$ after a projective scaling.  Choose coordinates $X_i\in\F_{q^n}^*$ for $P$.  The relation $\Frob_q(P)=\rho(P)$ gives a scalar $c\in\F_{q^n}^*$ with
\[
        X_i^q=cX_{i+1}
        \qquad(i\bmod n).
\]
Iterating around the cycle gives $\Nm_{\F_{q^n}/\F_q}(c)=1$.  By multiplicative Hilbert 90 there is $b\in\F_{q^n}^*$ with $b^{q-1}=c^{-1}$.  Replacing $X_i$ by $bX_i$ makes the scalar equal to $1$, so the new coordinates satisfy $X_i^q=X_{i+1}$.  Thus they are
\[
        [x:\sigma x:\cdots:\sigma^{n-1}x]
\]
with $x=X_0$.  Since the coordinates sum to zero, $x\in\Gamma_{n,q}^*$.
\end{proof}

For the fixed-exponent trace-zero map
\[
        P_d(x)=\sigma(x)^d-x^d,
\]
the projective quotient is the rational self-map of the trace-zero hyperplane
\begin{equation}\label{eq:Psi-n-d}
        \Psi_{n,d}([X_0:\cdots:X_{n-1}])
        =[X_1^d-X_0^d:X_2^d-X_1^d:\cdots:X_0^d-X_{n-1}^d].
\end{equation}
The target coordinates sum to zero, and $\Psi_{n,d}$ commutes with the cyclic shift $\rho$.

\begin{proposition}[higher-dimensional torsion base locus]\label{prop:higher-torsion-base}
After base change to an algebraically closed field $k$ of characteristic $p$, the base locus of $\Psi_{n,d}$ on the projective trace-zero hyperplane is the finite torsion scheme
\[
        B_{n,d}\cong
        \left\{(u_1,\ldots,u_{n-1})\in\boldsymbol\mu_d^{\,n-1}:
        1+u_1+\cdots+u_{n-1}=0\right\},
\]
where the isomorphism is obtained from the affine chart $X_0=1$ by setting $u_i=X_i/X_0$.
\end{proposition}

\begin{proof}
Let
\[
        S=k[X_0,\ldots,X_{n-1}]/(X_0+\cdots+X_{n-1})
\]
and let $I\subset S$ be the homogeneous ideal generated by the coordinates of \eqref{eq:Psi-n-d}, equivalently by
\[
        X_1^d-X_0^d,\ X_2^d-X_1^d,\ldots,\ X_0^d-X_{n-1}^d.
\]
The base locus is $\operatorname{Proj}(S/I)$, with the usual saturation by the irrelevant ideal.  If a homogeneous prime $\mathfrak p$ in this Proj contains one coordinate $X_i$, then the equations in $I$ force $X_{i+1},X_{i+2},\ldots$ cyclically into $\mathfrak p$; hence all coordinates lie in $\mathfrak p$, contradicting the condition that $\mathfrak p$ is a point of Proj.  Thus the support of the saturated base locus is contained in the torus open $X_0\cdots X_{n-1}\ne0$.  Let $J$ be the saturation of $I$ by the irrelevant ideal.  Because no associated geometric support remains on the complement of this torus open, the saturated closed subscheme is obtained scheme-theoretically by contraction from the torus localization:
\[
        J=S\cap I\,S[(X_0\cdots X_{n-1})^{-1}].
\]
This saturation-contraction step preserves possible nilpotent structure; in particular, no reduction is being taken when $p\mid d$.

On the affine chart $X_0=1$ inside the torus, put $u_i=X_i/X_0$.  The localized equations are exactly
\[
        u_i^d-1=0\qquad(1\le i\le n-1),
\]
and the trace-zero hyperplane equation becomes
\[
        1+u_1+\cdots+u_{n-1}=0.
\]
Therefore, as a closed subscheme after saturation by the irrelevant ideal, the base locus is
\[
        \operatorname{Spec} k[u_1^{\pm1},\ldots,u_{n-1}^{\pm1}]/
        (u_1^d-1,\ldots,u_{n-1}^d-1,1+u_1+\cdots+u_{n-1}),
\]
which is precisely the stated torsion scheme.  This argument also covers the nonreduced case when $p\mid d$.
\end{proof}

\begin{theorem}[all-dimensional quotient criterion]\label{thm:all-n-quotient-criterion}
Assume $n\ge2$.  The map $P_d(x)=\sigma(x)^d-x^d$ permutes $\Gamma_{n,q}$ if and only if the following three conditions hold:
\begin{enumerate}[label=\textup{(\roman*)}]
\item $\gcd(d,q-1)=1$;
\item the underlying set of geometric points of the base locus $B_{n,d}$ has no point fixed by $\Theta_{n,q}$;
\item under condition \textup{(ii)}, the rational map $\Psi_{n,d}$ is defined on $U_n(\barF_q)^{\Theta_{n,q}}$ and induces a bijection on
\[
        U_n(\barF_q)^{\Theta_{n,q}}.
\]
\end{enumerate}
For $n=3$, after the affine coordinate $z=X_1/X_0$, this criterion is the specialization of \cref{prop:mult-desc} to $r=d$ and $B(z)=z^d-1$, namely to the quotient formula $H_d^{\rm raw}$ and its reduced map $h_d$.  Only the underlying geometric points of $B_{n,d}$ enter condition \textup{(ii)}; nilpotent structure is irrelevant there because the criterion concerns values of the induced function on a finite twisted fixed set.
\end{theorem}

\begin{proof}
The zero element of $\Gamma_{n,q}$ is fixed.  On $\Gamma_{n,q}^*$, scalar multiplication by $c\in K^*$ changes $P_d(x)$ by the factor $c^d$.  Hence the map on each $K^*$-fiber is bijective precisely when $\gcd(d,q-1)=1$.

By \cref{prop:h90-all-n}, the quotient of $\Gamma_{n,q}^*$ by $K^*$ is the twisted fixed set $U_n^{\Theta_{n,q}}$.  Formula \eqref{eq:Psi-n-d} is obtained by applying $\sigma^i$ to $P_d(x)$:
\[
        \sigma^i(P_d(x))=\sigma^{i+1}(x)^d-\sigma^i(x)^d.
\]
Thus $\Psi_{n,d}$ is the quotient map wherever the image is nonzero.  The image is zero exactly when all displayed differences vanish, i.e. exactly at the underlying geometric point set of the base locus of \cref{prop:higher-torsion-base}.  Moreover, for a twisted fixed point coming from $x\in\Gamma_{n,q}^*$, the target coordinates are the conjugates $\sigma^i(P_d(x))$.  If any one of these coordinates is zero, then $P_d(x)=0$, so all of them are zero and the point is in the base locus.  Hence every non-base image of a twisted fixed point again lies in $U_n$.  Therefore condition (ii) is the higher-dimensional denominator condition, and condition (iii) is the quotient bijectivity condition.  Only the underlying geometric points of $B_{n,d}$ matter here: the criterion evaluates the rational map on actual twisted fixed points, so nilpotent structure in a nonreduced base locus cannot create additional inputs.  Combining the fiber and quotient conditions proves the criterion.

When $n=3$, the hyperplane $X_0+X_1+X_2=0$ with $z=X_1/X_0$ gives $X_2/X_0=-1-z$.  The quotient map \eqref{eq:Psi-n-d} becomes the same rational map as \eqref{eq:Hd}, and the base-locus condition becomes $\mu_d\cap\Lambda_q=\varnothing$.  Thus the result recovers precisely the fixed-exponent specialization $r=d$, $B(z)=z^d-1$ of \cref{prop:mult-desc}, not the full generality of that proposition for arbitrary $r$ and $B$.
\end{proof}

For the higher-dimensional collision complexes below, fix a prime $\ell\ne p$.

\begin{definition}[higher-dimensional collision complex]\label{def:higher-collision-complex}
Let $W\subset U_n$ be a nonempty open subset defined over $\F_p$ and stable under $\rho$.  Since $U_n$ is a smooth equidimensional variety of dimension $n-2$, the open subset $W$ is also smooth and equidimensional of dimension $n-2$.  Let
\[
        f:V\to W
\]
be a $\rho$-equivariant finite etale restriction of $\Psi_{n,d}$ over $W$, and assume that $V$ is also defined over $\F_p$ and that $f$ is separable and generically finite of degree $N$.  These hypotheses imply that $W$ and $V$ are preserved by
\[
        \Theta_{n,q}=\rho^{-1}\Frob_q
\]
for every $q=p^k$.  Define the off-diagonal collision variety
\[
        C_f=(V\times_WV)\setminus\Delta.
\]
Then $V$ is smooth and equidimensional of dimension $n-2$, and $C_f$, when nonempty, is smooth and equidimensional of dimension $n-2$; moreover $C_f$ is preserved by $\Theta_{n,q}$.  The higher-dimensional twisted collision complex is
\[
        R\Gamma_c(C_{f,\barF_p},\Ql),
\]
with the induced action of $\Theta_{n,q}$.
\end{definition}

\begin{definition}[higher-dimensional twisted exceptionality]\label{def:higher-twisted-exceptional}
Let $f:V\to W$ be a $\rho$-equivariant morphism of $\rho$-stable varieties over $\F_p$.  We say that $f$ is $\Theta_n$-twisted exceptional in dimension $n$ if, for infinitely many $k\ge1$, it induces a bijection
\[
        V(\barF_p)^{\Theta_{n,p^k}}
        \longrightarrow
        W(\barF_p)^{\Theta_{n,p^k}}.
\]
For a rational quotient map $\Psi_{n,d}$, this terminology refers to a chosen finite-etale restriction $f:V\to W$ as in \cref{def:higher-collision-complex}; an off-diagonal collision on $V(\barF_p)^{\Theta_{n,p^k}}$ rules out bijectivity for that same $k$ for any extension of the same quotient map whose domain contains the two colliding points and on which the map is defined at them.
\end{definition}

\begin{proposition}[higher-dimensional collision trace]\label{prop:higher-collision-trace}
With the notation of \cref{def:higher-collision-complex}, let $r=n-2$ and let $N_f(q)$ be the number of ordered off-diagonal collisions on $V(\barF_p)^{\Theta_{n,q}}$.  Then
\[
        N_f(q)=\sum_i(-1)^i
        \operatorname{Tr}\bigl(\Theta_{n,q}\mid H_c^i(C_{f,\barF_p},\Ql)\bigr).
\]
Define $a_f(q)$ as follows.  If $r\ge1$, let $a_f(q)$ be the trace of $\Theta_{n,q}$ on the permutation representation generated by the $r$-dimensional geometric irreducible components of $C_f$; equivalently, it is the number of such components fixed by $\Theta_{n,q}$.  If $r=0$, let $a_f(q)$ be the number of $\Theta_{n,q}$-fixed geometric points of $C_f$.  Then for $r\ge1$
\[
        N_f(q)=a_f(q)q^r+O_f(q^{r-1/2}).
\]
For $r=0$, the formula is the exact zero-dimensional count $N_f(q)=a_f(q)$.  In particular, if $r\ge1$ and $a_f(p^k)>0$ for all sufficiently large $k$, then $f$ is not $\Theta_n$-twisted exceptional in dimension $n$ in the sense of \cref{def:higher-twisted-exceptional}.
\end{proposition}

\begin{proof}
The fixed points of $\Theta_{n,q}$ on $C_f$ are exactly ordered pairs of distinct twisted fixed points with the same image.  The first formula is Grothendieck--Lefschetz.  If $r=0$, then $C_f$ is finite etale over the chosen locus and the trace formula is simply the exact count of fixed geometric points, which is the stated definition of $a_f(q)$.  If $r\ge1$, then \cref{def:higher-collision-complex} makes $C_f$ smooth and equidimensional of dimension $r$, and the top compactly supported cohomology contributes one copy of $\Ql(-r)$ for each $r$-dimensional irreducible component fixed by $\Theta_{n,q}$, giving the main term $a_f(q)q^r$.

It remains to justify the error term for the twisted operator.  The cyclic shift $\rho$ has order $n$ and commutes with Frobenius, so $\Theta_{n,q}^n=\Frob_{q^n}$ on $C_f$ and on compactly supported cohomology.  If $\alpha$ is an eigenvalue of $\Theta_{n,q}$ on $H_c^i(C_{f,\barF_p},\Ql)$, then $\alpha^n$ is an eigenvalue of $\Frob_{q^n}$ on the same space.  Deligne's weight bounds \cite{DeligneWeilII} give $|\alpha^n|\le q^{ni/2}$, hence $|\alpha|\le q^{i/2}$.  All terms with $i<2r$ therefore contribute $O_f(q^{r-1/2})$.  The final assertion for $r\ge1$ follows exactly as in \cref{thm:collision-trace}.
\end{proof}

\section{Additive full-field lifts}\label{sec:fullfield}

This appendix is supplemental to the monodromy results in the main text.  Throughout this section,
\[
        q=p^k,
        \qquad
        L=\F_{q^3},
        \qquad
        K=\F_q,
        \qquad
        \Gamma_q=\ker\Tr_{L/K}.
\]
The trace-zero permutations above can be lifted to permutations of $\F_{q^3}$ by the additive Hilbert--90 fiber.

\begin{lemma}[additive fiber criterion]\label{lem:additive-fiber}
Let $\delta(x)=x^q-x$, so $\delta:L\to\Gamma_q$ is surjective with fibers $u+K$.  Suppose $F:L\to L$ and $G:\Gamma_q\to\Gamma_q$ satisfy
\[
        \delta\circ F=G\circ\delta.
\]
Then $F$ permutes $L$ if and only if $G$ permutes $\Gamma_q$ and $F$ is injective on every fiber $u+K$ of $\delta$.
\end{lemma}

\begin{proof}
First note the additive Hilbert--90 exactness used in the statement.  For every $x\in L$,
\[
        \Tr_{L/K}(x^q-x)=0,
\]
so $\delta(L)\subseteq\Gamma_q$.  Also $\ker(\delta)=K$, hence
\[
        |\delta(L)|=|L|/|K|=q^2.
\]
Since $L/K$ is finite and separable, the trace map $\Tr_{L/K}:L\to K$ is a nonzero $K$-linear functional, hence is surjective; therefore $|\Gamma_q|=q^2$.  It follows that $\delta(L)=\Gamma_q$, and the fibers are exactly the cosets $u+K$.

Now apply the usual fiber criterion.  If $F$ is bijective, then it is injective on fibers and the commutative diagram forces $G$ to be bijective on the quotient.  Conversely, if $G$ is bijective and $F$ is injective on every fiber, then two points with the same image under $F$ must have the same image under $G\circ\delta$, hence lie in the same fiber of $\delta$, where injectivity applies.  Since $L$ is finite, injectivity implies bijectivity.
\end{proof}

\begin{corollary}[full-field Mersenne lifts]\label{cor:fullfield}
Let $q=2^k$ with $k\ge1$, let $a\ge1$, put $d=2^a-1$, and let $m\ge0$.  For $c\in\F_{q^3}$ define
\[
        \mathcal F_{a,m,c}(X)
        =(X^q-X)^d
        +cX^{2^m}+c^qX^{q\cdot 2^m}+c^{q^2}X^{q^2\cdot 2^m}.
\]
Then $\mathcal F_{a,m,c}$ permutes $\F_{q^3}$ if and only if
\[
        \gcd(a,k)=1,
        \qquad 3\nmid a,
        \qquad \Tr_{\F_{q^3}/\F_q}(c)\ne0.
\]
\end{corollary}

\begin{proof}
Let $\delta=X^q-X$.  The first term factors through $\delta$, and
\[
        \delta((\delta X)^d)=(\delta X)^{dq}-(\delta X)^d=P_d(\delta X).
\]
Thus it induces the trace-zero map of \cref{cor:mersenne}.  The remaining trace-linear term lies in $K$, so it is killed by $\delta$.  On a fiber $u+t$ with $t\in K$, the trace-linear term changes by
\[
        \Tr_{L/K}(c(u+t)^{2^m})-\Tr_{L/K}(cu^{2^m})
        =\Tr_{L/K}(c)t^{2^m}.
\]
Since $t\mapsto t^{2^m}$ is a bijection of $K$, the fiber map is injective if and only if $\Tr_{L/K}(c)\ne0$.  Apply \cref{lem:additive-fiber}.
\end{proof}

\begin{theorem}[stable sparsity obstruction for Mersenne lifts]\label{thm:mersenne-sparsity-obstruction}
Let $q=2^k$ with $k\ge1$, let $a\ge1$, put $d=2^a-1$, and assume
\[
        1<d<q-1.
\]
Let $\delta(X)=X^q-X=X^q+X$, and set $M=q^3-1$.  For a positive exponent $e$, write
\[
        \langle e\rangle_M=1+((e-1)\bmod M)\in\{1,\ldots,M\}.
\]
Throughout this theorem, reducing exponents in the function algebra of $\F_{q^3}$ means replacing each nonconstant monomial $X^e$ by $X^{\langle e\rangle_M}$ and then combining like monomials; in particular the class $0\bmod M$ is represented by $X^M$, not by the constant function $1$.  With polynomial representatives taken over $\F_{q^3}$ in this canonical sense, any polynomial $F$ satisfying
\[
        \delta(F)=P_d(\delta X)
\]
has at least $d$ nonconstant monomial terms.  This lower bound is sharp: one may take
\[
        F_0(X)=X^{q^2d}+\sum_{j=1}^{d-1}X^{q(d-j)+j}.
\]
Consequently, for any sequence of pairs $(a,k)$ with $1<2^a-1<2^k-1$ and $a\to\infty$, every such additive primitive has a number of nonconstant quotient monomial terms tending to infinity.  In the stable range there is therefore no bound independent of $a$ analogous to the bounded-term Ding--Song--Xiong trinomials or pentanomials.
\end{theorem}

\begin{proof}
Write
\[
        (X^q+X)^d=\sum_{j=0}^d X^{e_j},
        \qquad
        e_j=q(d-j)+j.
\]
All binomial coefficients are $1$ in characteristic two because $d=2^a-1$.  Under the stable hypothesis all exponents appearing below are positive and not congruent to $0$ modulo $M$, so the canonical exponent convention agrees with ordinary residue computations in $\mathbb Z/M\mathbb Z$.  For a coefficient $\alpha\in\F_{q^3}$, the operator $\delta$ sends a monomial $\alpha X^e$ to $\alpha^qX^{qe}+\alpha X^e$.  Thus coefficients may change by Frobenius, but exponents remain inside the orbit of $e$ under multiplication by $q$ on $\mathbb Z/M\mathbb Z$.

Because $d<q-1$, the exponents $e_j$ are distinct, nonzero, and less than $q^2+q+1$; hence none is fixed by multiplication by $q$ modulo $M=q^3-1$.  We next check the orbit collisions.  If $e_i\equiv qe_j\pmod M$, then in fact $qe_j<M$, since $e_j\le qd<q(q-1)$; hence equality holds in the integers.  The equation $e_i=qe_j$ is
\[
        qd-(q-1)i=q(qd-(q-1)j),
\]
which reduces to $i=q(j-d)$.  Thus the only solution with $0\le i,j\le d$ is $(i,j)=(0,d)$, giving
\[
        qe_d=e_0=qd.
\]
The case $e_i\equiv q^2e_j\pmod M$ is equivalent, after multiplying by $q$, to $qe_i\equiv e_j\pmod M$, and the same argument gives the same collision.  Therefore the $d+1$ exponents $e_0,\ldots,e_d$ occupy exactly $d$ distinct $q$-orbits.

Thus the support of
\[
        P_d(\delta X)=(X^q+X)^{dq}+(X^q+X)^d
\]
meets exactly $d$ distinct $q$-orbits: one orbit contributes the two exponents $d$ and $q^2d$, and for each $1\le j\le d-1$ one orbit contributes the two exponents $e_j$ and $qe_j$.

The target has no contribution on a $q$-fixed exponent class, and the support analysis above placed every nonzero target contribution on a length-three orbit.  A monomial $aX^e$ contributes only to the $q$-orbit of the exponent class of $e$, namely as $a^qX^{qe}+aX^e$ in characteristic two, with the same canonical reduction of positive exponents.  Thus no monomial from one orbit can help produce a nonzero target contribution on another orbit.  Consequently every preimage of $P_d(\delta X)$ must contain at least one monomial from each of these $d$ nonzero orbits, independently of the coefficient field.  This proves the lower bound.

Finally,
\[
        \delta\bigl(X^{q^2d}\bigr)=X^d+X^{q^2d}
\]
as functions on $\F_{q^3}$, and for $1\le j\le d-1$,
\[
        \delta\bigl(X^{e_j}\bigr)=X^{e_j}+X^{qe_j}.
\]
Adding these identities gives $\delta(F_0)=P_d(\delta X)$, so the bound is sharp.
\end{proof}

\begin{example}[stable support count for $q=8$ and $d=3$]\label{ex:stable-mersenne-q8-d3}
Take $q=8$, $a=2$, and $d=2^a-1=3$.  Then $M=q^3-1=511$, and \cref{thm:mersenne-sparsity-obstruction} predicts that any additive primitive of $P_3(X^8-X)$ has at least three nonconstant quotient monomials.  The sharp primitive supplied by the theorem is
\[
        F_0(X)=X^{192}+X^{17}+X^{10}.
\]
Indeed, in characteristic two,
\[
\begin{aligned}
        (X^8+X)^3 &= X^{24}+X^{17}+X^{10}+X^3,\\
        (X^8+X)^{24} &= X^{192}+X^{136}+X^{80}+X^{24},
\end{aligned}
\]
so the $X^{24}$ terms cancel and
\[
        P_3(X^8-X)=X^{192}+X^{136}+X^{80}+X^{17}+X^{10}+X^3.
\]
On the other hand,
\[
        \delta(F_0)=F_0^8-F_0=F_0^8+F_0
\]
in characteristic two, and reducing positive exponents modulo $511$ gives exactly the same six monomials.  The three terms of $F_0$ represent the three distinct $q$-orbits in the proof of \cref{thm:mersenne-sparsity-obstruction}.
\end{example}

\begin{theorem}[unstable Mersenne exponent-collision classification]\label{thm:unstable-mersenne-classification}
Let $q=2^k$ with $k\ge1$, let $a\ge1$, let $M=q^3-1$, and put $d_a=2^a-1$.  Let $b$ be the least nonnegative residue of $a$ modulo $3k$.

As polynomial functions on $\mathbb F_{q^3}$,
\[
        P_{d_a}=X^{d_aq}-X^{d_a}
        =
        \begin{cases}
        P_{2^b-1},& b>0,\\[2pt]
        0,& b=0.
        \end{cases}
\]
Thus every unstable Mersenne coincidence is a residue collapse modulo $q^3-1$; it is not a new exponent family on $\mathbb F_{q^3}$.

More precisely, let $S=q^2+q+1$, let $R$ be the least residue of $2^b$ modulo $S$, and let
\[
        I_R=\{0,1,\ldots,R-1\}\subset \mathbb Z/S\mathbb Z.
\]
For $b>0$ define the affine order-three map
\[
        T_{b,k}(j)=qj-q(2^b-1)\pmod S.
\]
Then, among polynomial representatives with coefficients in $\mathbb F_{q^3}$ and with positive exponents reduced by the same canonical convention $X^e=X^{\langle e\rangle_M}$, where $\langle e\rangle_M=1+((e-1)\bmod M)$, the minimum possible number of nonconstant monomial terms in a polynomial $F$ satisfying
\[
        F^q-F=P_{d_a}(X^q-X)
\]
as a function on $\mathbb F_{q^3}$ is
\[
        \nu_{b,k}=\frac12\bigl|I_R\triangle T_{b,k}(I_R)\bigr|.
\]
For $b=0$ this minimum is $0$.  For $0<b\le k$ one has $\nu_{b,k}=2^b-1$, recovering the stable lower bound when $b<k$.
\end{theorem}

\begin{proof}
The functional reduction is immediate from $M=q^3-1$ together with the canonical positive-exponent convention.  If $b>0$, then positive exponents congruent modulo $M$ define the same monomial function on $\mathbb F_{q^3}$, so $X^{d_a}=X^{2^b-1}$ and likewise for the $q$-multiple exponent.  If $b=0$, then $d_a$ is a positive multiple of $M$, so both $X^{d_aq}$ and $X^{d_a}$ are represented by the same nonconstant monomial function $X^M$; they also agree at $0$.  Hence $P_{d_a}=0$.

It remains to count the minimum number of terms in a primitive for the additive Hilbert--90 operator $\delta(F)=F^q-F=F^q+F$.  Write
\[
        (X^q+X)^{2^b-1}=\sum_{j=0}^{2^b-1}X^{e_j},
        \qquad
        e_j=q(2^b-1)-(q-1)j.
\]
The nonzero exponent class of $e_j$ modulo $M=(q-1)S$ is determined by the residue of $j$ modulo $S$.  Multiplication of exponents by $q$ sends the residue $j$ to
\[
        T_{b,k}(j)=qj-q(2^b-1)\pmod S,
\]
and $T_{b,k}^3=1$ because $q^3\equiv1\pmod S$ and $1+q+q^2=S$.

Only the parity of the number of occurrences of each residue matters, since the characteristic is two.  The interval $0\le j\le 2^b-1$ has length $2^b$.  Modulo $S$, its parity support differs from $I_R$ at most by the full set $\mathbb Z/S\mathbb Z$; the full set is fixed by $T_{b,k}$ and therefore cancels in the symmetric difference.  Hence the support of
\[
        P_{2^b-1}(X^q-X)=\delta\bigl((X^q+X)^{2^b-1}\bigr)
\]
on the set of $q$-orbits is represented exactly by
\[
        I_R\triangle T_{b,k}(I_R).
\]
On a $q$-orbit of exponent classes of length three, each monomial term contributes to two adjacent classes in that orbit, with Frobenius-conjugate coefficients.  In the present target all coefficients are in $\mathbb F_2$.  For an orbit written as $\{e,qe,q^2e\}$, the three possible two-position supports are produced by
\[
        \delta(X^e)=X^e+X^{qe},\qquad
        \delta(X^{qe})=X^{qe}+X^{q^2e},\qquad
        \delta(X^{q^2e})=X^{q^2e}+X^e.
\]
Hence every nonzero two-term target vector on a length-three orbit is produced by one monomial term, and no monomial from another orbit can contribute to it.  A length-three orbit on which the target vector is zero requires no term in a term-minimal primitive: after combining like canonical monomials, any terms supported there contribute an element of $\ker\delta$ on that orbit and may be deleted without changing the target.

The same omission argument is needed on $q$-fixed exponent classes, including the class $0\bmod M$ represented by $X^M$.  If such a class has combined coefficient $\alpha$ in a primitive, then its contribution is $(\alpha^q+\alpha)X^e$.  The target coefficient on every $q$-fixed class is zero, so necessarily $\alpha^q+\alpha=0$; the term is therefore in $\ker\delta$ and can be omitted from a term-minimal primitive.  Therefore the minimum number of monomial terms in a primitive is one for each nonzero length-three orbit, namely
\[
        \frac12|I_R\triangle T_{b,k}(I_R)|.
\]
If $0<b\le k$, then $R=2^b\le q$.  The two intervals $I_R$ and $T_{b,k}(I_R)$ meet in the single residue $0$, so the displayed formula gives $\nu_{b,k}=R-1=2^b-1$.
\end{proof}

\begin{corollary}[bounded unstable Mersenne lifts are residue collapses]\label{cor:no-new-unstable-mersenne}
Fix $B\ge0$.  For every $q=2^k$ with $k\ge1$ and every $a\ge k$, the minimum number of quotient monomial terms in an additive primitive for $P_{2^a-1}(X^q-X)$ is the explicitly computable number $\nu_{b,k}$ of \cref{thm:unstable-mersenne-classification}, where $b\equiv a\pmod{3k}$.  Hence, for each fixed $k$, every bounded-term unstable example is explained by one of the finitely many residues $b\pmod{3k}$ with $\nu_{b,k}\le B$; on $\mathbb F_{q^3}$ it is literally the lower residue exponent $2^b-1$, or the zero map when $b=0$.

Consequently, for each fixed $k$, the unstable range supplies isolated or residue-collapse sparse representatives among the finitely many residue classes $b\pmod{3k}$, but no new stable Mersenne family whose exponent grows independently of the reduction modulo $q^3-1$.
\end{corollary}

\begin{proof}
This is a direct restatement of the exact orbit count.  The function $P_{2^a-1}$ depends only on the residue of $a$ modulo $3k$, and the least number of quotient monomials in an additive primitive is exactly $\nu_{b,k}$.  Thus any sparse unstable representative is sparse because its exponent has collapsed to a sparse residue class in the function algebra of $\mathbb F_{q^3}$.
\end{proof}

\begin{remark}
The lifts in \cref{cor:fullfield} are standard additive fiber-separating lifts.  \Cref{thm:mersenne-sparsity-obstruction} and \cref{thm:unstable-mersenne-classification} together give the precise obstruction: in the stable range the quotient primitive needs $2^a-1$ terms, while in the unstable range every sparse example is explained by reduction of the exponent modulo $q^3-1$.
\end{remark}

\section{Computational certificates}\label{app:computations}

The following SageMath certificate verifies the finite-field calculations used in the sporadic and bad-reduction statements; see \cite{SageMath} for SageMath.  It uses only polynomial arithmetic over finite fields and matrix multiplication in $\operatorname{PGL}_2$.  The code is self-contained and can be run either in SageCell with language set to \emph{Sage} or locally as a \texttt{.sage} file.  A successful execution prints a start line, the Sage version, one confirmation line for each block, and terminates with a line reading \texttt{All computational certificates passed.}  If any identity fails, Sage raises an \texttt{AssertionError}.  The computer-assisted claims are deterministic finite-field identity checks, conditional only on successful execution of this displayed certificate.  The corresponding theorem and corollary are stated with this condition explicitly.  The certificate verifies the degree-two characteristic-$11$ calculation, \cref{thm:sporadic-V4}, and \cref{cor:bad-reduction-7-5}.  No random search, external files, or version-specific packages are used.

The expected successful output has the following form, with the Sage version depending on the installation:
\begin{verbatim}
Starting computational certificates...
Sage version: <SageMath version>
Characteristic 11 degree-two quotient: passed
Characteristic 19 Klein-four quotient: passed
Characteristic 7 bad Morse reduction: passed
All computational certificates passed.
\end{verbatim}

The certificate code is:
\begin{verbatim}
print("Starting computational certificates...")
import sage.version
print("Sage version:", sage.version.version)

# Degree-two sporadic quotient in characteristic 11.
F = GF(11); R.<z> = PolynomialRing(F)
N = -(z+1)^5 - z^5
D = z^5 - 1
g = gcd(N,D)
assert g == z^3 - 4*z^2 + 4*z - 1
assert N//g == -2*z^2 - 2*z + 1
assert D//g == z^2 + 4*z + 1
assert gcd(N//g,D//g) == 1
assert max((N//g).degree(), (D//g).degree()) == 2
print("Characteristic 11 degree-two quotient: passed")

# Klein-four quotient in characteristic 19.
F = GF(19); R.<x> = PolynomialRing(F)
K_num = 1 - x^6
K_den = x^6 - (1-x)^6
g = gcd(K_num,K_den)
hn = K_num//g; hd = K_den//g
assert hn == -x^4 - x^3 + x + 1
assert hd == 6*x^3 - 9*x^2 + 5*x - 1
assert gcd(hn,hd) == 1
assert max(hn.degree(), hd.degree()) == 4

# K_6 has source coordinate x and target coordinate z. The actual
# x-coordinate self-map is htilde = K_6/(K_6+1).
htn = hn
htd = hn + hd
assert htn == -x^4 - x^3 + x + 1
assert htd == -x^4 + 5*x^3 - 9*x^2 + 6*x
assert gcd(htn,htd) == 1
assert max(htn.degree(), htd.degree()) == 4

def mobius(M,t):
    a,b,c,d = M[0,0],M[0,1],M[1,0],M[1,1]
    return (a*t+b)/(c*t+d)

def K(t):
    return hn(t)/hd(t)

def htilde(t):
    return htn(t)/htd(t)

gammas = [matrix(F,[[1,0],[0,1]]),
          matrix(F,[[1,6],[2,-1]]),
          matrix(F,[[1,10],[9,-1]]),
          matrix(F,[[1,-2],[13,-1]])]

for M in gammas:
    assert M.det() != 0
    assert (K(mobius(M,x)) - K(x)).numerator() == 0
    assert (htilde(mobius(M,x)) - htilde(x)).numerator() == 0

def peq(A,B):
    avec = [A[0,0],A[0,1],A[1,0],A[1,1]]
    bvec = [B[0,0],B[0,1],B[1,0],B[1,1]]
    return all(avec[i]*bvec[j] == avec[j]*bvec[i]
               for i in range(4) for j in range(4))

# The four projective classes are distinct, the three nonidentity
# elements are involutions, and the set is closed under multiplication.
for i in range(4):
    for j in range(i+1,4):
        assert not peq(gammas[i], gammas[j])
for i in range(1,4):
    assert peq(gammas[i]*gammas[i], gammas[0])
for M in gammas:
    for Nmat in gammas:
        assert any(peq(M*Nmat,G) for G in gammas)
        assert peq(M*Nmat, Nmat*M)

tau = matrix(F,[[0,1],[-1,1]])       # x |-> 1/(1-x)
tau_inv = tau^(-1)
assert peq(tau_inv*gammas[1]*tau, gammas[2])
assert peq(tau_inv*gammas[2]*tau, gammas[3])
assert peq(tau_inv*gammas[3]*tau, gammas[1])
assert (htilde(mobius(tau,x)) - mobius(tau, htilde(x))).numerator() == 0
assert (K(mobius(tau,x)) - mobius(tau, K(x))).numerator() != 0

# Branch fixed-point exclusions for the V4 quotient.
fixed_quads = [x^2-x-3, x^2+4*x+1, x^2-6*x+6]

def fixed_poly(M):
    a,b,c,d = M[0,0],M[0,1],M[1,0],M[1,1]
    return R(c*x^2 + (d-a)*x - b)

for M,f in zip(gammas[1:], fixed_quads):
    assert M[1,0] != 0              # no nonidentity deck involution fixes infinity
    assert fixed_poly(M).monic() == f.monic()

assert all(f.is_irreducible() for f in fixed_quads)
assert gcd(x^2-x+1, prod(fixed_quads)) == 1

for S,f in zip([1,-4,6], fixed_quads):
    assert gcd(f, (S-x)*(1-x)-1) == 1

print("Characteristic 19 Klein-four quotient: passed")

# Bad Morse reduction: characteristic 7, d=5.
F = GF(7); R.<z,T> = PolynomialRing(F,2)
N = -(z+1)^5 - z^5
D = z^5 - 1
g = gcd(R(N),R(D))
assert g == 1
n = R(N)//g; dden = R(D)//g
assert max(n.degree(z), dden.degree(z)) == 5
W = n.derivative(z)*dden - n*dden.derivative(z)
expected_W = (-2*(z-2)*(z+3)*(z^2+1)
              *(z^2+z-3)*(z^2+2*z+2))
assert factor(W) == factor(expected_W)
B = (n - T*dden).resultant(W,z)
expected_B = (2*(T-3)^2*(T-1)^2*(T+2)^2
              *(T-2)*(T+3))
assert factor(B) == factor(expected_B)

print("Characteristic 7 bad Morse reduction: passed")
print("All computational certificates passed.")
\end{verbatim}

\end{document}